\documentclass[12pt]{amsart}

\usepackage{amssymb, amsmath, amsthm}
\usepackage{calligra, mathrsfs}
\usepackage{graphicx}
\usepackage{url}
\usepackage{mathtools}
\usepackage{enumerate}
\usepackage{verbatim}
\usepackage[retainorgcmds]{IEEEtrantools}
\usepackage{tikz-cd}
\usetikzlibrary{positioning}
\usepackage{microtype}

\usepackage[margin=1in,marginparwidth=0.8in, marginparsep=0.1in]{geometry}

\usepackage{xr-hyper}
\usepackage[linktocpage=true, unicode]{hyperref}
\usepackage{xcolor}
\hypersetup{
    colorlinks,
    linkcolor={red!65!black},
    citecolor={blue!78!black}
}

\newcommand{\Acal}{{\mathcal A}}

\newcommand{\Ccal}{{\mathcal C}}
\newcommand{\Dcal}{{\mathcal D}}
\newcommand{\Ecal}{{\mathcal E}}

\newcommand{\Ical}{{\mathcal I}}
\newcommand{\Jcal}{{\mathcal J}}

\newcommand{\Mcal}{{\mathcal M}}

\newcommand{\Pcal}{{\mathcal P}}

\newcommand{\Ucal}{{\mathcal U}}

\newcommand{\acal}{\Acal}
\newcommand{\ccal}{\Ccal}
\newcommand{\dcal}{\Dcal}
\newcommand{\mcal}{\Mcal}

\newcommand{\GG}{{\mathbb G}}

\newcommand{\ZZ}{{\mathbb Z}}
\renewcommand{\SS}{\mathbb{S}}
\newcommand{\NN}{{\mathbb N}}

\usepackage{scalerel,stackengine}
\stackMath
\newcommand\reallywidehat[1]{%
\savestack{\tmpbox}{\stretchto{%
  \scaleto{%
    \scalerel*[\widthof{\ensuremath{#1}}]{\kern-.6pt\bigwedge\kern-.6pt}%
    {\rule[-\textheight/2]{1ex}{\textheight}}
  }{\textheight}%
}{0.5ex}}%
\stackon[1pt]{#1}{\tmpbox}%
}

\makeatletter
\newcommand*\bigcdot{{\mathpalette\bigcdot@{.5}}}
\newcommand*\bigcdot@[2]{\mathbin{\vcenter{\hbox{\scalebox{#2}{$\m@th#1\bullet$}}}}}
\makeatother

\makeatletter
\g@addto@macro\bfseries{\boldmath}
\makeatother

\tikzset{shorten <>/.style={shorten >=#1,shorten <=#1}}

\newcommand{\kr}{\kern -2pt}

\DeclareMathOperator{\id}{id}
\DeclareMathOperator{\Hom}{Hom}

\DeclareMathOperator{\End}{End}
\DeclareMathOperator{\Funct}{Funct}

\DeclareMathOperator{\Ho}{Ho}

\DeclareMathOperator{\der}{D}

\DeclareMathOperator{\Ker}{Ker}

\DeclareMathOperator{\coker}{Coker}

\newcommand{\enh}{{\text{\normalfont enh}}}
\newcommand{\op}{{\text{\normalfont op}}}

\DeclareMathOperator{\Spc}{Spc}

\DeclareMathOperator{\cat}{Cat}
\DeclareMathOperator{\Cat}{Cat}
\newcommand{\catomega}{\smash{{\vstretch{.8}{\widehat{\vstretch{1.25}{\Cat}}}{}^\omega}}}
\newcommand{\cathat}{\smash{{\vstretch{.8}{\widehat{\vstretch{1.25}{\Cat}}}}}}
\newcommand{\catl}{\smash{{\vstretch{.8}{\widehat{\vstretch{1.25}{\Cat}}}{}^L}}}

\DeclareMathOperator{\colim}{colim}
\newcommand{\lex}{{\text{\normalfont lex}}}
\newcommand{\rex}{{\text{\normalfont rex}}}

\DeclareMathOperator{\Ind}{Ind}
\let\Pr\relax
\DeclareMathOperator{\Pr}{Pr}

\DeclareMathOperator{\Sp}{Sp}
\DeclareMathOperator{\Ab}{Ab}
\newcommand{\cn}{{\normalfont\text{cn}}}
\DeclareMathOperator{\Ext}{Ext}
\DeclareMathOperator{\Tor}{Tor}

\DeclareMathOperator{\Groth}{Groth}
\DeclareMathOperator{\groth}{Groth}
\newcommand{\cp}{{\normalfont \text{cp}}}

\DeclareMathOperator{\lincat}{LinCat}
\newcommand{\catlr}{\lincat_R}
\newcommand{\catlk}{\lincat_k}

\DeclareMathOperator{\CAlg}{CAlg}
\DeclareMathOperator{\Mod}{Mod}
\DeclareMathOperator{\LMod}{LMod}
\DeclareMathOperator{\RMod}{RMod}
\DeclareMathOperator{\BMod}{BMod}

\DeclareMathOperator{\Spec}{Spec}

\newtheorem{proposition}[subsubsection]{Proposition}
\newtheorem{lemma}[subsubsection]{Lemma}
\newtheorem{theorem}[subsubsection]{Theorem}
\newtheorem{corollary}[subsubsection]{Corollary}

\newtheoremstyle{note}{8.0pt plus 2.0pt minus 4.0pt}{8.0pt plus 2.0pt minus 4.0pt}{}{}{\bfseries}{.}{.5em}{} 
\theoremstyle{note}
\newtheorem{example}[subsubsection]{Example}
\newtheorem{remark}[subsubsection]{Remark}
\newtheorem{notation}[subsubsection]{Notation}
\newtheorem{construction}[subsubsection]{Construction}

\newtheorem{definition}[subsubsection]{Definition}

\AtBeginDocument{\def\MR#1{}}

\usepackage{standalone} 
\def\inmain{1}

\title{Classification of fully dualizable linear categories}

\author{G. Stefanich}

\date{}

\begin{document}


\begin{abstract}
We prove that if $R$ is a G-ring then every fully dualizable $R$-linear cocomplete category is equivalent to a twist by a $\GG_m$-gerbe of the category of modules over a finite \'etale $R$-algebra.  We also show that this holds more generally over an arbitrary commutative ring under an additional compact generation hypothesis. We  include variants of these results that apply to $R$-linear graded categories, and to the context of $\infty$-categories linear over connective commutative ring spectra.
\end{abstract}

\maketitle

\vspace{-0.25cm}

\tableofcontents

\newpage


\section{Introduction}

 Let $R$ be a commutative ring. Recall that an $R$-algebra $A$ is said to be Azumaya if there exists another $R$-algebra $B$ such that $A \otimes_R B$ is Morita equivalent to $R$. In other words, $A$ is Azumaya if it defines a tensor invertible object in the Morita $2$-category $\operatorname{Mor}_R$ of algebras and bimodules over $R$.

We may regard $\operatorname{Mor}_R$ as a full subcategory of the $2$-category $\catlr$ of  $R$-linear cocomplete categories and colimit preserving functors, via the embedding that maps each $R$-algebra to its category of left modules. In the same way that $\operatorname{Mor}_R$ has a symmetric monoidal structure induced by tensor product of algebras, there is a compatible symmetric monoidal structure on  $\catlr$, where for each pair of $R$-linear cocomplete categories $\Ccal, \Dcal$ the tensor product $\Ccal \otimes_R \Dcal$ is the universal recipient of a functor $\Ccal \times \Dcal \rightarrow \Ccal \otimes_R \Dcal$ which is colimit preserving and $R$-linear in each variable. From this point of view, an $R$-algebra $A$ is Azumaya if and only if its category of left modules is an invertible object of $\catlr$. 

To each Azumaya $R$-algebra $A$ one may attach a $\GG_m$-gerbe $\mathcal{G}(A)$ on $\Spec(R)$. In general, for any $\GG_m$-gerbe $\mathcal{G}$ 
 one may define an $R$-linear category $\Mod_{R,\mathcal{G}}^\heartsuit$ of $R$-modules twisted by $\mathcal{G}$, which in the case $\mathcal{G} = \mathcal{G}(A)$ recovers the category of left $A$-modules.  The category $\smash{\Mod_{R,\mathcal{G}}^\heartsuit}$ defines an invertible object of $\catlr$ for every gerbe $\mathcal{G}$. Our first main theorem states that if $R$ is a G-ring then every invertible category arises in this way:

\begin{theorem}\label{teo principal introduction}
Let $R$ be a G-ring. Then every invertible object of $\catlr$ is of the form $\Mod_{R,\mathcal{G}}^\heartsuit$ for some $\GG_m$-gerbe $\mathcal{G}$ on $\Spec(R)$. In particular, the group of equivalence classes of invertible objects of $\catlr$ is isomorphic to $H^2(\Spec(R), \GG_m)$.
\end{theorem}

We may regard theorem \ref{teo principal introduction} as providing a description of categorified line bundles in algebraic geometry. This  paper is concerned, more generally, with a categorification of the notion of vector bundle.  Classically, vector bundles on $\Spec(R)$ are the same as dualizable $R$-modules. In the categorical context, however, there are many dualizable objects of $\catlr$ which do not behave like vector bundles: for instance, the category of modules over any $R$-algebra is dualizable. 

The situation improves if instead of dualizable objects we study \emph{fully dualizable} objects: these are those dualizable objects of $\catlr$ for which the unit and counit of the duality admit a further colimit preserving right adjoint.  Our next result provides a classification of fully dualizable categories over G-rings:

\begin{theorem}\label{theorem principal 2 introduction}
 Let $R$ be a G-ring. Then every fully dualizable object of $\catlr$ is of the   form $\Mod^\heartsuit_{\tilde{R}, \mathcal{G}}$ for some finite \'etale $R$-algebra $\tilde{R}$ and $\GG_m$-gerbe $\mathcal{G}$ on $\Spec(\tilde{R})$.
\end{theorem}

Theorem  \ref{theorem principal 2 introduction} applies in particular in the case when $R$ is of finite type over $\ZZ$. As we shall see, a variant of the main result from \cite{ToenRings} shows that if $R$ is an arbitrary commutative ring, then every fully dualizable object $\ccal$ of $\catlr$ such that $\ccal$ and its dual are compactly generated is obtained by extension of scalars from a subring $S \subseteq R$ of finite type over $\ZZ$. It follows from this that our results remain valid over arbitrary commutative rings under an additional compact generation hypothesis.

Just like the basic example  of an invertible $R$-linear category is given by the category of modules over an Azumaya $R$-algebra, the basic example  of a fully dualizable category is given by the category of modules over an $R$-algebra which is both separable and dualizable as an $R$-module. Under the dictionary of theorem \ref{theorem principal 2 introduction}, these correspond to fully dualizable categories for which the gerbe $\mathcal{G}$ is of the form $\mathcal{G}(A)$ for some Azumaya $\tilde{R}$-algebra $A$. 

In general not every $\GG_m$-gerbe over an affine scheme is of this form.  In fact, it was shown by Gabber \cite{GabberAzumaya, deJong} that  a $\GG_m$-gerbe over an affine scheme arises from an Azumaya algebra if and only if its associated \'etale cohomology class is torsion.  Over a field every class is torsion, so theorem \ref{theorem principal 2 introduction} specializes to the following:

\begin{corollary}\label{corollary description fully dualizables}
 Let $k$ be a field. Then every fully dualizable\footnote{As we shall see, this corollary remains true if one only assumes smoothness instead of full dualizability.} object of $\catlk$ is the category of left modules over a separable $k$-algebra.
\end{corollary}

Fully dualizable objects of any $2$-category give rise, under the cobordism hypothesis, to two dimensional fully extended topological field theories \cite{LurieField}. Motivated by this connection, in \cite{BDSPV} the authors survey a number of notions of categorified vector space, and prove, building on results of Tillmann \cite{Tillmann}, that in all those cases the fully dualizable objects arise from separable $k$-algebras. Corollary \ref{corollary description fully dualizables} provides a strengthening of their classification, and recovers it after restriction to various subcategories of $\catlk$.\footnote{All the $2$-categories of categorified vector spaces considered in  \cite{BDSPV} embed inside the full subcategory of $\catlk$ on the $k$-linear Grothendieck abelian categories generated by compact projective objects. A lot of the work involved in proving our main results consists of showing that one may deduce the existence of compact projective generators (\'etale locally on $R$) from the condition of full dualizability.}

As remarked previously, while theorem \ref{theorem principal 2 introduction} shows that the class of fully dualizable categories is very constrained, there are many more categories which are only one time dualizable. In \cite{BCJF} the authors study the question of dualizability for various $k$-linear categories of interest, and conjecture that every dualizable $k$-linear category is generated by compact projective objects. This conjecture was verified in some cases in \cite{Chirvasitu}, however we will show it to be false in general (see examples \ref{example almost modules} and \ref{example non torsion}).

The theory of $R$-linear categories admits a globalization, which is given by the theory of quasicoherent sheaves of categories. Since general additive  categories do not satisfy Zariski descent, when globalizing one typically restricts attention to Grothendieck abelian categories, which were proven to satisfy fpqc descent  in appendix D of \cite{SAG}. In our context this is not a major restriction: as we shall see, for any commutative ring $R$ the dualizable objects of $\catlr$ are automatically Grothendieck abelian. Theorems \ref{teo principal introduction} and \ref{theorem principal 2 introduction} admit the following globalization:

\begin{corollary}\label{corollary invertible sheaves of categories}
Let $X$ be a stack that admits a cover by spectra of G-rings. Then every invertible (resp. fully dualizable) quasicoherent sheaf of Grothendieck abelian categories on $X$ is equivalent to a twist by a $\GG_m$-gerbe of the categorical structure sheaf of $X$ (resp. a finite \'etale stack over $X$).
\end{corollary}

Theorem \ref{theorem principal 2 introduction} is deduced from a more general result that may be applied not only to $R$-linear categories, but to graded $R$-linear categories as well. In this case, instead of considering categories enriched over $R$, we consider categories enriched over a symmetric monoidal $R$-linear category $\acal$ subject to certain tameness conditions (see theorem \ref{theo abelian con coefficients} for the precise requirements). While the main content of theorem \ref{theorem principal 2 introduction} is the \'etale local triviality of  fully dualizable categories, this no longer holds in the more general context: for instance, there exist $\ZZ/2\ZZ$-graded Azumaya algebras over algebraically closed fields which are not Morita equivalent to the unit algebra. Nevertheless, we are able to show that if $\ccal$ is a fully dualizable $\acal$-linear cocomplete category, then $\ccal$ is \'etale locally on $\Spec(R)$ equivalent to the category of modules over an Azumaya algebra in $\acal$.

\addtocontents{toc}{\protect\setcounter{tocdepth}{1}}
\subsection*{Spectral variants}

We devote the remainder of this introduction to discussing a variant of the above results that applies to linear $\infty$-categories. In this case we allow  $R$ to be a connective $E_\infty$-ring spectrum, and we are concerned with classifying fully dualizable objects of the symmetric monoidal $(\infty,2)$-category $\lincat_{R, \infty}$ of $R$-linear cocomplete $\infty$-categories. 

The role played by the abelian group $\GG_m$ in the previous discussion is now played by the $E_\infty$-group $\operatorname{GL}_1$ classifying units. As before, to each $\operatorname{GL}_1$-torsor on an affine scheme $\Spec(R)$ one may associate a twist $\Mod^\cn_{R, \mathcal{G}}$ of the $\infty$-category of connective $R$-modules, which is an invertible object of $\lincat_{R, \infty}$. We may formulate our main theorem in this setting as follows:

\begin{theorem}\label{teo introduction spectral}
 Let $R$ be a connective $E_\infty$-ring such that $\pi_0(R)$ is a G-ring. Then every fully dualizable object of $\lincat_{R, \infty}$ is of the   form $\Mod^\cn_{\tilde{R}, \mathcal{G}}$ for some finite \'etale $R$-algebra $\tilde{R}$ and $\operatorname{GL}_1$-gerbe $\mathcal{G}$ on $\smash{\Spec(\tilde{R})}$.
\end{theorem}

As in the classical setting, theorem \ref{teo introduction spectral} holds over any connective $E_\infty$-ring under an additional compact generation hypothesis. Furthermore, we deduce theorem \ref{teo introduction spectral} from a more general version that applies to $\infty$-categories linear over a base symmetric monoidal $R$-linear $\infty$-category $\Mcal$ subject to certain tameness conditions, see theorem \ref{theo prestable con coefficients}.

 Theorem \ref{teo introduction spectral} may be specialized to yield a classification of invertible linear $\infty$-categories:

\begin{corollary}\label{coro invertibles spectral}
Let $R$ be a connective $E_\infty$-ring such that $\pi_0(R)$ is a G-ring. Then every invertible object of $\lincat_{R, \infty}$ is of the   form $\Mod^\cn_{R, \mathcal{G}}$ for some  $\operatorname{GL}_1$-gerbe $\mathcal{G}$ on $\smash{\Spec(R)}$. In particular, the  group of equivalence classes of invertible objects of $\lincat_{R, \infty}$ is isomorphic to $H^2(\Spec(R), \operatorname{GL}_1) = H^2(\Spec(\pi_0(R)), \GG_m)$.
\end{corollary}

In \cite{ToenAzumaya}, To\"{e}n introduced a notion of derived Azumaya algebra over simplicial commutative rings, which was later extended to the setting of commutative ring spectra in \cite{BrauerSpectral} and \cite{AGBrauer}. If $R$ is a connective commutative ring spectrum, then an $R$-algebra $A$ is Azumaya if and only if its $\infty$-category of left module spectra defines an invertible object of the symmetric monoidal $(\infty,2)$-category $\lincat_{R, \infty, \text{st}}$ of $R$-linear cocomplete \emph{stable} $\infty$-categories. As shown by To\"{en} (in the setting of simplicial commutative rings) and Antieau-Gepner (in the setting of connective commutative ring spectra) every invertible object of $\lincat_{R, \infty, \text{st}}$  which is compactly generated arises from an Azumaya $R$-algebra, and furthermore these are classified up to Morita equivalence by $H^2(\Spec(R), \operatorname{GL}_1) \times H^1(\Spec(R), \ZZ)$.

In the unstable setting, the $\infty$-category of connective modules over a connective Azumaya algebra defines an invertible object of  $\lincat_{R, \infty}$, however not every invertible object arises in this way: this happens if and only if the associated cohomology class is torsion. It was shown by Lurie in \cite{SAG} chapter 11 that for any connective commutative ring spectrum $R$, there is an isomorphism between $H^2(\Spec(R), \operatorname{GL}_1)$ and the group of equivalence classes of invertible objects $\ccal$ of $\lincat_{R, \infty}$ such that  $\ccal$ and $\ccal^{-1}$ are compactly generated Grothendieck prestable $\infty$-categories. Corollary \ref{coro invertibles spectral} strengthens this result in the case when $\pi_0(R)$ is a G-ring, by removing all hypotheses on $\ccal$.\footnote{The requirement that $\ccal$ and $\ccal^{-1}$ are Grothendieck prestable is not a major restriction: as we shall see, every dualizable object of $\lincat_{R, \infty}$ is automatically Grothendieck prestable. The main feature of corollary \ref{coro invertibles spectral} is that it removes all compact generation hypotheses.}

One may wonder whether a variant of theorem \ref{teo introduction spectral} holds when working with $\lincat_{R, \infty, \text{st}}$ instead of $\lincat_{R, \infty}$. As stated this is too much to hope for: even if $R = k$ is a field, the derived $\infty$-category of quasicoherent sheaves on any smooth and proper variety over $k$ is a fully dualizable object which is not of the stated form. The study of compactly generated fully dualizable objects $\lincat_{k, \infty, \text{st}}$ and the question of to what extent these arise from geometric objects is the subject of active research \cite{OrlovSchemes, OrlovFinite, RaedscheldersStevenson}.

If we focus on invertible objects of $\lincat_{R, \infty, \text{st}}$ instead, one may expect that they all arise from (non necessarily connective) Azumaya $R$-algebras. As already remarked, this holds under an additional compact generation hypothesis. Our next result removes this compact generation hypothesis in the case when $R$ is truncated (that is, $R$ has finitely many nonzero homotopy groups)\footnote{An earlier version of this paper contained this result only in the case when $R$ is Artinian. An extension to the case of truncated Noetherian $E_\infty$-rings was also found, independently, by Antieau and Ramzi.}:

\begin{theorem}\label{teo principal stable introduction}
Let $R$ be a truncated connective $E_\infty$-ring. Then every invertible object of $\lincat_{R, \infty, \normalfont{\text{st}}}$ is the $\infty$-category of left modules over an Azumaya $R$-algebra.
\end{theorem}

In the language of \cite{ToenAzumaya} section 2, this shows that the derived algebraic Brauer group and the big derived categorical Brauer group agree in the truncated case.

\addtocontents{toc}{\protect\setcounter{tocdepth}{2}}
\subsection{Conventions and notation}

In the main body of the paper we use the convention where the word category stands for $\infty$-category, and use the term $(1,1)$-category or classical category if we wish to refer to the classical notion. Each category $\ccal$ has a Hom bifunctor $\Hom_\ccal(-,-)$ whose target is the category $\Spc$ of homotopy types. If $\ccal$ is equipped with an action of a monoidal category $\Mcal$ and the action admits Hom objects we denote by $\Hom^\enh_\ccal(-, -)$ the induced relative Hom bifunctor with target $\Mcal$. In other words, this is such that for every triple of objects $X$ in $\Mcal$ and $Y, Z$ in $\ccal$ we have $\Hom_\ccal(X \otimes Y, Z) = \Hom_\Mcal(X, \Hom^\enh_\ccal(Y, Z))$. We usually write $\End_\ccal(X)$ and $\End_\ccal^\enh(X)$ instead of $\Hom_\ccal(X, X)$ and $\Hom_\ccal^\enh(X, X)$.

For each category $\ccal$ and each $n \geq -2$ we denote by $\ccal_{\leq n}$ the full subcategory of $\ccal$ on the $n$-truncated objects. If the inclusion $\ccal_{\leq n} \rightarrow \ccal$ admits a left adjoint, this will be denoted by $\tau_{\leq n}$. We say that a category $\ccal$ is an $(n,1)$-category if $\ccal = \ccal_{\leq n-1}$. For each category $\ccal$ we denote by $\Ho(\ccal)$ its homotopy category; in other words, this is the universal $(1,1)$-category equipped with a functor from $\ccal$.

We denote by $\Ab$ the category of abelian groups, and $\Sp$ the category of spectra. For each commutative ring spectrum $R$ we denote by $\Mod_R$ the category of $R$-module spectra. If $R$ is connective we will denote by $\Mod_R^\cn$ the full subcategory of $\Mod_R$ on the connective $R$-module spectra, and by $\Mod_R^\heartsuit$ the full subcategory of $\Mod_R^\cn$ on the $0$-truncated objects. This applies in particular to the case when $R$ is a (classical) commutative ring: in this case $\Mod_R^\heartsuit$ is the category of $R$-modules in abelian groups, while $\Mod_R$ is its derived category.

We fix a sequence of nested universes. Objects belonging to each of the first three universes are called small, large and very large, respectively. We let $\cat$ be the category of small categories, $\cathat$ the category of large categories, $\catl$ the subcategory of $\cathat$ on the categories with small colimits and colimit preserving functors, and $\Pr^L$ the full subcategory of $\catl$ on the presentable categories. For each pair of categories $\ccal$ and $\dcal$ we denote by $\Funct(\ccal, \dcal)$ the category of functors from $\ccal$ to $\dcal$. If $\ccal$ and $\dcal$ belong to $\catl$, we let $\Funct^L(\ccal, \dcal)$ be the full subcategory of $\Funct(\ccal, \dcal)$ on the colimit preserving functors.

We will frequently consider $\catl$ and $\Pr^L$ as symmetric monoidal categories as in \cite{HA} chapter 4.8, where for each pair of objects $\Ccal, \Dcal$ the tensor product $\Ccal \otimes \Dcal$ is the universal recipient of a bifunctor from $\Ccal \times \Dcal$ which preserves colimits in each coordinate. In particular, commutative algebras in $\catl$ are the same as cocomplete symmetric monoidal categories for which the symmetric monoidal structure is compatible with colimits. Commutative algebras in $\Pr^L$ will be called presentable symmetric monoidal categories.

\subsection{Acknowledgments}

I would like to thank Bertrand Toën for mentioning to me the problem of classifying invertible presentable stable $\infty$-categories. I am also grateful to Ben Antieau, Ko Aoki, David Ben-Zvi, Jacob Lurie, Naruki Masuda, Akhil Mathew, David Reutter, Peter Scholze, and Markus Zetto for conversations related to the subject of this paper. Part of this work was carried out at the Max Planck Institute for Mathematics in Bonn, and I am grateful to the institute for its hospitality and support.


\ifx\inmain\undefined
\bibliographystyle{myamsalpha2}
\bibliography{References}
\fi


\section{Linear categories}

This section contains preliminary material on the theory of linear categories that will be used throughout the paper. We begin in \ref{subsection linear cats} with a review of the notion of cocomplete category linear over a commutative algebra in $\catl$.  This recovers in particular the notion of cocomplete category linear over a (connective) commutative ring spectrum. We include here a proof of the fact that dualizable categories linear over a presentable base are automatically presentable, which forms the first step in the proof of the main theorems of this paper.

In \ref{subsection abelian} we study the theory of Grothendieck abelian categories linear over a base symmetric monoidal Grothendieck abelian category $\Acal$. We show that some basic aspects of the theory of Grothendieck abelian categories (tensor products, Gabriel-Popescu theorem) hold in this relative context as long as we require $\Acal$ to be generated by compact projective objects and rigid. We then discuss the notion of flatness for objects in an $\Acal$-linear Grothendieck abelian category, which will be needed in section \ref{section G rings}.

In \ref{subsection spectral categories} we review the notions of spectral and semisimple categories, and prove an $\acal$-linear version of a basic structure result from \cite{Spectral} that relates spectral categories to self-injective von Neumann regular algebras. This will be used in our classification of smooth categories over a rigid semisimple base in section \ref{section invertible semisimple}.

Finally, in \ref{subsection prestables} we review the theory of Grothendieck prestable categories from \cite{SAG}, and discuss a version relative to a base symmetric monoidal Grothendieck prestable category. For the most part, the material here is parallel to  that of \ref{subsection abelian}. We also include a general discussion of how linearity interacts with the passage to derived categories.

\subsection{General notions}\label{subsection linear cats}

We begin with some  background on the notion of linear category.

\begin{definition}
Let $\Mcal$ be a commutative algebra in $\catl$. An $\Mcal$-linear cocomplete category is an $\Mcal$-module in $\catl$. An $\Mcal$-linear colimit preserving functor is a morphism of $\Mcal$-modules in $\catl$.
\end{definition}

\begin{example}
Let $\Mcal$ be a commutative algebra in $\catl$. Then $\Mcal$ has a structure of   $\Mcal$-linear cocomplete category. For every   $\Mcal$-linear cocomplete category $\ccal$, evaluation at the unit induces an equivalence between  the category of   $\mcal$-linear colimit preserving functors $\Mcal \rightarrow \ccal$ and $\ccal$. The inverse to this  equivalence associates to each object $X$ in $\ccal$ an $\Mcal$-linear enhancement of the functor $- \otimes X: \Mcal \rightarrow \ccal$. We may summarize this by saying that $\Mcal$ is the free $\Mcal$-linear cocomplete category on one object.
\end{example}

\begin{example}\label{example LMod}
Let $\Mcal$ be a commutative algebra in $\catl$ and let $A$ be an algebra in $\Mcal$. Then the category $\LMod_A(\Mcal)$ of left $A$-modules in $\Mcal$ has a structure of $\Mcal$-linear cocomplete category. Thinking about $A$ as a left $A$-module we obtain an object of $\LMod_A(\Mcal)$, which itself admits a right $A$-module structure. In fact $A$ is the universal algebra in $\Mcal$ equipped with a right action on the left $A$-module $A$: in other words, $A$ is (the opposite of) the algebra of endomorphisms of the left $A$-module $A$.

Assume now given another   $\Mcal$-linear  cocomplete category $\Ccal$. Then for each  $\Mcal$-linear colimit preserving functor $f: \LMod_A(\Mcal) \rightarrow \Ccal$ we obtain a right $A$-module $f(A)$ in $\Ccal$. The assignment $f \mapsto f(A)$ turns out to induce an equivalence between the category of  $\Mcal$-linear colimit preserving functors $\LMod_A(\Mcal) \rightarrow \Ccal$ and the category of right $A$-modules in $\ccal$ (\cite{HA} theorem 4.8.4.1). The inverse to this equivalence maps a right $A$-module $M$ in $\Ccal$ to an $\Mcal$-linear enhancement of  the relative tensor product functor $M \otimes_A - : \LMod_A(\Mcal) \rightarrow \Ccal$. We may summarize this by saying that $\LMod_A(\Mcal)$ is the universal   $\Mcal$-linear cocomplete  category on a right $A$-module.
\end{example}

If $\Mcal$ is a commutative algebra in $\catl$ then $\Mod_{\Mcal}(\catl)$ inherits a closed symmetric monoidal structure from $\catl$. It makes sense in particular to consider dualizable and invertible objects in $\Mod_{\Mcal}(\catl)$. The following proposition shows that as long as $\Mcal$ is presentable, dualizability automatically implies presentability.

\begin{proposition}\label{proposition dualizable is presentable}
Let $\Mcal$ be a presentable symmetric monoidal category and let $\ccal$ be a dualizable object of $\Mod_\Mcal(\catl)$. Then $\ccal$ is presentable.
\end{proposition}
\begin{proof}
Let $\kappa$ be the smallest large cardinal. It is proven in \cite{Pres} section 5.1 (in particular, proposition 5.1.7 and corollary 5.1.15) that $\Mod_\Mcal(\catl)$ is a very large presentable category $\kappa$-compactly generated by those $\Mcal$-modules which belong to $\Pr^L$. Since the symmetric monoidal structure on $\Mod_\Mcal(\catl)$ is compatible with large colimits and the unit is $\kappa$-compact, we have that every dualizable object is $\kappa$-compact and therefore presentable.
\end{proof}

For each commutative algebra $\Mcal$ in $\catl$ we will denote by $- \otimes_{\Mcal}-$ the tensor product on $\Mod_\Mcal(\catl)$, and by $\Funct_{\Mcal}(-,-)$ the internal Hom. If $\Mcal$ is presentable, then these bifunctors restrict to $\Mod_{\Mcal}(\Pr^L)$.

\begin{example}\label{example tensor with amod}
Let $\Mcal$ be a commutative algebra in $\catl$. Let $A$ be an algebra in $\Mcal$ and let $\Ccal$ be an   $\Mcal$-linear cocomplete category. Then we have an $\Mcal$-bilinear functor 
\[
\LMod_A(\Mcal) \times \Ccal \rightarrow \LMod_A(\Ccal)
\]
that sends a pair $(M, X)$ to $M \otimes X$.  This induces an equivalence 
\[
\LMod_A(\Mcal) \otimes_\Mcal \Ccal = \LMod_A(\Ccal)
\]
 (see \cite{HA} theorem 4.8.4.6).  In particular, if $\Ccal$ is the category $\RMod_B(\Mcal)$ of right modules over some algebra $B$ in $\Mcal$ we obtain an equivalence 
\[
\LMod_A(\Mcal) \otimes_\Mcal \RMod_B(\Mcal) = \LMod_A(\RMod_B(\Mcal)) = {}_A\kr\BMod_B(\Mcal).
\]
\end{example}

\begin{example}
Let $\Mcal$ be a commutative algebra in $\catl$ and let $A$ be an algebra in $\Mcal$. Then by example \ref{example tensor with amod} we have an equivalence 
\[
\LMod_A(\Mcal) \otimes_\Mcal \RMod_A(\Mcal) = \LMod_A(\RMod_B(\Mcal)) = {}_A\kr\BMod_A(\Mcal).
\]
 The diagonal bimodule for $A$ defines an object in $\LMod_A(\Mcal) \otimes_\Mcal \RMod_A(\Mcal)$, which then extends uniquely to an $\Mcal$-linear colimit preserving functor 
 \[
 \eta: \Mcal \rightarrow \LMod_A(\Mcal) \otimes_\Mcal \RMod_A(\Mcal).
 \]
 As discussed in \cite{HA} remark 4.8.4.8, the map $\eta$ exhibits $\LMod_A(\Mcal)$ and $\RMod_A(\Mcal)$ as dual objects in $\Mod_{\Mcal}(\catl)$. 
\end{example}

\begin{remark}
Let $\Mcal$ be a commutative algebra in $\catl$. Then we have a symmetric monoidal colimit preserving functor $\Cat \rightarrow \Mod_\Mcal(\catl)$ obtained by composing the free cocompletion functor $\Cat \rightarrow \catl$ with the free module functor $\catl \rightarrow \Mod_\mcal(\catl)$. It follows from this that $\Mod_\Mcal(\catl)$ has a structure of symmetric monoidal $2$-category, with the Hom category between two objects $\ccal$ and $\dcal$ being given by the category underlying $\Funct_{\Mcal}(\ccal, \dcal)$.

It makes sense in particular to consider adjunctions in $\Mod_\Mcal(\catl)$. Given an $\Mcal$-linear colimit preserving functor $f: \ccal \rightarrow \dcal$, we have that $f$ admits a right (resp. left) adjoint in $\Mod_\mcal(\catl)$ if and only if it admits a colimit preserving right (resp. left) adjoint as a functor of categories, which commutes strictly with the action of $\Mcal$.
\end{remark}

Assume given a morphism $f: \Mcal \rightarrow \Mcal'$ of commutative algebras in $\catl$. Then we obtain a symmetric monoidal extension of scalars functor 
\[
- \otimes_\Mcal \Mcal' : \Mod_{\Mcal}(\catl) \rightarrow \Mod_{\Mcal'}(\catl),
\] and a restriction of scalars right adjoint to it. If $\Ccal$ is an $\Mcal$-linear cocomplete category then the unit of the adjunction provides an $\Mcal$-linear colimit preserving functor $\Ccal \rightarrow \Ccal \otimes_\Mcal \Mcal'$ which we call extension of scalars along $f$. A right adjoint to it (which automatically exists if we work with presentable categories) is called restriction of scalars along $f$.

\begin{example}
Let $f: \Mcal \rightarrow \Mcal'$ be a morphism of commutative algebras in $\catl$. Let $A$ be an algebra in $\Mcal$ and consider the algebra $f(A)$ in $\Mcal'$. We may regard $\LMod_{f(A)}(\Mcal')$ as a cocomplete $\Mcal$-linear category by restriction of scalars long $f$. Then $f(A)$ becomes a right $A$-module in $\LMod_{f(A)}(\Mcal')$, so it induces an $\Mcal$-linear colimit preserving functor $\LMod_A(\Mcal) \rightarrow \LMod_{f(A)}(\Mcal')$, as discussed in example \ref{example LMod}. It follows from the universal properties of $\LMod_A(\Mcal)$ and $\LMod_{f(A)}(\Mcal')$  that this functor induces an equivalence 
\[
\LMod_A(\Mcal) \otimes_{\Mcal} \Mcal' = \LMod_{f(A)}(\Mcal').
\]
\end{example}

\begin{example}\label{example stabilize linear}
 Let $\Mcal$ be a commutative algebra in $\catl$ and let $\Mcal' = \Mcal \otimes \Sp$ be the stabilization of $\Mcal$. Then for each  $\Mcal$-linear cocomplete category $\ccal$ the extension of scalars functor $\Ccal \rightarrow \Ccal \otimes_\Mcal \Mcal'$ presents $\Ccal \otimes_\Mcal \Mcal'$ as the stabilization of $\Ccal$. In other words, the stabilization of an $\Mcal$-module is automatically a module over the stabilization of $\Mcal$. It follows in particular that the functor of restriction of scalars $\Mod_{\Mcal'}(\catl) \rightarrow \Mod_{\Mcal}(\catl)$ is fully faithful, and its image consists of those $\Mcal$-modules which are stable.
\end{example}

\begin{example}\label{example truncate linear}
Let $\Mcal$ be  a commutative algebra in $\catl$. Fix an $n \geq 1$ and  let $\Mcal' = \Mcal \otimes \Spc_{\leq n-1}$. Then for each $\Mcal$-linear cocomplete category $\ccal$ the extension of scalars functor $\Ccal \rightarrow \Ccal \otimes_{\Mcal} \Mcal'$ induces an equivalence $\Ccal \otimes_{\Mcal} \Mcal' = \Ccal \otimes \Spc_{\leq n-1}$. It follows in particular that the functor of restriction of scalars $\Mod_{\Mcal'}(\catl) \rightarrow \Mod_{\Mcal}(\catl)$ is fully faithful, and its image consists of those $\Mcal$-modules which are $(n,1)$-categories.
\end{example}

We now specialize the above discussion to obtain a notion of cocomplete categories linear over a base connective commutative ring spectrum.

\begin{definition}\label{def R linear cocomplete}
Let $R$ be a connective commutative ring spectrum. An $R$-linear cocomplete category is a $\Mod_R^\cn$-linear cocomplete category. An $R$-linear colimit preserving functor is a morphism in $\Mod_{\Mod_R^\cn}(\catl)$.
\end{definition} 

\begin{remark}
Let $R$ be a connective commutative ring spectrum. Then every $R$-linear cocomplete category is automatically additive. 
\end{remark}

\begin{remark}
Let $R$ be a connective commutative ring spectrum and let $\Ccal$ be an $R$-linear cocomplete category. The action of $\Mod_R^\cn$ on $\Ccal$ provides a monoidal functor $\Mod_R^\cn \rightarrow \Funct(\Ccal, \Ccal)$, which after passing to endomorphisms of the identity yields an $E_2$-map from $R$ into the center of $\Ccal$. In particular, given an element $x$ in $R$ and an object $X$ in $\Ccal$ we have an endofunctor of $X$ given by
\[
X = R \otimes X \xrightarrow{x \otimes \id} R \otimes X = X
\]
which we usually denote by $x: X \rightarrow X$ and call the action of $x$ on $X$.
\end{remark}

\begin{remark}
If $R$ is a (non necessarily connective) commutative ring spectrum then one may consider $\Mod_R$-linear cocomplete categories. We call these $R$-linear cocomplete stable categories. In the case when $R$ is connective, it follows from example \ref{example stabilize linear} that an $R$-linear cocomplete stable category in this sense is the same as an $R$-linear cocomplete category in the sense of definition \ref{def R linear cocomplete} which is in addition stable.
\end{remark}

\begin{remark}
For every connective commutative ring spectrum $R$ one may consider $(\Mod^\cn_R)_{\leq n-1}$-linear categories. We call these $R$-linear cocomplete $(n,1)$-categories. It follows from example \ref{example truncate linear} that an $R$-linear cocomplete $(n,1)$-category in this sense is the same as an $R$-linear cocomplete category in the sense of definition \ref{def R linear cocomplete} which is in addition an $(n,1)$-category. This will frequently be used in the case when $n = 0$ and $R$ is a (classical) commutative ring: in this case one obtains a notion of classical $R$-linear cocomplete category, which is simply a $\Mod^\heartsuit_R$-linear category.
\end{remark}

\begin{example}
Let $R$ be a connective commutative ring spectrum and let $\ccal$ be an $R$-linear cocomplete category. Then specializing examples \ref{example stabilize linear} and \ref{example truncate linear} yields the following:
\begin{itemize}
\item The stabilization $\ccal \otimes \Sp$ has a structure of $R$-linear cocomplete stable category.
\item For each $n \geq 0$ the $(n,1)$-category $\ccal \otimes \Spc_{\leq n-1}$ as a structure of $\tau_{\leq n-1}(R)$-linear cocomplete $(n,1)$-category.
\end{itemize}
\end{example}

\begin{remark}
Let $f: R \rightarrow R'$ be a morphism of connective commutative ring spectra. Then we obtain a symmetric monoidal extension of scalars functor $f^*: \Mod^\cn_R \rightarrow \Mod^\cn_{R'}$. For each $R$-linear cocomplete category $\Ccal$ we will denote by $\Ccal \otimes_R R'$ its extension of scalars along $f^*$, and by $- \otimes_R R': \Ccal \rightarrow \Ccal \otimes_R R'$ the corresponding extension of scalars functor. The composition $\Ccal \rightarrow \Ccal \otimes_R R' \rightarrow \Ccal$ is given by tensoring with the $R$-module $R'$, while the unit of the adjunction is given by tensoring with the map of $R$-modules $R \rightarrow R'$. Similar considerations apply to the case when $R$ and $R'$ are non necessarily connective commutative ring spectra and we extend scalars along $f^*: \Mod_R  \rightarrow \Mod_{R'} $.
\end{remark}


\subsection{Grothendieck abelian categories}\label{subsection abelian}

We now proceed with some recollections on the theory of Grothendieck abelian categories.

\begin{definition}\label{def grothendieck abelian}
A Grothendieck abelian category is an abelian category $\Ccal$ which is presentable and such that filtered colimits in $\Ccal$ are exact. We denote by $\Groth_1$ the category of Grothendieck abelian categories and colimit preserving functors.
\end{definition}

Definition \ref{def grothendieck abelian} is equivalent to the (perhaps more common) definition where presentability of $\Ccal$ is replaced by the requirement that $\Ccal$ is locally small, admits small colimits, and admits a generator. The following result provides an ample source of examples.

\begin{proposition}[\cite{SAG} proposition 10.6.3.1]
Let $\Ccal, \Dcal$ be presentable categories and assume given a functor $G: \Dcal \rightarrow \Ccal$ which is conservative and preserves small limits\footnote{In fact only preservation of finite limits is necessary.} and colimits. If $\Ccal$ is a Grothendieck abelian category, then so is $\Dcal$.
\end{proposition}

\begin{corollary}
Let $\Acal$ be a Grothendieck abelian category equipped with a monoidal structure compatible with colimits. Let $A$ be an algebra in $\Acal$. Then the category $\LMod_A(\Acal)$ of left $A$-modules in $\Acal$ is a Grothendieck abelian category.
\end{corollary}

Limits of Grothendieck abelian categories exist along left exact colimit preserving functors:

\begin{proposition}[\cite{SAG} proposition C.5.4.21]\label{prop limits along left exact}
Let $F: \Ical \rightarrow \Groth_1$ be a diagram whose transition maps are left exact. Then $F$ admits a limit which is preserved by the inclusion of $\Groth_1$ inside $\Pr^L$.
\end{proposition}

\begin{corollary}
The category $\Groth_1$ admits small products, and these are preserved by the inclusions $\Groth_1 \rightarrow \Pr^L \rightarrow \cathat$.
\end{corollary}

We can also form colimits of diagrams of right adjointable diagrams:

\begin{proposition}\label{proposition colimits of right adjointable}
Let $F: \Ical \rightarrow \Groth_1$ be a diagram whose transition maps admit colimit preserving right adjoints. Then $F$ admits a colimit which is preserved by the inclusion of $\Groth_1$ inside $\Pr^L$.
\end{proposition}
\begin{proof}
It suffices to show that the colimit of $F$ in $\Pr^L$ is Grothendieck abelian. This colimit agrees with the limit of the diagram $F^R : \Ical^\op \rightarrow \cathat$ obtained by passing to right adjoints of the morphisms in $F$. This a diagram of Grothendieck abelian categories and left exact colimit preserving functors. The fact that the limit is Grothendieck abelian now follows from proposition \ref{prop limits along left exact}.
\end{proof}

\begin{corollary}
The category $\Groth_1$ admits small direct sums, and these are preserved by the inclusion $\Groth_1 \rightarrow \Pr^L$. In particular, small direct sums and small direct products agree in $\Groth_1$.
\end{corollary}

There is a good theory of tensor products of Grothendieck abelian categories:

\begin{theorem}[\cite{SAG} theorem C.5.4.16, \cite{Tensor} theorem 5.4]\label{teo tensor product abelian}
Let $\Ccal, \Dcal$ be Grothendieck abelian categories. Then their tensor product $\Ccal \otimes \Dcal$ (formed in $\Pr^L$) is Grothendieck abelian. In particular, the symmetric monoidal structure on the category $\Mod_{\Ab}(\Pr^L)$ of presentable additive $(1,1)$-categories and colimit preserving functors restricts to a symmetric monoidal structure on $\Groth_1$.
\end{theorem}

In particular, it makes sense to consider commutative algebras in $\Groth_1$. We call these symmetric monoidal Grothendieck abelian categories. Note that this terminology leaves implicit the fact that the tensor operation commutes with colimits in each variable.

\begin{definition}
Let $\acal$ be a symmetric monoidal Grothendieck abelian category. An $\acal$-linear Grothendieck abelian category is an object of $\Mod_\Acal(\Groth_1)$.
\end{definition}

In other words, an $\acal$-linear Grothendieck abelian category  is an $\acal$-linear presentable category (in the sense of section \ref{subsection linear cats}) which is in addition a Grothendieck abelian category. In the case when $\acal = \Mod_R^\heartsuit$ is the category of classical modules over a connective $E_\infty$-ring $R$, we call these $R$-linear Grothendieck abelian categories.

Since the class of colimits that we have available in $\Groth_1$ is relatively restricted, one has to be careful when forming relative tensor products. There is however a class of commutative algebras that admit a well behaved theory of relative tensor products.

\begin{definition}
Let $\acal$ be a commutative algebra in $\groth$. Assume that $\acal$ is generated by compact projective objects. We say that $\acal$ is rigid if compact projective and dualizable objects of $\acal$ coincide.
\end{definition}

\begin{example}
Let $R$ be a commutative ring. Then $\Mod_R(\Ab)$ is rigid.
\end{example}

We fix for the remainder of this section a base symmetric monoidal Grothendieck abelian category $\acal$, generated by compact projective objects and rigid.

\begin{remark}\label{remark tensor compact projectives}
 Let $\Ccal$ be an $\Acal$-linear Grothendieck abelian category. Then if $X$ is a compact projective object of $\Acal$ the functor $X \otimes - : \Ccal \rightarrow \Ccal$ admits both a left and a right adjoint, given by $X^\vee \otimes -$. In particular, $X \otimes -$ admits a colimit preserving right adjoint, and therefore it maps compact projective objects of $\Ccal$ to compact projective objects.
\end{remark}

\begin{remark}
Let $f: \Ccal \rightarrow \Dcal$ be a morphism in $\Mod_\Acal(\catl)$ and assume that $f$ admits a right adjoint $f^R$ (as a functor of categories, ignoring the $\Acal$-action). Then $f^R$ commutes laxly with the action of $\Acal$, and strictly with the action of the dualizable objects in $\Acal$. In particular, since $\Acal$ is generated under colimits by dualizable objects, we see that if $f^R$ is colimit preserving then it commutes strictly with the action of $\Acal$.
\end{remark}

\begin{proposition}\label{proposition action has colimit preserving adjoint}
Let $\ccal$ be an $\acal$-module in $\catl$. Then the action map $\acal \otimes \ccal \rightarrow \ccal$ admits a colimit preserving right adjoint.
\end{proposition}
\begin{proof}
We first prove the proposition in the case when $\Ccal = \Acal \otimes \Ccal'$ is a free $\Acal$-module. In this case the action map is obtained by tensoring the multiplication map $\mu: \Acal \otimes \Acal \rightarrow \Acal$ with the identity on $\Ccal'$. We may thus reduce to showing that $\mu$ admits a colimit preserving right adjoint. This is a consequence of the fact that $\Acal$ is generated by compact projective objects and that such objects are preserved by tensor products.

We now prove the general case. Consider the Bar resolution $\Acal^{\otimes \bullet + 1} \otimes \Ccal$ of $\Ccal$. Then the action map $\Acal \otimes \Ccal \rightarrow \Ccal$ is the colimit of the action maps $\Acal \otimes (\Acal^{\otimes \bullet + 1} \otimes \Ccal) \rightarrow \Acal^{\otimes \bullet + 1} \otimes \Ccal$. Since the Bar resolution is levelwise free we see that each of these maps admits a colimit preserving right adjoint. To prove the proposition it remains to show that for each face map $\sigma: [n] \rightarrow [n+1]$ in $\Delta$ the induced commutative square
\begin{equation}
\begin{tikzcd}
\Acal \otimes (\Acal^{\otimes n+2} \otimes \Ccal) \arrow{d}{} \arrow{r}{} & \Acal \otimes (\Acal^{\otimes n+1} \otimes \Ccal) \arrow{d}{} \\
\Acal^{\otimes n+2} \otimes \Ccal \arrow{r}{}  & \Acal^{\otimes n+1} \otimes \Ccal
\end{tikzcd}
\end{equation}
is vertically right adjointable. If $\sigma$ is not the $0$-th face then the above square is a tensor product of the square
\[
\begin{tikzcd}
\Acal \otimes \Acal \arrow{d}{\mu} \arrow{r}{\id} & \Acal \otimes \Acal \arrow{d}{\mu} \\
\Acal \arrow{r}{\id} & \Acal 
\end{tikzcd}
\] 
with
\[
\begin{tikzcd}
\Acal^{\otimes n+1} \otimes \Ccal \arrow{d}{\id} \arrow{r}{} & \Acal^{\otimes n} \otimes \Ccal \arrow{d}{\id} \\
\Acal^{\otimes n+1} \otimes \Ccal \arrow{r}{}  & \Acal^{\otimes n} \otimes \Ccal
\end{tikzcd}
\]
and our claim follows from the fact that these two are vertically right adjointable. It remains to analyze the case when $\sigma$ is the $0$-th face. In this case the square (1) is obtained by tensoring the square
\[
\begin{tikzcd}
\Acal \otimes \Acal \otimes \Acal \arrow{d}{\mu \otimes \id} \arrow{r}{\id \otimes \mu} & \Acal \otimes \Acal \arrow{d}{\mu} \\
\Acal \otimes \Acal \arrow{r}{\mu} & \Acal
\end{tikzcd}
\]
with $\Acal^{\otimes n} \otimes \Ccal$.  We may thus reduce to showing that the above is vertically right adjointable. This amounts to showing that the right adjoint to the multiplication map $\Acal \otimes \Acal \rightarrow \Acal$ is $\Acal$-linear. This follows from the fact that $\Acal$ is generated under colimits by dualizable objects.
\end{proof}

\begin{corollary}\label{coro tensor products over A}
The full subcategory of $\Mod_\acal(\Pr^L)$ on the $\acal$-linear Grothendieck abelian categories is closed under tensor products. In other words, $\Mod_\acal(\groth_1)$ admits a symmetric monoidal structure that makes the inclusion $\Mod_\acal(\groth_1) \rightarrow \Mod_\acal(\Pr^L)$ symmetric monoidal.
\end{corollary}
\begin{proof}
Let $\Ccal$ and $\Dcal$ be a pair of $\acal$-linear Grothendieck abelian categories. The relative tensor product $\ccal \otimes_\acal \Dcal$ in $\Pr^L$  is the geometric realization of the Bar construction $\ccal \otimes \acal^{\otimes \bullet} \otimes \Dcal$. By virtue of proposition \ref{proposition colimits of right adjointable}, to show that $\ccal \otimes_\acal \Dcal$ is Grothendieck abelian, it suffices to show that the face maps in the Bar construction admit colimit preserving right adjoints. This follows from proposition \ref{proposition action has colimit preserving adjoint}.
\end{proof}

\begin{corollary}
Let $\ccal$ and $\dcal$ be a pair of $\acal$-linear Grothendieck abelian categories. Then the functor $\ccal \otimes \dcal \rightarrow \ccal \otimes_\acal \dcal$ admits a colimit preserving right adjoint. In particular, for each pair of compact projective objects $X$ in $\ccal$ and $Y$ in $\dcal$, the object $X \otimes Y$ in $\ccal \otimes_\acal \dcal$ is compact projective.
\end{corollary}
\begin{proof}
Follows directly from the fact that the functor $\ccal \otimes \dcal \rightarrow \ccal \otimes_\acal \dcal$ arises from the geometric realization of a simplicial diagram whose face maps have colimit preserving right adjoints.
\end{proof}

We will frequently use the following relative variant of the notion of generator:

\begin{definition}\label{definition A generator}
Let $\ccal$ be an $\acal$-linear Grothendieck abelian category. An object $G$ in $\ccal$ is said to be an $\acal$-generator if $\ccal$ is generated by the family of objects $X \otimes G$ with {$X$ in $\acal$.} 
\end{definition}

\begin{remark}
Let $\Ccal$ be an $\acal$-linear Grothendieck abelian category. Then an object $G$ in $\ccal$ is an $\acal$-generator if and only if $\ccal$ is generated by the family of objects $X \otimes G$ with $X$ a compact projective object of $\acal$.
\end{remark}

The following is an $\acal$-linear version of the Gabriel-Popescu theorem:

\begin{proposition}\label{prop lex localization}
Let $\Ccal$ be an $\Acal$-linear Grothendieck abelian category and let $G$ be an $\acal$-generator for $\Ccal$. Let $A$ be the opposite of the algebra of endomorphisms of $G$ associated to the action of $\Acal$ on $\Ccal$. Then the functor
\[
G \otimes_A - : \LMod_A(\Acal) \rightarrow \Ccal 
\]
is an $\Acal$-linear left exact localization. Furthermore, it is an equivalence if and only if $G$ is compact projective.
\end{proposition}
\begin{proof}
Let $\Ccal_0$ be the full subcategory of $\ccal$ on the objects of the form $X \otimes G$ where $X$ is a compact projective object of $\Acal$, and let $\Dcal$ be the full subcategory of $\LMod_A(\Acal)$ on the objects of the form $X \otimes A$ where $X$ is a compact projective object of $\Acal$. If $X, Y$ are compact projective objects of $\Acal$ then we have
\begin{align*}
\Hom_{\LMod_A(\Acal)}(X \otimes A, Y \otimes A) &= \Hom_{\LMod_A(\Acal)}(X \otimes Y^\vee \otimes A, A) \\ &= \Hom_{\Ccal}(X \otimes Y^\vee \otimes G, G)  \\ &= \Hom_{\Ccal}(X \otimes G, Y \otimes G)
\end{align*}
and therefore $G \otimes_A -$ restricts to an equivalence $\Dcal \rightarrow \Ccal_0$. We now have a commutative square of  categories
\[
\begin{tikzcd}
\Pcal^1_\Sigma(\Dcal) \arrow{d}{} \arrow{r}{} & \Pcal^1_\Sigma(\Ccal_0) \arrow{d}{} \\
\LMod_A(\Acal) \arrow{r}{ G \otimes_A -} & \Ccal
\end{tikzcd}
\]
where the categories on the top row are the $(1,1)$-categories obtained from $\Dcal$ and $\Ccal_0$ by freely adjoining sifted colimits, and the vertical arrows are the unique sifted colimit preserving extensions of the inclusions $\Dcal \rightarrow \LMod_A(\Acal)$ and $\Ccal_0 \rightarrow \Ccal$. The upper horizontal arrow is an equivalence since $G \otimes_A -$ restricts to an equivalence $\Dcal \rightarrow \Ccal_0$. Furthermore, $\Dcal$ is a generating family of compact projective objects of $\LMod_A(\Acal)$, and therefore the left vertical arrow is an equivalence as well. 

The fact that $G \otimes_A -$ is a left exact localization now follows from the fact that the functor $\Pcal^1_\Sigma(\Ccal_0) \rightarrow \Ccal$ is a left exact localization, due to the many object version of the classical Gabriel-Popescu theorem (see \cite{Kuhn} theorem 2.1, or theorem C.2.2.1 of \cite{SAG}). It remains to show that $G \otimes_A -$ is an equivalence if and only if $G$ is compact projective. The only if direction follows from the fact that $A$ is a compact projective left $A$-module. To prove the if direction, we observe that if $G$ is compact projective then $\Ccal_0$ is a generating family of compact projective objects of $\Ccal$, so that the functor $\Pcal^1_\Sigma(\Ccal_0) \rightarrow \Ccal$ is an equivalence.
\end{proof}

We finish this section with a discussion of flatness in the context of $\Acal$-linear Grothendieck abelian categories.

\begin{definition}\label{definition flat abelian}
Let $\Ccal$ be an $\Acal$-linear Grothendieck abelian category. We say that an object $X$ in $\Ccal$ is flat over $\acal$ if the functor $- \otimes X : \Acal \rightarrow \Ccal$ is left exact. In cases when the base $\acal$ is clear from the context we simply  say that $X$ is flat.
\end{definition}

\begin{example}
Let $R$ be a commutative ring and let $\Ccal$ be an $R$-linear Grothendieck abelian category. Since every monomorphism of $R$-modules is a transfinite composition of pushouts of inclusions of ideals into $R$, we have that an object $X$ in $\Ccal$ is flat if and only if the morphism
\[
I \otimes X \xrightarrow{i \otimes \id} R \otimes X = X
\]
is a monomorphism for all inclusions of ideals $i: I \rightarrow R$.
\end{example}

As a particular case of definition \ref{definition flat abelian} we obtain a notion of flatness for objects of $\acal$. These are characterized by the following variant of Lazard's theorem:

\begin{proposition}\label{prop lazard classico}
Let $X$ be an object of $\Acal$. Then $X$ is flat if and only if it is a filtered colimit of compact projective objects.
\end{proposition}
\begin{proof}
We first show that if $X$ is a filtered colimit of compact projective objects then it is flat. Since filtered colimits in $\Acal$ are left exact it suffices to consider the case when $X$ is compact projective. This follows from the fact that the functor $- \otimes X : \Acal \rightarrow \Acal$ has a left adjoint.

Assume now that $X$ is flat. Let $\acal^\cp$ be the full subcategory of $\acal$ on the compact projective objects and consider the functor $F(-): (\acal^\cp)^\op \rightarrow \Spc$ represented by $X$. We wish to show that this functor defines an ind-object of $\acal^\cp$. Let $D: \acal^\cp \rightarrow (\acal^\cp)^\op$ be the dualization equivalence. We will prove that $F(D(-)): \acal^\cp \rightarrow \Spc$ defines a pro-object of $\acal^\cp$.

Let $p: \Ecal \rightarrow \acal$ be the left fibration associated to the functor $\Hom_{\acal}(1_\acal, - \otimes X)$. Then the base change of $p$ to $\acal^\cp$ is the left fibration classifying $F(D(-))$. We have to show that every finite diagram $G: \Ical \rightarrow \Ecal \times_\acal \acal^\cp$ admits a left cone. The fact that $X$ is flat implies that the functor $\Hom_{\acal}(1_\acal, - \otimes X)$ is left exact, and therefore $G$ extends to a left cone $G^\lhd: \Ical^\lhd \rightarrow \Ecal$. Let $\overline{Y} = (Y, \rho: 1_\acal \rightarrow Y \otimes X)$ be the value of $G^\lhd$ at the cone point. To show that $G$ extends to a left cone in $\Ecal \times_\acal \acal^\cp$ it is enough to prove that $\overline{Y}$ receives a map from an object in $\Ecal \times_\acal \acal^\cp$. This amounts to showing that there exists a map $Y' \rightarrow Y$ from a compact projective object with the property that $\rho$ factors through $Y' \otimes X$. This follows from the fact that $1_\acal$ is compact projective.
\end{proof}

We now study the behavior of flatness under tensor products.

\begin{proposition}\label{prop tensoring left exact}
Let $f: \Ccal \rightarrow \Ccal'$ and $g: \Dcal \rightarrow \Dcal'$   morphisms in $\Mod_{\Acal}(\Groth_1)$. If $f$ and $g$ are left exact then $f \otimes_{\Acal} g : \Ccal \otimes_\acal \Dcal \rightarrow \Ccal' \otimes_\Acal \Dcal'$ is left exact.
\end{proposition}
\begin{proof}
Since $f \otimes_\Acal g$ is the composition of $f \otimes_\Acal \id_{\Dcal}$ and $\id_{\Ccal'} \otimes_\Acal g$, it suffices to prove that these two functors are left exact. Changing the role of $f$ and $g$ we may reduce to showing that $f \otimes_\Acal \id_{\Dcal}$ is left exact. Pick an algebra $B$ in $\Acal$ and an $\Acal$-linear left exact localization $q: \LMod_B(\Acal) \rightarrow \Dcal$. We have a commutative square of categories
\[
\begin{tikzcd}
\Ccal\otimes_\Acal \LMod_B(\Acal) \arrow{r}{f \otimes  \id} \arrow{d}{\id \otimes  q} & \Ccal' \otimes_\Acal \LMod_B(\Acal) \arrow{d}{\id \otimes  q} \\
\Ccal \otimes_\Acal \Dcal \arrow{r}{f \otimes  \id} & \Ccal' \otimes_\Acal \Dcal.
\end{tikzcd}
\]
Here the upper horizontal arrow is equivalent to the functor $\LMod_B(\Ccal) \rightarrow \LMod_B(\Ccal')$ induced by $f$, and is therefore left exact. To prove the proposition it will suffice to show that the left vertical arrow is a left exact localization, and that the right vertical arrow is left exact. Changing the role of $\Ccal$ and $\Ccal'$ we see that it suffices to show that the left vertical arrow is a left exact localization.

Pick an algebra $A$ in $\Acal$ and an $\Acal$-linear left exact localization $p: \LMod_A(\Acal) \rightarrow \Ccal$. We now have a commutative square of categories
\[
\begin{tikzcd}
\LMod_A(\Acal) \otimes_\acal \LMod_B(\Acal) \arrow{r}{p \otimes  \id} \arrow{d}{\id \otimes  q} & \Ccal \otimes_\Acal \LMod_B(\Acal) \arrow{d}{\id \otimes   q} \\
\LMod_A(\Acal) \otimes_\acal  \Dcal \arrow{r}{p \otimes \id} & \Ccal \otimes_\Acal \Dcal.
\end{tikzcd}
\]
The upper horizontal arrow is equivalent to the functor $\LMod_B(\LMod_A(\Acal))  \rightarrow \LMod_B(\Ccal)$ induced by $p$, and is therefore a left exact localization. Similarly, the left vertical arrow is a left exact localization. To prove the proposition it will suffice to show that the diagonal map $p \otimes  q$ is a left exact localization as well.

We have that $p \otimes q$ is an epimorphism in $\Pr^L$, with the property that a map 
\[
f: \LMod_A(\Acal) \otimes_\Acal \LMod_B(\Acal)  \rightarrow \Ecal
\]
factors through $\Ccal \otimes_\Acal \Dcal$ if and only if its restriction to $\LMod_A(\Acal) \otimes \LMod_B(\Acal)$ factors through $\Ccal \otimes \Dcal$. Similarly, the upper horizontal arrow (resp. left vertical arrow) is an epimorphism with the property that $f$ factors through it if and only if the restriction of $f$ to $\LMod_A(\Acal) \otimes \LMod_B(\Acal)$  factors through $\Ccal \otimes \LMod_{B}(\Acal)$ (resp. $\LMod_A(\Acal) \otimes \Dcal$). It follows that $f$ factors through $p \otimes  q$ if and only if it factors through both $p \otimes  \id$ and $\id \otimes  q$, so that $p \otimes  q$ is a localization at the union of the class of arrows inverted by the latter two maps. The fact that $p \otimes  q$ is left exact now follows from \cite{SAG} lemma C.4.3.1.
\end{proof}

\begin{corollary}\label{coro tensor flats}
Let $\Ccal, \Dcal$ be $\Acal$-linear Grothendieck abelian categories, and let $X, Y$ be flat objects of $\Ccal$ and $\Dcal$ respectively. Then the object $X \otimes Y$ in $\Ccal \otimes_\Acal \Dcal$ is flat.
\end{corollary}
\begin{proof}
Let $F: \Acal \rightarrow \Ccal$ (resp. $G: \Acal \rightarrow \Dcal$) be the unique $\Acal$-linear colimit preserving functor sending the unit to $X$ (resp. $Y$). Then $X \otimes Y$ is the image of the unit under the composite functor
\[
\Acal = \Acal \otimes_\Acal \Acal \xrightarrow{F \otimes G} \Ccal \otimes_\Acal \Dcal.
\]
To show that $X \otimes Y$ is flat we must show that the above functor is left exact. This is a direct consequence of proposition \ref{prop tensoring left exact}.
\end{proof}


\subsection{Spectral categories}\label{subsection spectral categories}

We now review the notion of spectral categories from \cite{Spectral}.

\begin{definition}
Let $\Ccal$ be a Grothendieck abelian category. We say that $\Ccal$ is spectral if every exact sequence in $\Ccal$ splits.
\end{definition}

Various finiteness conditions become equivalent for objects in a spectral category:

\begin{proposition}\label{prop equiv finiteness}
Let $\Ccal$ be a spectral category and let $X$ be an object in $\Ccal$. The following are equivalent:
\begin{enumerate}[\normalfont(a)]
\item $X$ is a finite direct sum of simple objects.
\item $X$ is compact.
\item $X$ is finitely generated.\footnote{An object $X$ in a Grothendieck abelian category is said to be finitely generated if $X$ is compact as an object in its poset of subobjects.}
\end{enumerate}
\end{proposition}
\begin{proof}
Condition (b) clearly implies (c). Assume now that $X$ is finitely generated. Since every subobject of $X$ is a direct summand of $X$, we see that every subobject of $X$ is also finitely generated. Hence $X$ is Noetherian. Assume now given a decreasing sequence of subobjects $X_n$ of $X$. Using the fact that $\Ccal$ is spectral we may inductively construct a sequence of complements $X_n^c$ for $X_n$ with the property that $X_n^c \subseteq X_{n+1}^c$ for all $n$. Since $X$ is Noetherian we have that the sequence $X_n^c$ is eventually constant, and hence $X_n$ is eventually constant as well. We conclude that $X$ is also Artinian, and so it has finite length. The fact that $X$ is spectral now implies that $X$ is a finite direct sum of simple objects. Thus we see that (c) implies (a).

It remains to show that (a) implies (b). For this it suffices to show that if $S$ is a simple object in $\Ccal$, then $S$ is compact. We will do so by showing that  $\Hom^\enh_\Ccal(S, -) : \Ccal \rightarrow \Ab$ preserves colimits. The fact that it is right exact follows from the fact that $\Ccal$ is spectral. We may therefore reduce to showing that $\Hom^\enh_\Ccal(S, -)$ preserves infinite direct sums. 

Let $Y_\alpha$ be a family of objects of $\Ccal$ indexed by a set $\Lambda$. We need to show that the map
\[
\bigoplus_{\alpha \in \Lambda} \Hom^\enh_\Ccal(S, Y_\alpha) \rightarrow \Hom^\enh_\Ccal(S, \bigoplus_{\alpha \in \Lambda} Y_\alpha)
\]
is an isomorphism. The fact that the above is a monomorphism is a general fact about Grothendieck abelian categories (and does not use the simplicity of $S$). It remains to show that every morphism $S \rightarrow \bigoplus_{\alpha \in \Lambda} Y_\alpha$ factors through $\bigoplus_{\alpha \in \Lambda'} Y_\alpha$ for some finite subset $\Lambda' \subseteq \Lambda$. This follows from the fact that $S$ is finitely generated, since $\bigoplus_{\alpha \in \Lambda}Y_\alpha$ is the filtered union of the subobjects $\bigoplus_{\alpha \in \Lambda'} Y_\alpha$ over all finite $\Lambda'$.
\end{proof}

The most familiar spectral categories are the semisimple ones, which admit a number of equivalent characterizations:

\begin{proposition}\label{prop equivalences semisimple}
Let $\Ccal$ be a Grothendieck abelian category. The following are equivalent:
\begin{enumerate}[\normalfont (a)]
\item $\Ccal$ is locally finitely generated\footnote{A Grothendieck abelian category $\Ccal$ is said to be locally finitely generated if it is generated by its finitely generated objects.} and spectral.
\item $\Ccal$ is generated by compact projective objects and spectral.
\item Every object of $\Ccal$ is a direct sum of simple objects.
\end{enumerate}
\end{proposition}
\begin{proof}
The fact that (a) and (b) are equivalent, and the fact that (c) implies these, are both consequences of proposition \ref{prop equiv finiteness}. It remains to show that if (a) and (b) hold then every object of $\Ccal$ is a direct sum of simple objects. Let $X$ be an object of $\Ccal$. We construct a strictly increasing transfinite sequence of subobjects $X_\alpha$ of $X$ by induction as follows:
\begin{itemize}
\item Let $X_0 = 0$.
\item If $\alpha$ is a limit ordinal then we let $X_\alpha = \colim_{\beta < \alpha} X_\beta$.
\item Assume $\alpha = \beta + 1$ is a successor ordinal and $X_\alpha \neq X$. Choose a complement $Y$ for $X_\beta$ inside $X$. Since $\Ccal$ is assumed to be locally finitely generated we may pick a nonzero finitely generated subobject $Y'$ of $Y$. Since $\Ccal$ is spectral, an application of proposition \ref{prop equiv finiteness} shows that $Y'$ contains a simple subobject $S$. We let $X_\alpha = X_\beta \oplus S$.
\end{itemize}
The above construction ends whenever it reaches a small cardinal $\alpha$ with $X_\alpha = X$. The proposition now follows from the fact that $X_\alpha$ is a direct sum of simple objects for all $\alpha$.
\end{proof}

\begin{definition}
Let $\Ccal$ be a Grothendieck abelian category. We say that $\ccal$ is semisimple if it satisfies the equivalent conditions of proposition \ref{prop equivalences semisimple}.
\end{definition}

We will be interested in understanding spectral categories linear over a base. We fix for the remainder of this section a symmetric monoidal Grothendieck abelian category $\acal$, rigid and generated by compact projective objects.  We will need the notion of von Neumann regularity for algebras in $\acal$. Before introducing this notion we recall some basic concepts from ring theory in the relative context:

\begin{definition}
Let $A$ be an algebra in $\Acal$. 
\begin{itemize}
\item A left ideal of $A$ is a subobject of $A$ in $\LMod_A(\Acal)$.
\item A left $A$-module $M$ is said to be flat if the functor $- \otimes_A M : \RMod_A(\acal) \rightarrow \acal$ is left exact.
\item We say that $R$ is left self-injective if $R$ is an injective object of $\LMod_A(\acal)$.
\end{itemize}
We also define the right variants of the above notions in a similar way.
\end{definition}

\begin{lemma}\label{lemma equivalences dualizable}
Let $A$ be an algebra in $\Acal$ and let $M$ be a left $A$-module. The following are equivalent:
\begin{enumerate}[\normalfont(a)]
\item The left module $M$ admits a left dual.
\item The functor $- \otimes_A M : \RMod_A(\acal) \rightarrow \acal$   preserves limits.
\item $M$ is finitely generated projective.
\item $M$ is a retract of a free left $A$-module on a dualizable object of $\Acal$.
\item $M$ is finitely presented and flat.
\end{enumerate}
\end{lemma}
\begin{proof}
The existence of a left dual to $M$ is equivalent to the existence of an $\Acal$-linear left adjoint to $-\otimes_A M$. Since $\Acal$ is generated under colimits by its dualizable objects any left adjoint to $- \otimes_A M$ is automatically $\Acal$-linear. The equivalence of (a) and (b) now follows directly from the adjoint functor theorem.

The fact that (c) implies (d) is a direct consequence of the fact that $\LMod_A(\Acal)$ is generated by free modules on dualizable objects. The fact that (d) implies (c) follows from the fact that the property of being compact projective is stable under retracts and passage to free modules. 

We now show that (a) implies (d). If $M$ admits a left dual $M^\vee$ then  we have 
\[
\Hom_{\RMod_A(\Acal)}(M^\vee, -) = \Hom_\Acal(1_\acal, - \otimes_A M)
\] which preserves sifted colimits since the unit in $\Acal$ is compact projective. It follows that in this case $M^\vee$ is a compact projective right $A$-module. Since $\RMod_A(\Acal)$ is generated by objects of the form $A \otimes V$ with $V$ a dualizable object of $\Acal$ we conclude that $M^\vee$ is a retract of $V \otimes A$ for some dualizable $V$. The right $A$-module $V \otimes A$ admits a right dual given by $V^\vee \otimes A$. It follows that $M$ is a retract of $V^\vee \otimes A$, so that (d) holds.

Assume now that (d) holds, so that $M$ is retract of $A \otimes V$ for some dualizable $V$. Then $- \otimes_A M$ is a retract of the composite functor
\[
\RMod_A(\Acal) \xrightarrow{ - \otimes_A A}  \Acal \xrightarrow{ \otimes V} \Acal.
\]
Each of the two functors above preserve limits. It follows that $- \otimes_A M$ preserves limits, so that (b) holds.

It remains to show that properties (a) through (d) are equivalent to (e). Assume first that (a) through (d) hold. The flatness of $M$ is then a consequence of (b), while the fact that $M$ is finitely presented follows from (d).
 
We finish the proof by showing that (e) implies (b). Since $M$ is flat, it is enough to show that $- \otimes_A M$ preserves products. Pick an exact sequence $P' \rightarrow P \rightarrow M \rightarrow 0$ with $P, P'$ finitely generated projective left $A$-modules. Let $N_\alpha$ be a family of right $A$-modules. Then we have a commutative diagram
\[
\begin{tikzcd}
(\prod_\alpha N_\alpha) \otimes_A P' \arrow{d}{} \arrow{r}{} & (\prod_\alpha N_\alpha) \otimes_A P \arrow{r}{} \arrow{d}{} & (\prod_\alpha N_\alpha) \otimes_A M \arrow{r}{} \arrow{d}{} & 0  \\
\prod_\alpha N_\alpha \otimes_A P' \arrow{r}{} &\prod_\alpha N_\alpha \otimes_A P \arrow{r}{} & \prod_\alpha N_\alpha \otimes_A M \arrow{r}{}& 0.
\end{tikzcd}
\]
The upper row is evidently exact, and the bottom row is exact since $\Acal$ is generated by compact projective objects (and thus products are exact in $\acal$). We now finish by observing that the left and middle vertical arrows are isomorphisms, by applying the equivalence of (b) and (c) for the modules $P$ and $P'$.
\end{proof}

\begin{proposition}\label{prop equivalences VN}
Let $A$ be an algebra in $\Acal$. The following are equivalent:
\begin{enumerate}[\normalfont(a)]
\item Every finitely generated submodule of a finitely generated projective left $A$-module is a direct summand.
\item Every finitely generated submodule of a finitely generated projective right $A$-module is a direct summand.
\item Every finitely presented left $A$-module is projective.
\item Every finitely presented right $A$-module is projective.
\item Every left $A$-module is flat.
\item Every right $A$-module is flat.
\end{enumerate}
\end{proposition}
\begin{proof}
Let $N \subseteq M$ be a finitely generated submodule of a finitely generated projective left $A$-module. Then $M/N$ is a finitely presented left $R$-module. Conversely, every finitely presented left $A$-module may be written in such a way. It follows from this that (a) and (c) are equivalent.

Since every left $A$-module is a filtered colimit of finitely presented left $A$-modules, and filtered colimits preserve flatness, we see that (c) implies (e). The fact that (e) implies (c) is a direct consequence of lemma \ref{lemma equivalences dualizable}.

The same arguments applied to $A^\op$ prove the equivalence of (b), (d), and (f). To finish it suffices to show that the left versions imply the right versions. Assume that (a) holds. Let $i: M' \rightarrow M$ be an inclusion of left $A$-modules and let $N$ be a right $A$-module. We will show that $N \otimes_A i$ is a monomorphism.

Write $M$ as a filtered colimit of a family of finitely presented left $A$-modules $M_\alpha$. Then $i$ is a filtered colimit of the induced inclusions $i_\alpha: M' \times_{M} M_\alpha \rightarrow M_\alpha$. It suffices to show that $N \otimes_A i_\alpha$ is a monomorphism for all $\alpha$. In other words, we may reduce to the case when $M$ is finitely presented.

Write $M'$ as a filtered union of finitely generated subobjects $M'_\beta$. Then $i$ is the filtered colimit of the inclusions $i_\beta: M'_\beta \rightarrow M$, and it suffices to show that $N \otimes_A i_\beta$ is a monomorphism for all $\beta$. In other words, we may further reduce to the case when $M'$ is finitely generated. This now follows from the fact that $i$ is the inclusion of a summand.
\end{proof}

\begin{definition}
Let $A$ be an algebra in $\Acal$. We say that $A$ is von Neumann regular if it satisfies the equivalent conditions of proposition \ref{prop equivalences VN}.
\end{definition}

The following is an $\acal$-linear version of \cite{Spectral} theorem 2.1:

\begin{proposition}\label{prop classif spectral}
Let $\Ccal$ be an $\Acal$-linear spectral category. Then there exists a left self-injective von Neumann regular algebra $A$ in $\Acal$ and an $\Acal$-linear left exact localization 
\[
\LMod_A(\Acal) \rightarrow \Ccal.
\]
\end{proposition}
\begin{proof}
Let $G$ be an $\acal$-generator for $\Ccal$ and let $A$ be the opposite to the algebra of endomorphisms of $G$. We will show that $A$ is left self-injective von Neumann regular.

Denote by $q: \LMod_A(\Acal) \rightarrow \Ccal$ the functor of tensoring with $G$ and by $i$ its right adjoint. Since $q$ is left exact and every object of $\Ccal$ is injective, we see that every left $A$-module in the image of $i$ is injective. In particular this holds for $i(G) = A$, and so $A$ is left self-injective.

It remains to show that $A$ is von Neumann regular. We will do so by showing that condition (a) in proposition \ref{prop equivalences VN} holds. Let $M$ be a finitely generated projective left $A$-module and let $N$ be a finitely generated submodule of $M$. We may write $N$ as the image of a map $\alpha: M' \rightarrow M$ of finitely generated projective left $A$-modules. Each of $M$ and $M'$ is a direct summand of a left $A$-module of the form $A \otimes X$ with $X$ a dualizable object of $\Acal$. Since $A \otimes X = i(G \otimes X)$ and $\Ccal$ is idempotent complete, we see that $M$ and $M'$, and therefore also $\alpha$, belong to the image of $i$. Since $\Ccal$ is spectral, every morphism in $\Ccal$ may be written as the composition of a retraction followed by a section. Hence $\alpha$ is a composition of a retraction followed by a section. This is necessarily equivalent to the image factorization for $\alpha$, so we deduce that the inclusion $N \rightarrow M$ is a section, as desired.
\end{proof}


\subsection{Grothendieck prestable categories} \label{subsection prestables}

We now review the theory of Grothendieck prestable categories, as introduced in \cite{SAG} appendix C.

\begin{definition}\label{def prestable}
A Grothendieck prestable category is a presentable category $\Ccal$ satisfying the following properties:
\begin{enumerate}[\normalfont (a)]
\item The initial and final objects of $\Ccal$ agree (that is, $\Ccal$ is pointed).
\item Every cofiber sequence in $\Ccal$ is also a fiber sequence.
\item Every map in $\Ccal$ of the form $f: X \rightarrow \Sigma(Y)$ is the cofiber of its fiber.
\item Filtered colimits and finite limits commute in $\Ccal$.
\end{enumerate}

We denote by $\Groth_\infty$ the full subcategory of $\Pr^L$ on the Grothendieck prestable categories.
\end{definition}

\begin{remark}\label{remark groth prestable vs t structure}
Let $\Ccal$ be a Grothendieck prestable category. Then the functor $\Ccal \rightarrow \Ccal \otimes \Sp = \Sp(\Ccal)$ is fully faithful, and identifies $\Ccal$ with the connective half of t-structure on $\Sp(\Ccal)$. In particular, $\Ccal$ is additive, and moreover it makes sense to consider for each nonnegative integer $n$ the homology functor $H_n = \Omega^n: \Ccal \rightarrow \Ccal^\heartsuit = \Ccal_{\leq 0}$. Note that $\Ccal^\heartsuit$ is a Grothendieck abelian category.
\end{remark}

It turns out that the assignment $\Ccal \mapsto \Sp(\Ccal)$ provides a one to one correspondence between Grothendieck prestable categories and presentable stable categories equipped with a with right complete t-structure compatible with filtered colimits. One virtue of working with Grothendieck prestable categories instead of t-structures is that being Grothendieck prestable is a property of a category (as opposed to a t-structure on a presentable stable category which is a piece of structure).

\begin{example}
Let $\Ccal$ be a Grothendieck abelian category. Then the derived category $\der(\Ccal)$ is a presentable stable category with a right complete t-structure compatible with filtered colimits. By virtue of remark \ref{remark groth prestable vs t structure} the connective half of this t-structure is Grothendieck prestable. We denote this category by $\der(\Ccal)_{\geq 0}$.
\end{example}

The following result provides an ample source of Grothendieck prestable categories:

\begin{proposition}[\cite{SAG} proposition 10.4.3.1]
Let $\Ccal, \Dcal$ be presentable categories and assume given a functor $G: \Ccal \rightarrow \Dcal$ which is conservative and preserves small limits\footnote{Only preservation of finite limits is necessary.} and colimits. If $\Ccal$ is a Grothendieck prestable category then so is $\Dcal$.
\end{proposition}

\begin{corollary}
Let $\Acal$ be a Grothendieck prestable category equipped with a monoidal structure compatible with colimits. Let $A$ be an algebra in $\Acal$. Then the category $\LMod_A(\Acal)$ of left $A$-modules in $\Acal$ is a Grothendieck prestable category.
\end{corollary}

The notion of projective object in a Grothendieck abelian category admits a version in the setting of Grothendieck prestable categories:

\begin{definition}
Let $\Ccal$ be a Grothendieck prestable category. We say that an object $P$ in $\Ccal$ is projective if every map $X \rightarrow P$ in $\Ccal$ which is an epimorphism on $H_0$ admits a section.
\end{definition}

\begin{remark}
Let $\Ccal$ be a Grothendieck prestable category. Then an object $P$ in $\Ccal$ is projective if and only if $\Hom_\Ccal(P, - ): \Ccal \rightarrow \Spc$ preserves geometric realizations. In other words, if and only if $P$ is projective in the sense of \cite{HTT} section 5.5.8.
\end{remark}

We may think about Grothendieck prestable categories generated under colimits by compact projective objects as many object versions of connective ring spectra. In that setting there is a close relation between projective modules over a connective ring spectrum $R$ and projective modules over $\pi_0(R)$ (see \cite{HA} corollary 7.2.2.19). The following proposition is an extension of that relation:

\begin{proposition}\label{proposition projectives vs heart}
Let $\Ccal$ be a Grothendieck prestable category generated under colimits by compact projective objects. Then
\begin{enumerate}[\normalfont (1)]
\item The truncation functor $H_0 : \ccal \rightarrow \ccal^\heartsuit$ sends projective objects to projective objects and compact objects to compact objects.
\item The $0$-truncations of the compact projective objects of $\ccal$ provide a family of compact projective generators for $\ccal^\heartsuit$.
\item The functor $\Ho(H_0): \Ho(\ccal) \rightarrow \Ho(\ccal^\heartsuit) = \ccal^\heartsuit$ induced at the level of homotopy categories restricts to an equivalence between the full subcategories on the projective objects, which in turn restricts to an equivalence on the full subcategories on the compact projective objects.
\end{enumerate}
\end{proposition}
\begin{proof}
We first prove (1). The fact that $H_0$ sends compact objects to compact objects follows directly from the fact that the inclusion $\ccal^\heartsuit \rightarrow \ccal$ preserves filtered colimits. The fact that $H_0$ sends projective objects to projective objects follows from the fact that the inclusion $\ccal^\heartsuit \rightarrow \ccal$ maps epimorphisms to morphisms which induce epimorphisms on $H_0$. 

Item (2) follows directly from (1) together with the fact that $H_0$ is a localization. It  remains to establish (3). We first prove fully faithfulness. Let $X, Y$ be a pair of projective objects of $\ccal$. Then the map $\Hom_{\ccal}(X, Y) \rightarrow \Hom_{\ccal^\heartsuit}(H_0(X), H_0(Y))$ induced by $H_0$ is equivalent to the map $\eta_*: \Hom_{\ccal}(X, Y) \rightarrow \Hom_{\ccal}(X, H_0(Y))$ of composition with the unit $\eta: Y \rightarrow H_0(Y)$. The fact that $X$ is projective and $\eta$ induces an equivalence on $H_0$ implies that $\eta_*$ is an effective epimorphism. Its fiber is given by $\Hom_{\ccal}(X, \tau_{\geq 1}(Y))$ which is connected since $X$ is projective. We conclude that $\eta_*$ induces an equivalence on $\pi_0$, and therefore $\Ho(H_0)$ is fully faithful on the full subcategory on the projective objects.

It remains to prove surjectivity. In other words, we have to show that every (compact) projective object of $\ccal^\heartsuit$ is the image under $H_0$ of a (compact) projective object of $\ccal$. We establish the case of compact projective objects, the proof in the projective case being similar. Let $Y$ be a compact projective object of $\ccal^\heartsuit$.  Applying (2) we may find a compact projective object $X$ in $\ccal$ such that $Y$ is a retract of $H_0(X)$. Let $r: H_0(X) \rightarrow H_0(X)$ be the induced retraction. The fully faithfulness part of (3) allows us to lift $r$ to an idempotent endomorphism $\rho$ of the image of $X$ inside $\Ho(\ccal)$. Let $X'$ be a representative in $\ccal$ of the image of $\rho$. Then $X'$ is a direct summand of $X$ and therefore it is compact projective. The proof finishes by observing that $H_0(X') = \operatorname{Im}(r) = Y$. 
\end{proof}

There is a good theory of tensor products of Grothendieck prestable categories:

\begin{theorem}[\cite{SAG} theorem C.4.2.1]\label{teo tensor product prestable}
Let $\Ccal, \Dcal$ be Grothendieck prestable categories. Then their tensor product $\Ccal \otimes \Dcal$ (formed in $\Pr^L$) is Grothendieck prestable. In particular, the symmetric monoidal structure on the category $\Mod_{\Sp^\cn}(\Pr^L)$ of presentable additive categories and colimit preserving functors restricts to a symmetric monoidal structure on $\Groth_\infty$.
\end{theorem}

In particular, it makes sense to consider commutative algebras in $\Groth_\infty$. We call these symmetric monoidal Grothendieck prestable categories. Note that this terminology leaves implicit the fact that the tensor operation commutes with colimits in each variable.

\begin{definition}
Let $\Mcal$ be a commutative algebra in $\Groth_\infty$. An $\Mcal$-linear Grothendieck prestable category is an object of $\Mod_\Mcal(\Groth_\infty)$.
\end{definition}

In other words, an $\Mcal$-linear Grothendieck prestable category is an $\Mcal$-linear presentable category (in the sense of section \ref{subsection linear cats}) which is in addition a Grothendieck prestable category. In the case when $\Mcal = \Mod^\cn_R$ is the category of connective modules over a connective $E_\infty$-ring $R$ we call these $R$-linear Grothendieck prestable categories.

Since the class of colimits that we have available in $\Groth_\infty$ is relatively restricted, one has to be careful when forming relative tensor products. There is however a class of commutative algebras that admit a well behaved theory of relative tensor products.

\begin{definition}
Let $\Mcal$ be a commutative algebra in $\Groth_\infty$. Assume that $\Mcal$ is generated under colimits by compact projective objects. We say that $\Mcal$ is rigid if compact projective and dualizable objects of $\Mcal$ coincide.
\end{definition}

\begin{example}
Let $R$ be a connective commutative ring spectrum. Then $\Mod_R^\cn$ is rigid.
\end{example}

We fix for the remainder of this section a symmetric monoidal Grothendieck prestable category $\Mcal$ generated under colimits by compact projective objects and rigid. We begin with the observation that $\Mcal$-linear Grothendieck prestable categories are closed under tensor products:

\begin{proposition}
The full subcategory of $\Mod_\Mcal(\Pr^L)$ on the $\Mcal$-linear Grothendieck prestable categories is closed under tensor products. In other words, $\Mod_\Mcal(\Groth_\infty)$ admits a symmetric monoidal structure that makes the inclusion $\Mod_\Mcal(\Groth_\infty) \rightarrow \Mod_\Mcal(\Pr^L)$ symmetric monoidal.
\end{proposition}
\begin{proof}
Completely analogous to the proof of corollary \ref{coro tensor products over A}. One first shows that for any $\Mcal$-module $\Ccal$ in $\Pr^L$ the action map  $\Mcal \otimes \Ccal \rightarrow \Ccal$ admits a colimit preserving right adjoint, imitating the proof of proposition \ref{proposition action has colimit preserving adjoint}. The role of proposition \ref{proposition colimits of right adjointable} is then played by \cite{SAG} remark C.3.5.4.
\end{proof}

We now formulate a Grothendieck prestable version of proposition \ref{prop lex localization}.

\begin{definition}\label{definition generator in groth prestable}
Let $\ccal$ be an $\Mcal$-linear Grothendieck prestable category. We say that an object $G$ in $\ccal$ is an $\Mcal$-generator if for every object $Y$ in $\ccal$ there exists an object $X$ in $\Mcal$ and a morphism $X \otimes G \rightarrow Y$ inducing an epimorphism on $H_0$.
\end{definition}

In the context of Grothendieck prestable categories some attention needs to be paid to the distinction between generators and colimit generators. Unlike the situation with Grothendieck abelian categories, it is possible for an object $G$ in $\ccal$ to be an $\Mcal$-generator in the sense of definition \ref{definition generator in groth prestable} and the family of objects $X \otimes G$ not generate $\ccal$ under colimits (for instance, $0$ is always an $\Mcal$-generator whenever $\ccal$ is stable). As shown in \cite{SAG} theorem 2.1.6, the distinction disappears when $\ccal$ is assumed to be separated:

\begin{definition}
Let $\Ccal$ be a Grothendieck prestable category. We say that $\Ccal$ is separated if it contains no nonzero $\infty$-connective objects.\footnote{An object of $\Ccal$ is $\infty$-connective if all its homology objects vanish.}
\end{definition}

\begin{proposition}\label{prop lex localization prestable}
Let $\Ccal$ be a separated $\Mcal$-linear Grothendieck prestable category and let $G$ be an $\Mcal$-generator for $\Ccal$. Let $A$ be the opposite of the algebra of endomorphisms of $G$ associated to the action of $\Mcal$ on $\Ccal$. Then the functor
\[
G \otimes_A - : \LMod_A(\Mcal) \rightarrow \Ccal 
\]
is an $\Mcal$-linear left exact localization. Furthermore, it is an equivalence if and only if $G$ is compact projective.
\end{proposition}
\begin{proof}
Completely analogous to the proof of proposition \ref{prop lex localization}, where the role of the classical many object Gabriel-Popescu theorem  is played by \cite{SAG} theorem C.2.1.6.
\end{proof}

We now turn to a discussion of flatness in the context of $\Mcal$-linear Grothendieck prestable categories.

\begin{definition}
Let $\Ccal$ be an $\Mcal$-linear Grothendieck prestable category. We say that an object $X$ in $\Ccal$ is flat over $\mcal$ if the functor $- \otimes X : \Mcal \rightarrow \Ccal$ is left exact. In cases when the base $\mcal$ is clear from the context we simply say that $X$ is flat.
\end{definition}

\begin{example}
Let $R$ be a connective $E_\infty$-ring and let $\Ccal$ be an $R$-linear Grothendieck prestable category. It follows from \cite{SAG} proposition C.3.2.1 that an object $X$ in $\Ccal$ is flat if and only if $M \otimes X$ is $0$-truncated for all $0$-truncated $R$-modules $M$. Since every $0$-truncated $R$-module is a filtered colimit of cyclic $\pi_0(R)$-modules, we see that $X$ is flat if and only if $M \otimes X$ is $0$-truncated for every cyclic $\pi_0(R)$-module $M$.
\end{example}

\begin{proposition}\label{prop lazard prestable}
Let $X$ be an object of $\Mcal$. Then $X$ is flat if and only if it is a filtered colimit of compact projective objects.
\end{proposition}
\begin{proof}
Analogous to the proof of proposition \ref{prop lazard classico}.
\end{proof}

\begin{proposition}\label{prop tensoring left exact prestable}
Let $f: \Ccal \rightarrow \Ccal'$ and $g: \Dcal \rightarrow \Dcal'$ be  morphisms in $\Mod_{\Mcal}(\Groth_\infty)$. If $f$ and $g$ are left exact then $f \otimes_{\Mcal} g : \Ccal \otimes_\Mcal \Dcal \rightarrow \Ccal' \otimes_\Mcal \Dcal'$ is left exact.
\end{proposition}
\begin{proof}
Analogous to the proof of proposition \ref{prop tensoring left exact}.
\end{proof}
\begin{corollary}\label{coro tensor flats prestable}
Let $\Ccal, \Dcal$ be $\Mcal$-linear Grothendieck prestable categories, and let $X, Y$ be flat objects of $\Ccal$ and $\Dcal$ respectively. Then the object $X \otimes Y$ in $\Ccal \otimes_\Mcal \Dcal$ is flat.
\end{corollary}
\begin{proof}
Analogous to the proof of corollary \ref{coro tensor flats}.
\end{proof}

In the presence of flatness, projectivity of objects may be checked after passing to $H_0$:

\begin{proposition}\label{prop check compact projective on heart}
Let $\ccal$ be an $\Mcal$-linear Grothendieck prestable category. Assume that $\ccal$ is generated under colimits by compact projective objects and that compact projective objects in $\ccal$ are flat. Then an object $X$ in $\ccal$ is projective if and only if it is flat and $H_0(X)$ is projective in $\ccal^\heartsuit$.
\end{proposition}
\begin{proof}
Since every projective object is a retract of a direct sum of compact projective objects, and flat objects are closed under retracts and direct sums, we see that every projective object of $\ccal$ is flat. The fact that $H_0$ sends projective objects to projective objects was already observed in proposition \ref{proposition projectives vs heart}. This finishes the proof of the only if direction.

 Assume now that $X$ is flat and $H_0(X)$ is projective. Applying proposition \ref{proposition projectives vs heart} we may find a projective object $X'$ in $\ccal$ and an isomorphism $H_0(X') = H_0(X)$. The fact that $X'$ is projective allows us to lift this isomorphism to a morphism $f: X' \rightarrow X$.   We claim that $f$ is an isomorphism. To do so it suffices to prove that $f \otimes 1_{\mcal}$ is an isomorphism. Since both $X$ and $X'$ are flat and $\ccal$ is separated we may reduce to proving that $H_0(f \otimes H_n(1_{\mcal}))$ is an isomorphism for all $n \geq 0$. This agrees with $H_0(H_0(f) \otimes H_n(1_\mcal))$, which is an isomorphism by virtue of the fact that $H_0(f)$ is an isomorphism.
\end{proof}

We now discuss the operation of passage to derived categories for linear Grothendieck abelian categories. Fix for the remainder of this section a symmetric monoidal Grothendieck abelian category $\acal$, rigid and generated by compact projective objects.

\begin{construction}\label{construction smon structure on der 1}
Let $\Acal^\cp$ be the full subcategory of $\Acal$ on the compact projective objects and equip $\acal^\cp$ with the symmetric monoidal structure restricted from $\acal$. Note that   $\der(\Acal)_{\geq 0}$ is  obtained by freely adjoining colimits to $\Acal^\cp$. We equip $\der(\Acal)_{\geq 0}$  with the unique colimit preserving extension of the existing symmetric monoidal structure on $\acal^{cp}$. Since $\der(\acal)$ is the stabilization of $\der(\acal)$, we may further extend our symmetric monoidal structure uniquely to a symmetric monoidal structure compatible with colimits on $\der(\acal)$.
\end{construction}

\begin{remark}
Construction \ref{construction smon structure on der 1} makes $\der(\Acal)_{\geq 0}$ into a rigid commutative algebra in $\Groth_\infty$. It is in fact the unique way to equip  $\der(\Acal)_{\geq 0}$ with a rigid commutative algebra structure making the truncation functor $\der(\Acal)_{\geq 0} \rightarrow \acal$ symmetric monoidal.
\end{remark}

A variant of construction \ref{construction smon structure on der 1} allows us to give the connective derived category of an $\acal$-linear Grothendieck abelian category the structure of a  $\der(\acal)_{\geq 0}$-linear Grothendieck prestable category:

\begin{construction}\label{construction smon structure on der 2}
Let $\ccal$ be an $\acal$-linear Grothendieck abelian category. We view the $\acal$-linear structure as a monoidal finite coproduct preserving functor $f: \acal^\cp \rightarrow \Funct^L_{\lex}(\ccal, \Ccal)$ where the target is the category of left exact colimit preserving endofunctors of $\ccal$. Passing to derived functors provides a monoidal equivalence 
\[
\der(-): \Funct^L_{\lex}(\ccal, \Ccal) \rightarrow \Funct^L_{\lex}(\der(\Ccal)_{\geq 0}, \der(\Ccal)_{\geq 0}).
\]
 We equip $\der(\ccal)_{\geq 0}$ with the $\der(\acal)_{\geq 0}$-linear structure arising from the  monoidal finite coproduct preserving functor $\der(-) \circ f : \acal^\cp \rightarrow \Funct^L_{\lex}(\der(\Ccal)_{\geq 0}, \der(\Ccal)_{\geq 0})$. We equip $\der(\ccal) = \Sp(\der(\ccal)_{\geq 0})$ with the induced $\der(\acal) = \Sp(\der(\acal)_{\geq 0})$-linear structure.
\end{construction}

\begin{remark}
Let $\Ccal$ be an $\acal$-linear Grothendieck abelian category. Then the $\der(\acal)_{\geq 0}$-linear structure on $\der(\ccal)_{\geq 0}$ from construction \ref{construction smon structure on der 2} is the unique such structure making the truncation functor $\der(\ccal)_{\geq 0} \rightarrow \ccal$ into a $\der(\ccal)_{\geq 0}$-linear functor.
\end{remark}

\begin{notation}
The inclusion $\acal \rightarrow \der(\acal)$ is generally not symmetric monoidal. When we wish to emphasize the distinction between both symmetric monoidal structures we will write $\otimes^L$ for the tensor product in $\der(\acal) $, and $\otimes$ for the tensor product in $\acal$. If it is clear from the context which operation is being used, we will simply write $\otimes$ instead of $\otimes^L$. The same considerations apply to the case of the action of $\der(\acal)$ on the   derived category of an $\acal$-linear Grothendieck abelian category.
\end{notation}

The following proposition relates the notion of flatness  of objects in $\acal$-linear Grothendieck abelian categories from section \ref{subsection abelian} with the notion studied in this section.

\begin{proposition}\label{prop flat in C vs DC}
Let $\ccal$ be a separated $\der(\acal)_{\geq 0}$-linear Grothendieck prestable category and let $X$ be an object of $\ccal$. The following are equivalent:
\begin{enumerate}[\normalfont(1)]
\item  $X$ is flat over $\der(\acal)_{\geq 0}$.
\item $X$ is $0$-truncated and $H_1(Y \otimes X) = 0$ for all $Y$ in $\acal$.
\item  $X$ is $0$-truncated and flat over $\acal$ (as an object of $\ccal^\heartsuit$).
\end{enumerate}
\end{proposition}
\begin{proof}
We first show that (1) implies (2). Since $X$ is flat and the unit $1_\acal$ is $0$-truncated, we have that $X = 1_\acal \otimes X$ is $0$-truncated. Furthermore, appealing once again to the flatness of $X$ we have that for every $Y$ in $\acal$ the object $Y \otimes X$ is $0$-truncated, and hence $H_1(Y \otimes X) = 0$, as desired. To see that (2) implies (3) we must show that if  $i: Z \rightarrow Z'$ is a monomorphism in $\Acal$ then $H_0(i \otimes X)$ is a monomorphism in $\ccal$. Indeed, the kernel of $H_0(i \otimes X)$ receives an epimorphism from $H_1(\coker(i)\otimes X)$, which vanishes.

It remains to prove that (3) implies (1). Let $Y$ be a $0$-truncated object of $\der(\acal)_{\geq 0}$. Pick a resolution $Y_\bullet$ of $Y$ by compact projective objects. Then $Y \otimes X$ is the realization of $Y_\bullet \otimes X$. Since $Y_\bullet$ is levelwise compact projective we have that $Y_\bullet \otimes X$ is a diagram of $0$-truncated objects. The fact that $X$ is a flat object of $\ccal^\heartsuit$ now implies that $Y_\bullet \otimes X$ is a resolution of $H_0(Y \otimes X)$. Since $\ccal$ is separated we have that $Y \otimes X = H_0(Y \otimes X)$ is $0$-truncated. Since $Y$ was arbitrary we conclude that $X$ is flat, as desired.
\end{proof}

\begin{remark}
Let $\ccal$ be an $\acal$-linear Grothendieck abelian category. For each $n \geq 0$ we denote by $\Tor_n(-,-): \acal \times \ccal \rightarrow \ccal$ the composite functor
\[
\acal \times \ccal = \der(\acal)^\heartsuit \times \der(\ccal)^\heartsuit \hookrightarrow \der(\acal)_{\geq 0} \times \der(\ccal)_{\geq 0} \xrightarrow{-\otimes-} \der(\ccal)_{\geq 0} \xrightarrow{H_n(-)} \ccal.
\]
Specializing proposition \ref{prop flat in C vs DC} we see that an object $X$ in $\der(\ccal)_{\geq 0}$ is flat over $\der(\acal)_{\geq 0}$ if and only if it is $0$-truncated and $\Tor_1(Y, X) = 0$ for all $Y$ in $\acal$.
\end{remark}


\ifx\inmain\undefined
\bibliographystyle{myamsalpha2}
\bibliography{References}
\fi


\section{Compact assembly, duality, and filtered colimits of categories}\label{section compactly assembled}

In \cite{JJcont}, Johnstone and Joyal introduced the notion of a continuous $(1,1)$-category as a categorification of the more classical notion of continuous poset, with the purposes of characterizing exponentiable Grothendieck $(1,1)$-topoi. In \cite{SAG}, an $(\infty,1)$-categorical generalization of this notion was developed, under the name of compactly assembled categories. In the appendix D, Lurie proves a fundamental result that shows that a presentable stable category is dualizable if and only if it is compactly assembled. 

We begin this section in \ref{subsection compactly assembled} by reviewing Lurie's results, and formulating a variant that we call $n$-strong compact assembly, which plays a similar role in classifying dualizable presentable additive $(n,1)$-categories for each $1 \leq n \leq \infty$. This is of fundamental importance in our paper: together with   proposition \ref{proposition dualizable is presentable}, it implies that dualizable cocomplete  additive  $(1,1)$-categories (resp. $(\infty,1)$-categories) are automatically Grothendieck abelian (resp. separated Grothendieck prestable) and have exact products.

In the same way that compact assembly is a weakening of the notion of generation by compact objects, $n$-strong compact assembly is a weakening of the notion of generation by $n$-strongly compact objects, where an object is said to be $n$-strongly compact if the functor it corepresents commutes with $(n,1)$-categorical sifted colimits. For categories which are compactly generated, a natural class of functors to consider between them are those which preserve compact objects. In the case of compactly assembled categories, this condition may be weakened to yield the notion of compact functor. We explore this in \ref{subsection colimits}, together with the related notion of $n$-strongly compact functor. 

Finally, in \ref{subsection lifting} we prove the main result of this section, that guarantees that the functor that assigns to each ($n$-strongly) compactly assembled category its full subcategory on the ($n$-strongly) compact objects preserves filtered colimits of diagrams with ($n$-strongly) compact transitions.  This will be of fundamental importance in section \ref{section G rings}, allowing us deduce the general case of theorem \ref{theorem principal 2 introduction} from the case of complete local Noetherian rings.

\subsection{Compact assembly and strong compact assembly}\label{subsection compactly assembled}

We begin with a review of the notion of compactly assembled category.

\begin{notation}
We denote by $\catomega$ the subcategory of $\cathat$ on those large categories admitting small filtered colimits and functors which preserve small filtered colimits. We will denote by $\Ind:  \cathat \rightarrow \catomega$ the left adjoint to the forgetful functor. In other words, $\Ind$ is the functor that freely adjoins small filtered colimits.
\end{notation}

\begin{definition}\label{def compactly assembled}
Let $\Ccal$ be a category admitting small filtered colimits. We say that $\Ccal$ is compactly assembled if the functor $\Ind(\Ccal) \rightarrow \Ccal$ which ind-extends the identity on $\Ccal$, admits a left adjoint.
\end{definition}

\begin{remark}
The notion of compactly assembled category was introduced, in the $(1,1)$-categorical context, in \cite{JJcont} under the name of continuous category. The $\infty$-categorical notion is explored in \cite{SAG} section 21.1.2. Note that our definition is slightly different from that of \cite{SAG} in that we do not require accessibility of $\ccal$. This allows us to treat $\ccal$ and $\Ind(\ccal)$ on equal footing: if $\ccal$ is a category with small filtered colimits then $\Ind(\ccal)$ is compactly assembled according to definition \ref{def compactly assembled}, but usually not accessible.
\end{remark}

\begin{remark}\label{remark compactly assembled presentable}
Let $\ccal$ be a presentable category, and assume that $\ccal$ is compactly assembled. Let $p: \Ind(\ccal) \rightarrow \ccal$ be the projection and $i: \ccal \rightarrow \Ind(\ccal)$ be its left adjoint. For each regular cardinal $\kappa$ let $\ccal^\kappa$ be the full subcategory of $\ccal$ on the $\kappa$-compact objects. Then $\Ind(\ccal) = \colim \Ind(\ccal^\kappa)$, and since $\ccal$ is presentable we have that $i$ factors through $\Ind(\ccal^\kappa)$ for some $\kappa$. The restriction of $p$ to $\Ind(\ccal^\kappa)$ is a limit preserving localization between presentable categories.
\end{remark}

The following is the content of \cite{SAG} theorem 21.1.2.10 (with appropriate modifications to remove the accessibility conditions):

\begin{theorem}\label{theorem comp assembled iff retract}
Let $\Ccal$ be a category admitting small filtered colimits. Then $\Ccal$ is compactly assembled if and only if it is a retract in $\catomega$ of a category which is generated under small filtered colimits by compact objects.
\end{theorem}

The following result explains the relevance of the notion of compact assembly for our purposes:

\begin{theorem}[\cite{SAG} proposition D.7.3.1 and corollary D.7.7.3] \label{theorem dualizable iff comp assembled}
Let $\Mcal$ be a compactly generated presentable symmetric monoidal stable category. Assume that compact and dualizable objects in $\mcal$ coincide (in other words, $\Mcal$ is rigid). Then an object of $\Mod_{\Mcal}(\Pr^L)$ is dualizable if and only if the underlying category is compactly assembled.
\end{theorem}

We now discuss a variant of the notion of compact assembly, where filtered colimits are replaced by sifted colimits. We fix throughout a constant $1 \leq n \leq \infty$. We begin with a preliminary discussion on the notion of sifted colimits in $(n,1)$-categories.

\begin{proposition}\label{prop equivalences initial}
Let $F: \Ical \rightarrow \Jcal$ be a functor of $(n,1)$-categories. The following are equivalent:
\begin{enumerate}[\normalfont (1)]
\item For every $(n,1)$-category $\Ccal$ and every limit diagram $\Jcal^\lhd \rightarrow \Ccal$, the induced diagram $\Ical^\lhd \rightarrow \Ccal$ is a limit diagram.
\item For every object $X$ in $\Jcal$, the geometric realization of the category $\Ical \times_{\Jcal} \Jcal_{/X}$ is $n$-connective.
\end{enumerate}
\end{proposition}
\begin{proof}
Assume first that (1) holds and let $X$ be an object in $\Jcal$. Then the projection $\Ical \times_{\Jcal} \Jcal_{/X} \rightarrow \Ical$ is a right fibration which classifies the functor $\Hom_{\Jcal}(-, X)|_{\Ical^\op}$. The geometric realization of $\Ical \times_{\Jcal} \Jcal_{/X}$ is therefore equivalent to the colimit of $\Hom_{\Jcal}(-, X)|_{\Ical^\op}: \Ical^\op \rightarrow \Spc$, which in turn agrees with the limit of the functor $\Hom_{\Jcal}(-, X)|_{\Ical}: \Ical \rightarrow \Spc^\op$. To prove (2) it suffices to show that the limit of the functor $\Hom_{\Jcal}(-, X)|_{\Ical}: \Ical \rightarrow (\Spc_{\leq n-1})^\op$ is the terminal space. Applying (1) we reduce to showing that the limit of the functor $\Hom_{\Jcal}(-, X): \Jcal \rightarrow (\Spc_{\leq n-1})^\op$ is the terminal space. As before, we may identify this with the $(n-1)$-truncation of the geometric realization of the right fibration classifying the functor $\Hom_{\Jcal}(-, X): \Jcal^\op \rightarrow \Spc$. The fact that (2) holds now follows from the fact that this right fibration admits a terminal object and is therefore contractible.

Assume now that (2) holds. Since the inclusion $\ccal \rightarrow \Pcal(\ccal)$ of $\ccal$ into its free cocompletion preserves all limits we may after replacing $\ccal$ with $\Pcal(\ccal)_{\leq n-1}$ assume that $\ccal$ is presentable. Let $\Dcal = (\Pcal(\ccal^\op)_{\leq n-1})^\op$ be the $(n,1)$-category obtained from $\ccal$ by freely adjoining small limits. Since the projection $\dcal \rightarrow \ccal$ preserves limits, it is enough to show that restriction along $F$ preserves limits valued in $\dcal$. Note that $\dcal$ is equivalent to the opposite of the category of accessible functors $\ccal \rightarrow \Spc_{\leq n-1}$, and in particular it has a limit preserving embedding into $\Funct(\ccal^\op, (\Spc_{\leq n-1})^\op)$. It is therefore enough to show that restriction along $F$ preserves limits valued in $\Funct(\ccal^\op, \Spc_{\leq n-1}^\op)$. Since the evaluation functors create limits we may now reduce to the case $\ccal = \Spc_{\leq n-1}^\op$. In this case every functor $G: \Jcal \rightarrow \Spc_{\leq n-1}^\op$ is the limit of corepresentable functors. Since the property that the limit of $G$ is preserved by restriction along $F$ is preserved by limits in $G$, it is enough to consider the case where $G$ is corepresentable. The arguments from the first paragraph of the proof show that this case is implied by (2).
\end{proof}

\begin{definition}
We say that a functor $F: \Ical \rightarrow \Jcal$ of $(n,1)$-categories is initial if it satisfies the equivalent conditions of proposition \ref{prop equivalences initial}. We say that an $(n,1)$-category $\Ical$ is $n$-sifted if the diagonal $\Ical^\op \rightarrow \Ical^\op \times  \Ical^\op$ is initial.
\end{definition}

\begin{definition}
Let $\Ccal$ be an $(n,1)$-category admitting small colimits indexed by $n$-sifted $(n,1)$-categories and let $X$ be an object in $\Ccal$. We say that $X$ is $n$-strongly compact if the functor $\Hom_\Ccal(X, -): \Ccal \rightarrow \Spc_{\leq n-1}$ preserves small $n$-sifted colimits.
\end{definition}

\begin{remark}
In the case $n = \infty$, our notion of $\infty$-strongly compact object agrees with the notion of compact projective object from \cite{HTT} section 5.5.8.
\end{remark}

\begin{notation}
Let $\cathat_n$ be the category of large $(n,1)$-categories. We denote by $\cathat^\Sigma_n$ the subcategory of $\cathat_n$ on the large $(n,1)$-categories admitting small $n$-sifted colimits, and functors which preserve small sifted colimits. The forgetful functor $\cathat_n^\Sigma \rightarrow \cathat_n$ admits a left adjoint, which we denote by $\Pcal_\Sigma^n$. In other words, $\Pcal_\Sigma^n$ is the functor that freely adjoins small $n$-sifted colimits to an $(n,1)$-category. If $n = \infty$ we denote this simply by $\Pcal_\Sigma$.
\end{notation}

\begin{remark}
For each $(n,1)$-category $\Ccal$ we have a Yoneda embedding $\Ccal \rightarrow \Pcal_\Sigma^n(\Ccal)$. Its image consists of $n$-strongly compact objects of $\Pcal_\Sigma^n(\Ccal)$, and generates it under $n$-sifted colimits. These properties characterize the categories of the form $\Pcal_{\Sigma}^n(\Ccal)$: if $\Dcal$ is an $(n,1)$-category with $n$-sifted colimits, generated under $n$-sifted colimits by $n$-strongly compact objects, then $\Dcal = \Pcal_{\Sigma}^n(\Ccal)$ where $\Ccal$ is the full subcategory of $\Dcal$ on the $n$-strongly compact objects.
\end{remark}

With the notion of $n$-sifted colimits at hand, we may define $n$-strong compact assembly in a way which is completely analogous to compact assembly.

\begin{definition}
Let $\Ccal$ be an $(n,1)$-category admitting small $n$-sifted colimits. We say that $\Ccal$ is $n$-strongly compactly assembled if the unique $n$-sifted colimit preserving extension  $\Pcal_{\Sigma}^n(\Ccal) \rightarrow \Ccal$  of the identity on $\Ccal$ admits a left adjoint.
\end{definition}

\begin{remark}\label{remark presentable strongly compactly assembled}
As in remark \ref{remark compactly assembled presentable}, one shows that if $\ccal$ is an $n$-strongly compactly assembled presentable category then $\ccal$ is a limit preserving localization of $\Pcal^n_\Sigma(\ccal^\kappa)$ for some regular cardinal $\kappa$.
\end{remark}

In the additive context, $n$-strongly compactly assembled categories have particularly good behavior:

\begin{proposition}\label{prop properties strongly compactly assembled} 
\hspace{1cm}
\begin{enumerate}[\normalfont (1)]
\item Let $\ccal$ be a $1$-strongly compactly assembled presentable additive  $(1,1)$-category. Then $\ccal$ is a Grothendieck abelian category with exact products.
\item Let $\ccal$ be an $\infty$-strongly compactly assembled presentable additive category. Then $\ccal$ is a separated Grothendieck prestable category such that products in $\Sp(\ccal)$ are exact.
\end{enumerate}
\end{proposition}
\begin{proof}
The property of being a Grothendieck abelian category with exact products (resp. a separated Grothendieck prestable category with exact products) is preserved under passage to limit preserving accessible localizations. The proposition now follows from remark \ref{remark presentable strongly compactly assembled}. 
\end{proof}

The following is a variant of theorem \ref{theorem comp assembled iff retract}:

\begin{proposition}\label{prop characterize strongly comp assembled}
Let $\Ccal$ be an $(n,1)$-category admitting $n$-sifted colimits. Then $\Ccal$ is $n$-strongly compactly assembled if and only if it is a retract in $\cathat_n^\Sigma$ of a category which is generated under $n$-sifted colimits by $n$-strongly compact objects.
\end{proposition}
\begin{proof}
If $\ccal$ is $n$-strongly compactly assembled then $\ccal$ is a retract in $\cathat_n^\Sigma$ of $\Pcal_\Sigma^n(\ccal)$. Conversely, assume that $\ccal$ is a retract in $\cathat_n^\Sigma$ of a category $\dcal$ which is generated under $n$-sifted colimits by $n$-strongly compact objects. Then the projection $p_\Ccal: \Pcal^n_\Sigma(\ccal) \rightarrow \ccal$ is a retract of the projection $p_\Dcal: \Pcal^n_\Sigma(\Dcal) \rightarrow \Dcal$. The functor $p_\Dcal$ admits a left adjoint which sends each $n$-strongly compact object to its image under the Yoneda embedding. The fact that $p_\Ccal$ has a left adjoint now follows from \cite{SAG} lemma 21.1.2.14. 
\end{proof}

We also have the following variant of theorem \ref{theorem dualizable iff comp assembled}:

\begin{theorem}\label{theorem dualizables over truncated connective}
Let $\Mcal$ be a presentable symmetric monoidal additive $(n,1)$-category. Assume that $\Mcal$ is generated under colimits by $n$-strongly compact objects, and that $n$-strongly compact objects and dualizable objects coincide in $\Mcal$ (in other words, $\Mcal$ is rigid). Then an object of $\Mod_{\Mcal}(\Pr^L)$ is dualizable if and only if the underlying $(n,1)$-category is $n$-strongly compactly assembled.
\end{theorem}

\begin{remark}\label{remark description duality comp}
Let $\ccal$ be a presentable additive $(n,1)$-category  generated under colimits by $n$-strongly compact objects. Denote by $\ccal^\Sigma$ the full subcategory of $\ccal$ on the $n$-strongly compact objects. Then there is a colimit preserving functor
\[
\ccal \otimes \Pcal_\Sigma^n((\ccal^\Sigma)^\op) \rightarrow \Sp_{\leq n-1}
\]
with target the category of $(n-1)$-truncated spectra, defined by the property that its restriction to $\ccal \times (\ccal^\Sigma)^\op$ recovers the restriction of the functor $\Hom^\enh_\ccal(-,-): \ccal^\op \times \ccal \rightarrow \Sp_{\leq n-1}$. The above functor is  the counit of a duality in $\Mod_{\Sp_{\leq n-1}}(\Pr^L)$. In particular, we see that the dual to $\ccal$ is also generated under colimits by $n$-strongly compact objects.
\end{remark}

\begin{proof}[Proof of theorem \ref{theorem dualizables over truncated connective}]
The assumptions on $\Mcal$ imply that the tensoring functor $\Mcal \otimes \Mcal \rightarrow \Mcal$ admits an $\Mcal$-linear colimit preserving right adjoint, and its image generates $\Mcal$ under colimits. It follows from this that $\Mcal$ is equivalent to the category of modules over the endomorphism algebra of $1_\Mcal$ in $\Mcal \otimes \Mcal$, and therefore it is dualizable as a module over $\Mcal \otimes \Mcal$.  Combined with remark \ref{remark description duality comp} we see that $\Mcal$ is smooth and proper as an algebra in $\Mod_{\Sp_{\leq n-1}}(\Pr^L)$, so that an $\Mcal$-module is dualizable if and only if its image in $\Mod_{\Sp_{\leq n-1}}(\Pr^L)$ is dualizable.

It remains to show that an object $\ccal$ in  $\Mod_{\Sp_{\leq n-1}}(\Pr^L)$ is dualizable if and only if it is $n$-strongly compactly assembled. Assume first that $\ccal$ is $n$-strongly compactly assembled. It follows from remark \ref{remark presentable strongly compactly assembled} that $\ccal$ is a retract in $\Mod_{\Sp_{\leq n-1}}(\Pr^L)$ of a category which is generated under colimits by $n$-strongly compact objects. The fact that $\ccal$ is dualizable now follows from remark \ref{remark description duality comp} together with the fact that $\Mod_{\Sp_{\leq n-1}}(\Pr^L)$ is idempotent complete.

Assume now that $\ccal$ is a dualizable object of $\Mod_{\Sp_{\leq n-1}}(\Pr^L)$. Pick a regular cardinal $\kappa$ such that $\ccal$ is $\kappa$-accessible. Then the colimit preserving functor $p: \Pcal_\Sigma^n(\ccal^\kappa) \rightarrow \ccal$ that extends the inclusion $\ccal^\kappa \rightarrow \ccal$ is a localization. It follows that $\id_{\ccal^\vee} \otimes p$ is also a localization, which implies in particular that the functor $\Funct^L(\ccal, \Pcal_\Sigma^n(\ccal^\kappa)) \rightarrow \Funct^L(\ccal, \ccal)$ of composition with $p$ is surjective. Therefore the identity on $\ccal$ admits a colimit preserving lift along $p$, which means that $\ccal$ is a retract in $\Pr^L$ of $\Pcal_\Sigma^n(\ccal^\kappa)$. The fact that $\ccal$ is $n$-strongly compactly assembled now follows from proposition \ref{prop characterize strongly comp assembled}.
\end{proof}

\begin{corollary}\label{coro properties dualizable 11}
Let $\acal$ be a symmetric monoidal Grothendieck abelian category, rigid and generated by compact projective objects. Let $\ccal$ be a dualizable   $\acal$-linear cocomplete category. Then $\ccal$ is a Grothendieck abelian category with exact products.
\end{corollary}
\begin{proof}
Combine propositions \ref{proposition dualizable is presentable} and \ref{prop properties strongly compactly assembled} with theorem \ref{theorem dualizables over truncated connective}.
\end{proof}

\begin{corollary}\label{coro properties dualizable infty}
Let $\Mcal$ be a symmetric monoidal Grothendieck prestable category, rigid and generated under colimits by compact projective objects. Let $\ccal$ be a dualizable   $\Mcal$-linear cocomplete category. Then $\ccal$ is a separated Grothendieck prestable category and products are t-exact in $\Sp(\ccal)$.
\end{corollary}
\begin{proof}
Combine propositions \ref{proposition dualizable is presentable} and \ref{prop properties strongly compactly assembled} with theorem \ref{theorem dualizables over truncated connective}.
\end{proof}

\begin{remark}\label{remark description duality over M}
Let $\Mcal$ be as in theorem \ref{theorem dualizables over truncated connective}. Let $\ccal$ be a dualizable $\Mcal$-module and let $\epsilon: \Ccal \otimes_\Mcal \ccal^\vee \rightarrow \Mcal$ be the counit of the duality. Then the composite map
\[
\ccal \otimes \ccal^\vee \rightarrow \ccal \otimes_{\Mcal} \ccal^\vee \xrightarrow{\epsilon} \Mcal \xrightarrow{\Hom^\enh_{\Mcal}(1_\Mcal, -)} \Sp_{\leq n-1}
\]
is the counit of a duality between $\ccal$ and $\ccal^\vee$ in the symmetric monoidal category of presentable additive $(n,1)$-categories. 

Assume now that $\ccal$ is generated under colimits by $n$-strongly compact objects. Then combining with remark \ref{remark description duality comp} we obtain an identification $\ccal^\vee = \Pcal_\Sigma^n((\ccal^\Sigma)^\op)$, such that $\epsilon$ is a colimit preserving $\Mcal$-linear functor with the property that its restriction to $\ccal \times  (\ccal^\Sigma)^\op$ recovers the restriction of the functor $\Hom^\enh_\ccal(-,-): \ccal^\op \times \ccal \rightarrow \Mcal$.
\end{remark}

The following proposition shows that passage to derived categories preserves the property of being strongly compactly assembled:

 \begin{proposition}\label{lemma derive strong compact assembly}
 Let $\ccal$ be a $1$-strongly compactly assembled Grothendieck abelian category. Then $\der(\ccal)_{\geq 0}$ is $\infty$-strongly compactly assembled.
 \end{proposition}
 \begin{proof}
 By remark \ref{remark presentable strongly compactly assembled}, we may pick a regular cardinal $\kappa$ such that the canonical functor $p : \Pcal_\Sigma^1(\ccal^\kappa) \rightarrow \ccal$ is a limit preserving localization. In particular, passing to derived categories it induces a left exact colimit preserving functor $p': \Pcal_\Sigma(\ccal^\kappa) \rightarrow \der(\ccal)_{\geq 0}$. It follows from the universal property of derived categories that $p'$ is the universal left exact colimit preserving functor from $\Pcal_\Sigma(\ccal^\kappa)$ into a separated  Grothendieck prestable category such that $p'(X) = 0$ for all $X$ in $\Pcal_\Sigma^1(\ccal^\kappa) $ such that $p(X) = 0$. Therefore $p'$ exhibits $\der(\ccal)_{\geq 0}$ as the left exact localization of $\Pcal_\Sigma(\ccal^\kappa)$ at the localizing class given by those objects $X$ such that $p(H_n(X)) = 0$ for all $n$ (see \cite{SAG} proposition C.5.2.7). 
 
 To prove that $\ccal$ is $\infty$-strongly compactly assembled it will suffice to show that $p'$ admits a left adjoint. By the adjoint functor theorem, it is enough to show that $p'$ preserves limits. Since $p'$ is already known to preserve finite limits, we may reduce to showing that $p'$ preserves products. Let $X_\alpha$ be a set-indexed family of objects of $\Pcal_\Sigma(\ccal^\kappa)$. We wish to show that $p'(\prod X_\alpha) = \prod p'(X_\alpha)$. By proposition \ref{prop properties strongly compactly assembled} it suffices to show that $H_n(p'(\prod X_\alpha)) = \prod H_n(p'(X_\alpha))$ for all $n \geq 0$. Since $p'$ is left exact and colimit preserving we may reduce to proving that $p'(H_n(\prod X_\alpha)) = \prod p'(H_n(X_\alpha))$. This follows from the fact that $p$ preserves products and that products are exact in $\Pcal_\Sigma(\ccal^\kappa)$.
  \end{proof}

\begin{remark}
Let $\ccal$ be an $\infty$-strongly compactly assembled Grothendieck prestable category. It follows from remark \ref{remark presentable strongly compactly assembled} that $\ccal^\heartsuit$ is a limit preserving localization of a Grothendieck abelian category generated by compact projective objects. Applying proposition \ref{prop characterize strongly comp assembled} we see that $\ccal^\heartsuit$ is $1$-strongly compactly assembled. Similarly, $\Sp(\ccal)$ is  a limit preserving localization of a compactly generated presentable stable category, and therefore $\Sp(\ccal)$ is compactly assembled.
\end{remark}

We finish by discussing two families of examples of categories which are (strongly) compactly assembled but not (strongly) compactly generated.

\begin{example}\label{example almost modules}
Let $R$ be a valuation ring with value group $\ZZ[1/2]$ and let $\mathfrak{m}$ be the maximal ideal of $R$. Then $\mathfrak{m}$ is idempotent and flat. It follows that the functor $- \otimes \mathfrak{m} : \Mod_R^\heartsuit \rightarrow \Mod_R^\heartsuit$ is a colimit preserving left exact colocalization. Its image is the Grothendieck abelian category $\operatorname{aMod}_{R}^\heartsuit$ of almost $R$-modules from \cite{GabberRamero}. This is a retract of $\Mod^\heartsuit_R$ in $\Pr^L$, and in particular it is $1$-strongly compactly generated and a dualizable Grothendieck abelian category. Any projective object of  $\operatorname{aMod}_R^\heartsuit$ defines a projective $R$-module $P$ with the property that $\mathfrak{m} \otimes_R P = P$. Since any nonzero projective $R$-module contains a copy of $R$ and $\mathfrak{m} \otimes_R R \neq R$, we have that $P$ is necessarily zero. In particular we have that  $\operatorname{aMod}_R^\heartsuit$ is not generated by $1$-strongly compact objects.

The category $\operatorname{aMod}_R^\heartsuit$ inherits an $R$-linear structure from $\operatorname{Mod}_R^\heartsuit$. If $R$ is an algebra over a field $k$ then $\operatorname{aMod}_R^\heartsuit$ gives an example of a nonzero dualizable $k$-linear Grothendieck abelian category without nonzero projective objects.

This example extends to the derived context. Namely, $- \otimes \mathfrak{m}: \Mod_R \rightarrow \Mod_R$ is a colimit preserving t-exact colocalization. Its image is the presentable stable category $\operatorname{aMod}_R$ of almost $R$-modules. It inherits a t-structure from $\Mod_R$, and is equivalent to the derived category of $\operatorname{aMod}_R^\heartsuit$. For each $1 \leq n \leq \infty$ the full subcategory $(\operatorname{aMod}^\cn_R)_{\leq n-1}$  on the $(n-1)$-truncated connective objects is $n$-strongly compactly assembled and has no nontrivial projective objects.

The presentable stable category $\operatorname{aMod}_R$ is compactly assembled by the same reasoning as above. We claim that it contains no nontrivial compact objects. This is equivalent to the assertion that if $M$ is a compact $R$-module such that $\mathfrak{m} \otimes_R M = M$ then $M = 0$. Assume for the sake of contradiction that $M \neq 0$, and let $n$ be such that $H_n(M) \neq 0$. Since $R$ is coherent we have that $H_n(M)$ is coherent. We may then write $H_n(M)$ as a finite direct sum of modules of the form $R/Rx$ for $x$ a nonunit in $R$.  The fact that $\mathfrak{m} \otimes_R M = M$ then implies that $\mathfrak{m} \otimes_R R/Rx = R/Rx$ for some nonunit $x$.   We now derive a contradiction from the fact that $\mathfrak{m} \otimes_R R/Rx$ is not finitely generated.
\end{example}

\begin{example}\label{example non torsion}
Let $R$ be a connected commutative ring such that the \'etale cohomology group $H^2(\Spec(R), \GG_m)$ is non-torsion (see \cite{Grothendieck} remark 1.11b, or chapter 7 in \cite{BrauerBook} for examples). Let $\mathcal{G}$ be a non-torsion $\GG_m$-gerbe on $\Spec(R)$, and let $\Mod^\heartsuit_{R, \mathcal{G}}$ be the associated twist of $\Mod^\heartsuit_R$  (see notation \ref{notation twist mod}). Then $\Mod^\heartsuit_{R, \mathcal{G}}$ is an invertible $R$-linear Grothendieck abelian category, and in particular it is $1$-strongly compactly assembled. 

We claim that the only compact projective object of $\Mod^\heartsuit_{R, \mathcal{G}}$ is zero. Assume that $X$ is a compact projective object, and let $\Dcal$ be the full subcategory of $\smash{\Mod^\heartsuit_{R, \mathcal{G}}}$ generated under colimits by $X$.  Since $\mathcal{G}$ is non-torsion, we have that $\Mod^\heartsuit_{R, \mathcal{G}}$  is not the category of modules over an Azumaya $R$-algebra, and therefore $\dcal \neq \smash{\Mod^\heartsuit_{R, \mathcal{G}}}$. Pick a faithfully flat \'etale $R$-algebra $R'$ such that the pullback of $\mathcal{G}$ along the projection $\Spec(R') \rightarrow \Spec(R)$ is trivial. Then we may identify $X \otimes_R R'$ with a finitely generated free $R'$-module. The subcategory of $\Mod_{R'}^\heartsuit$ generated by $X \otimes_R R'$ is $\dcal \otimes_R R'$, so we see that $X \otimes_R R'$ is not a generator for $\Mod_{R'}^\heartsuit$. The fact that $R$ is connected now implies that $X$ vanishes.

Taking $R$ to be a commutative algebra over a field $k$ yields an example of a dualizable $k$-linear Grothendieck abelian category which has no nontrivial compact projective objects. We note that this example extends to the connective derived context as well: for each $1 \leq n \leq \infty$ the full subcategory $(\der(\Mod^\heartsuit_{R, \mathcal{G}})_{\geq 0})_{\leq n-1}$ of the derived category of $\Mod^\heartsuit_{R, \mathcal{G}}$ on the connective $(n-1)$-truncated objects is $n$-strongly compactly assembled and has no nontrivial compact projective objects. Note that the situation is different once one stabilizes: the full derived category $\der(\Mod^\heartsuit_{R, \mathcal{G}})$ is compactly generated, as explained in \cite{ToenAzumaya}.
\end{example}


\subsection{Compact and strongly compact functors}\label{subsection colimits}

Our next goal is to study certain classes of filtered colimits in $\catl$ which are preserved by the forgetful functor to $\catomega$.  The relevant condition that we will need to impose on our diagrams is that the transition maps be compact functors. In the presence of right adjoints, compactness of functors is equivalent to the right adjoints preserving filtered colimits. In general, one has to pass to a suitable completion of the categories to be able to define this right adjoint.

\begin{notation}\label{notation right adjoint after completing}
For each object $\Ccal$ in $\catomega$ we denote by $\widehat{\Ccal}$ the (very large) category obtained by adjoining large colimits to $\Ccal$ while preserving the filtered small colimits in existence. Given a functor $F: \ccal \rightarrow \dcal$ in $\catomega$ we denote by $\widehat{F}: \widehat{\ccal} \rightarrow \widehat{\Dcal}$ the induced functor. Observe that $\widehat{F}$ is a colimit preserving functor between (very large) presentable categories, and it therefore admits a right adjoint.
\end{notation}

\begin{definition}
Let $F: \ccal \rightarrow \Dcal$ be a morphism in $\catomega$. We say that $F$ is compact if the right adjoint to the functor $\widehat{F}$ from notation \ref{notation right adjoint after completing} preserves large colimits.
\end{definition}

\begin{example}\label{example compact is compact}
Let $\Dcal$ be a category with small filtered colimits, and $F: \Delta^0 \rightarrow \Dcal$ be a functor that picks out an object $X$ in $\Dcal$. Then $F$ is compact if and only if $X$ is compact.
\end{example}

\begin{remark}
Let $F: \ccal \rightarrow \Dcal$ be a morphism in $\catomega$ and assume that $F$ admits a right adjoint $F^R$. Then $F$ is compact if and only if $F^R$ preserves small filtered colimits.
\end{remark}

\begin{remark}\label{remark compact sends compact to compact}
Compact functors are closed under compositions. In particular, it follows from example \ref{example compact is compact} that if a functor $F: \Ccal \rightarrow \Dcal$ in $\catomega$ is compact then it sends compact objects to compact objects. The converse is true provided that $\Ccal$ is generated under filtered colimits by compact objects.
\end{remark}

\begin{remark}\label{remark colimit of compacts}
Let $(\Ccal_\alpha)$ be a diagram in $\catomega$ with compact transition maps and colimit $\Ccal$. Denote by $F_{\alpha, \beta}: \Ccal_\alpha \rightarrow \ccal_\beta$ the transition maps, and $F_\alpha: \Ccal_\alpha \rightarrow \Ccal$ the induced functors. Then the functors $(\widehat{F_\alpha})^R$ exhibit $\widehat{\Ccal}$ as the limit of the categories $\widehat{\ccal_\alpha}$ along the transition maps $(\widehat{F_{\alpha, \beta}})^R$. It follows from this that $F_\alpha$ is compact for all $\alpha$.
\end{remark}

\begin{remark}
Let $\Ccal$ be an object of $\catl$. Then we may consider the (very large) category $\widetilde{\Ccal}$ obtained from $\Ccal$ by freely adjoining large colimits while preserving the small colimits in existence. Alternatively, $\widetilde{\Ccal}$ is obtained from $\Ccal$ by freely adjoining large $\kappa$-filtered colimits, where $\kappa$ is the smallest large cardinal. The category $\widehat{\Ccal}$ is then obtained from $\widetilde{\Ccal}$ by freely adjoining large colimits while preserving the large filtered colimits in existence.

Assume now given a functor $F: \Ccal \rightarrow \Dcal$ in  $\catl$. Then the induced functor $\widetilde{F}: \widetilde{\ccal} \rightarrow \widetilde{\dcal}$ is a colimit preserving functor between very large presentable categories, and it therefore has a right adjoint. The functor $F$ is compact if and only if the right adjoint to $\widetilde{F}$ preserves large filtered colimits.  In particular, similar arguments as in remark \ref{remark colimit of compacts} show that if $\Ccal$ is a colimit of a diagram $(\Ccal_\alpha)$ in $\catl$ with compact transition maps, then the functors $\Ccal_\alpha \rightarrow \Ccal$ are compact.
\end{remark}

\begin{example}
Let $\Dcal$ be a cocomplete category and $F: \Spc \rightarrow \Dcal$ be a colimit preserving functor. Then $F$ is compact if and only if $F(\Delta^0)$ is a compact object of $\Dcal$.
\end{example}

\begin{proposition}\label{prop filtered colimits preserved}
The forgetful functor $\catl \rightarrow \catomega$ preserves colimits of filtered diagrams with compact transitions.
\end{proposition}

The proof of proposition  \ref{prop filtered colimits preserved} works by reduction to the case of compactly generated categories:

\begin{lemma}\label{lemma colim of ind completions}
The forgetful functor $\catl \rightarrow \catomega$ preserves colimits of small filtered diagrams of compactly generated presentable categories and compact functors.
\end{lemma}
\begin{proof}
Let $(\Ccal_\alpha)$ be a small filtered diagram of compactly generated presentable categories and compact functors, and let $\Ccal$ be its colimit in $\catl$. Then $\Ccal$ is the colimit of the diagram $(\Ccal_\alpha)$ on the category $\Pr^L_\omega$ of compactly generated presentable categories and compact functors. Passing to compact objects induces an equivalence between this and the category $\Cat^{\rex, \id}$ of finitely cocomplete idempotent complete categories and right exact functors, so that $\Ccal^\omega$ is the colimit of $(\Ccal_\alpha^\omega)$ in $\Cat^{\rex,\id}$. Since the forgetful functor $\Cat^{\rex, \id} \rightarrow \Cat$ preserves filtered colimits we furthermore have that $\Ccal^\omega$ is the colimit of the diagram $(\Ccal_\alpha^\omega)$ in $\Cat$. The lemma now follows from the fact that the functor $\Cat \rightarrow \catomega$ of ind-completion is colimit preserving.
\end{proof}

\begin{proof}[Proof of proposition \ref{prop filtered colimits preserved}]
Assume given a filtered system $(\Ccal_\alpha)$ in $\catl$ with compact transitions, and let $\Ccal$ be its colimit. We must show that $\Ccal$ is also the colimit of this diagram in $\catomega$. Replacing $\Ccal_\alpha$ and $\Ccal$ with $\widetilde{\Ccal_\alpha}$ and $\widetilde{\Ccal}$ and changing the universe, we may assume that all the categories involved are in fact presentable, and the diagram is small.

Let $\kappa$ be a  regular cardinal such that all the categories $\Ccal_\alpha$ are $\kappa$-accessible for all $\alpha$ and all the transition functors and their right adjoints are $\kappa$-accessible. For each $\alpha$ the inclusion $\Ccal_\alpha^\kappa \rightarrow \Ccal_\alpha$ of the subcategory of $\kappa$-compact objects admits an ind-extension $p_\alpha: \Ind(\Ccal_\alpha^\kappa) \rightarrow \Ccal_\alpha$ which is   an accessible localization. Let $p: \Dcal \rightarrow \Ccal$ be the colimit of the maps $p_\alpha$, computed in $\Pr^L$. For each $\alpha$ denote by $g_\alpha: \Ind(\Ccal^\kappa_\alpha) \rightarrow \Dcal$ the canonical map, and note that $p$ exhibits $\Ccal$ as an accessible localization of $\Dcal$, at the union over all $\alpha$ of the image under $g_\alpha$ of the class of $p_\alpha$-local morphisms.

Assume now given an object $\Ecal$ in $\widehat{\Cat}^\omega$. We have a commutative square of spaces
\[
\begin{tikzcd}
\lim \Hom_{\catomega}(\Ind(\Ccal_\alpha^\kappa), \Ecal) & \arrow{l}{}  \Hom_{\catomega}(\Dcal, \Ecal) \\
\lim \Hom_{\catomega}(\Ccal_\alpha, \Ecal) \arrow{u}{\lim p_\alpha^*} & \arrow{l}{} \arrow{u}{p^*} \Hom_{\catomega}(\Ccal, \Ecal).
\end{tikzcd}
\]

By lemma \ref{lemma colim of ind completions}, the upper horizontal arrow is an equivalence. Furthermore, since $p$ and the maps $p_\alpha$ are all localizations, the vertical arrows are inclusions. Hence the bottom horizontal arrow is an inclusion. 

To prove our proposition it remains to show that the bottom horizontal arrow is surjective. In other words, we must show that if $F: \Dcal \rightarrow \Ecal$ is a filtered colimit preserving functor whose restriction to $\Ind(\Ccal^\kappa_\alpha)$ factors through $\Ccal_\alpha$ for all $\alpha$, then $F$ factors through $\Ccal$. It suffices for this to show that every $p$-local map is a filtered colimit of a diagram consisting of maps each of which is a image under $g_\alpha$ of a $p_\alpha$-local morphism for some $\alpha$.

For each $\alpha$ denote by $g_\alpha^R$ the right adjoint to $g_\alpha$. Since the identity of $\Dcal$ is the filtered colimit of the endofunctors $g_\alpha g_\alpha^R$, it suffices to show that $g_\alpha^R$ maps $p$-local maps to $p_\alpha$-local maps for each $\alpha$. This would follow if we are able to show that for every $\alpha$ the square
\[
\begin{tikzcd}
\Ind(\Ccal^\kappa_\alpha) \arrow{d}{p_\alpha} \arrow{r}{g_\alpha} & \Dcal \arrow{d}{p} \\
\Ccal_\alpha \arrow{r}{} & \Ccal
\end{tikzcd}
\]
is horizontally right adjointable. For this it is enough to prove that for each transition map the square
\[
\begin{tikzcd}
\Ind(\Ccal^\kappa_\alpha) \arrow{d}{p_\alpha} \arrow{r}{} & \Ind(\Ccal^\kappa_\beta) \arrow{d}{p_\beta} \\
\Ccal_\alpha \arrow{r}{} & \ccal_\beta 
\end{tikzcd}
\]
is horizontally right adjointable. This is a consequence of the fact that the right adjoints to the transition maps preserve filtered colimits.
\end{proof}

We now discuss a variant of the above where compactness is replaced by strong compactness. In what follows we fix a constant $1 \leq n  \leq \infty$.

\begin{notation}\label{notation extend F}
For each object $\Ccal$ in $\cathat_n^\Sigma$ we denote by $\widehat{\Ccal}$ the (very large) $(n,1)$-category obtained by adjoining large colimits to $\Ccal$ while preserving the $n$-sifted small colimits in existence. Given a functor $F: \Ccal \rightarrow \Dcal$ in $\cathat_n^\Sigma$ we denote by $\widehat{F}: \widehat{\Ccal} \rightarrow \widehat{\Dcal}$ the induced functor. Observe that $\widehat{F}$ is a colimit preserving functor between (very large) presentable categories, and it therefore admits a right adjoint.
\end{notation}

\begin{definition}
Let $F: \Ccal \rightarrow \Dcal$ be a morphism in $\cathat_n^\Sigma$. We say that $F$ is $n$-strongly compact if the right adjoint to the functor $\widehat{F}$ from notation \ref{notation extend F} preserves large colimits.
\end{definition}

\begin{example}\label{example n strongly compact from delta0}
Let $\Dcal$ be an object of $\cathat_n^\Sigma$ and $F: \Delta^0 \rightarrow \Dcal$ be a functor that picks out an object $X$ in $\Dcal$. Then $F$ is $n$-strongly compact if and only if $X$ is $n$-strongly compact.
\end{example}

\begin{remark}
Let $F: \Ccal \rightarrow \Dcal$ be a morphism in $\cathat_n^\Sigma$ and assume that $F$ admits a right adjoint $F^R$. Then $F$ is $n$-strongly compact if and only if $F^R$ preserves $n$-sifted colimits.
\end{remark}

\begin{remark}
Compositions of $n$-strongly compact functors are $n$-strongly compact. In particular, it follows from example \ref{example n strongly compact from delta0} that if a functor $F: \Ccal \rightarrow \Dcal$ in $\cathat_n^\Sigma$ is $n$-strongly compact then it maps  $n$-strongly compact objects to $n$-strongly compact objects. The converse is true provided that $\Ccal$ is generated under $n$-sifted colimits by $n$-strongly compact objects.
\end{remark}

The proof of proposition \ref{prop filtered colimits preserved} adapts to this setting, yielding the following:

\begin{proposition}
Let $\cathat^L_n$ be the subcategory of $\cathat$ on the large $(n,1)$-categories with small colimits and colimit preserving functors. Then the forgetful functor $\cathat^L_n \rightarrow \cathat_n^\Sigma$ preserves colimits of filtered diagrams with $n$-strongly compact transitions.
\end{proposition}


\subsection{Lifting of (strongly) compact objects}\label{subsection lifting}

We are now ready to discuss to what extent the operation of passing to compact objects commutes with taking filtered colimits of categories. The following is our main result on this topic:

\begin{theorem}\label{theorem compacts in filtered colimit}
Let $(\Ccal_\alpha)$ be a filtered diagram in $\catomega$ with compact transition maps, and let $\ccal$ be its colimit. Assume that $\Ccal_{\alpha}$ is compactly assembled for all $\alpha$. Then $\Ccal$ is compactly assembled and the induced functor $\colim \Ccal_\alpha^\omega \rightarrow \Ccal^\omega$ is an equivalence.
\end{theorem}

The proof of theorem \ref{theorem compacts in filtered colimit} requires some preliminary lemmas.
\begin{lemma}\label{lemma filtered colimit of cg}
Let $(\Ccal_\alpha)$ be a filtered diagram in $\catomega$ with compact transition maps, and let $\ccal$ be its colimit. Assume that $\ccal_\alpha$ is generated under filtered colimits by compact objects, for all $\alpha$. Then  $\Ccal$ is generated under filtered colimits by compact objects, and the induced map $\colim \Ccal_\alpha^\omega \rightarrow \Ccal^\omega$ is an equivalence.
\end{lemma}
\begin{proof}
By remark \ref{remark colimit of compacts} we have that the functors $\Ccal_\alpha \rightarrow \Ccal$ are all compact, and in particular they send compact objects to compact objects (remark \ref{remark compact sends compact to compact}). The category $\Ccal$ is generated under filtered colimits by the images of the functors $\Ccal_{\alpha}$. Since $\Ccal_{\alpha}$ is generated under filtered colimits by compact objects for all $\alpha$, we conclude that $\Ccal$ is generated under filtered colimits by compact objects. It remains to show that the map $\colim \Ccal_\alpha^\omega \rightarrow \Ccal^\omega$  is an equivalence.

Since the categories $\Ccal_\alpha$ and $\ccal$ admit filtered colimits, they are in particular idempotent complete, and hence $\ccal_\alpha^\omega$ and $\ccal^\omega$ are all idempotent complete. By \cite{HTT} corollary 4.4.5.21, we have that $\colim \ccal^\omega_\alpha$ is also idempotent complete. Hence the map $\colim \Ccal_\alpha^\omega \rightarrow \Ccal^\omega$ can be recovered by passing to compact objects of the induced functor $\Ind(\colim \Ccal_\alpha^\omega) \rightarrow \Ind(\Ccal^\omega)$. It thus suffices to show that the latter functor is an equivalence. 

Since $\Ind$ is a left adjoint, we have that $\Ind(\colim \Ccal_\alpha^\omega) = \colim \Ind(\ccal_\alpha^\omega)$ (where here the second colimit is computed in $\catomega$). Thus we must show that $\Ind(\ccal^\omega)$ is the filtered colimit of the categories $\Ind(\ccal_\alpha^\omega)$. This follows from the fact that $\Ccal$ and $\Ccal_\alpha$ are all generated under filtered colimits by compact objects.
\end{proof}

\begin{lemma}\label{lemma vertically left adj}
Let $F: \Ccal \rightarrow \Dcal$ be a compact functor in $\catomega$, where $\Ccal$ and $\Dcal$ are compactly assembled. Then the commutative square
\[
\begin{tikzcd}
\Ind(\Ccal) \arrow{r}{\Ind(F)} \arrow{d}{p_\Ccal} & \Ind(\Dcal) \arrow{d}{p_\Dcal} \\
\Ccal \arrow{r}{F} & \Dcal
\end{tikzcd}
\]
is vertically left adjointable. 
\end{lemma}
\begin{proof}
The fact that $\Ccal$ and $\Dcal$ are compactly assembled guarantees that the vertical arrows admit left adjoints, so we only need to verify that the square obtained by passage to left adjoints of the vertical arrows is strictly commutative. Consider the commutative cube
\[
\begin{tikzcd}
& \widehat{\Ind(\Ccal)} \arrow{rr}{\widehat{\Ind(F)}} \arrow[dd, pos=0.7, "\widehat{p_\Ccal}"] & & \widehat{\Ind(\Dcal)} \arrow{dd}{\widehat{p_\Dcal}} \\ 
\Ind(\Ccal) \arrow[rr, pos=0.7, "\Ind(F)"] \arrow{dd}{p_\Ccal} \arrow{ur}{} &  & \Ind(\Dcal) \arrow[dd, pos=0.7, "p_\Dcal"] \arrow{ur}{} & \\
 & \widehat{\Ccal} \arrow[rr,"\widehat{F}", pos=0.6] & & \widehat{\Dcal} \\
\Ccal \arrow{rr}{F} \arrow{ur}{} &  & \Dcal \arrow{ur}{} & 
\end{tikzcd}
\]
where the back face is obtained from the front face by adjoining all large colimits while preserving the small filtered colimits in existence. Here the diagonal arrows are fully faithful, and the left and right face are vertically left adjointable. We may thus reduce to showing that the square
\[
\begin{tikzcd}
\widehat{\Ind(\Ccal)} \arrow{r}{\widehat{\Ind(F)}} \arrow{d}{\widehat{p_\Ccal}} & \widehat{\Ind(\Dcal)} \arrow{d}{\widehat{p_\Dcal}} \\
\widehat{\Ccal} \arrow{r}{\widehat{F}} & \widehat{\Dcal}
\end{tikzcd}
\]
is vertically left adjointable. Since the vertical arrows admit left adjoints it will suffice to show that the above square horizontally right adjointable. The horizontal arrows in the above square are colimit preserving functors between very large presentable categories, so they admit right adjoints $(\widehat{F})^R$ and $\smash{(\widehat{\Ind(F)})^R}$ by the adjoint functor theorem. To finish the proof we must show that the induced natural transformation 
\[
\mu: \widehat{p_\Ccal} (\widehat{\Ind(F)})^R  \rightarrow (\widehat{F})^R \widehat{p_\Dcal} 
\]
is an isomorphism.

Since $p_\Ccal$ admits a fully faithful left adjoint it is a colocalization in $\catomega$, and hence $\widehat{p_\Ccal}$ is a colocalization of categories with large colimits. The adjoint functor theorem implies that $\widehat{p_\Ccal}$ also has a right adjoint $(\widehat{p_\Ccal})^R$, which is then necessarily fully faithful. Similarly, we have that $\widehat{p_\Dcal}$ admits a fully faithful right adjoint $(\widehat{p_\Dcal})^R$. It follows that $\mu$ is an isomorphism when restricted to the image of $(\widehat{p_\Dcal})^R$. 

We claim that the image of $(\widehat{p_\Dcal})^R$ contains $\Dcal$. Assume given an object $X$ in $\Dcal$. Then we may regard $X$ as an object of $\widehat{\Dcal}$, and its image under $(\widehat{p_\Dcal})^R$ is an object $Y$ in $\smash{\widehat{\Ind(\Dcal)}}$ with the property that 
\[
\Hom_{\widehat{\Ind(\Dcal)}}(-, Y)|_{\Dcal^\op} = \Hom_{\Dcal}(-, X) = \Hom_{\widehat{\Ind(\Dcal)}}(-, X)|_{\Dcal^\op}.
\] 
Since $\widehat{\Ind(\Dcal)}$ is the free category with large colimits on $\Dcal$ the inclusion $\Dcal \rightarrow \widehat{\Ind(\Dcal)}$ is dense, and hence $Y$ is equivalent to $X$. This shows that every object in $\Dcal$ belongs to the image of $(\widehat{p_\Dcal})^R$, as claimed. In particular, $\mu$ is an isomorphism when restricted to $\Dcal$.

The fact that $F$ and $\Ind(F)$ are compact implies that $(\widehat{F})^R$ and $(\widehat{\Ind(F)})^R$ preserve large colimits. It follows that the source and target of $\mu$ are colimit preserving, and hence $\mu$ is an isomorphism when restricted to the colimit closure of $\Dcal$. The lemma now follows from the fact that $\widehat{\Ind(\Dcal)}$ is generated under large colimits by $\Dcal$.
\end{proof}

\begin{proof}[Proof of theorem \ref{theorem compacts in filtered colimit}]
Denote by $g_{\alpha, \beta}: \Ccal_\alpha \rightarrow \ccal_\beta$ the transition functors. For each $\alpha$ denote by $p_\alpha: \Ind(\ccal_\alpha) \rightarrow \ccal_\alpha$ the projection, and by $i_\alpha$ its (fully faithful) left adjoint.

For each transition we have a commutative square
\[
\begin{tikzcd}[column sep = large]
\Ind(\ccal_\alpha) \arrow{d}{p_\alpha} \arrow{r}{\Ind(g_{\alpha, \beta})} & \Ind(\ccal_\beta) \arrow{d}{p_\beta}\\
\ccal_\alpha \arrow{r}{g_{\alpha, \beta}} & \ccal_\beta
\end{tikzcd}
\]
which is vertically left adjointable by lemma \ref{lemma vertically left adj}. Passing to the colimit in $\catomega$ we conclude that the induced map $p: \colim \Ind(\Ccal_\alpha) \rightarrow \ccal$ admits a fully faithful left adjoint, which is obtained as the colimit of the maps $i_\alpha$. We denote this left adjoint by $i$. 

It follows from lemma \ref{lemma filtered colimit of cg} that $\colim \Ind(\Ccal_\alpha)$ is compactly generated, and in particular it is compactly assembled. Since $\Ccal$ is a retract of it we conclude that $\Ccal$ is compactly assembled as well.

Passing to compact objects we obtain a commutative square
\[
\begin{tikzcd}
\colim (\Ind(\Ccal_\alpha)^\omega) \arrow{r}{} & (\colim \Ind(\Ccal_\alpha))^\omega \\
\colim \ccal_\alpha^\omega \arrow{u}{\colim {i_\alpha}} \arrow{r}{} & \ccal^\omega \arrow{u}{i} 
\end{tikzcd}
\]
where the colimits on the left column take place in $\cathat$ and the colimit on the top right corner takes place in $\catomega$. Here the vertical arrows are fully faithful, and the top horizontal arrow is an equivalence by lemma \ref{lemma filtered colimit of cg}. Hence the bottom horizontal arrow is fully faithful.

It only remains to show surjectivity. Let $X$ be a compact object in $\Ccal$. Since the top horizontal arrow in the above commutative square is an equivalence, there exists an $\alpha$ and a compact object $Y_\alpha$ in $\Ind(\ccal_\alpha)$ whose image in $\colim \Ind(\Ccal_\alpha)$ is given by $i(X)$. Let $\epsilon_\alpha: i_\alpha p_\alpha(Y_\alpha) \rightarrow Y_\alpha$ be the counit of the adjunction.

The image of $\epsilon_\alpha$ in $\colim \Ind(\ccal_\alpha)$ is given by the counit map $\epsilon: ipi(X) \rightarrow i(X)$, which is an isomorphism. Let $\nu: i(X) \rightarrow ipi(X)$ be an inverse to $\epsilon$. The compactness of $Y_\alpha$ implies that there exists a transition map $g_{\alpha, \beta}$ and a map $\nu_\beta: \Ind(g_{\alpha, \beta})(Y_\alpha) \rightarrow i_\beta p_\beta(\Ind(g_{\alpha, \beta})(Y_\alpha))$ whose image in $\colim(\Ind(\ccal_\alpha))$ is $\nu$. Replacing $\alpha$ with $\beta$ we may in fact assume that $\alpha = \beta$, so that $\nu_\alpha: Y_\alpha \rightarrow i_\alpha p_\alpha(Y_\alpha)$ lifts $\nu$.

The image of $\epsilon_\alpha \nu_\alpha$ in $\colim \Ind(\ccal_\alpha)$ is the identity on $i(X)$. As before, since $Y_\alpha$ is compact we may assume, after replacing $\alpha$ with some other index $\beta$ if necessary, that $\epsilon_\alpha \nu_\alpha$ is the identity on $Y_\alpha$. In other words, $Y_\alpha$ is a retract of $i_\alpha p_\alpha(Y_\alpha)$. Since $\ccal_\alpha$ is idempotent complete, we see that $Y_\alpha$ belongs to the image of $i_\alpha$. Write $Y_\alpha = i_\alpha(X_\alpha)$. Then $X_\alpha$ is a compact object in $\ccal_\alpha$ whose image in $\ccal$ is given by $X$.
\end{proof}

We also have a variant of theorem \ref{theorem compacts in filtered colimit} that applies to $n$-strong compactness for each $1 \leq n \leq \infty$.

\begin{notation}
For each object $\ccal$ in  $\cathat^\Sigma_n$ we denote by $\Ccal^\Sigma$ the full subcategory of $\Ccal$ on the $n$-strongly compact objects.
\end{notation}

\begin{theorem}\label{theorem lift strongly compacts}
Let $(\Ccal_\alpha)$ be a filtered diagram in $\cathat^\Sigma_n$ with $n$-strongly compact transition maps, and let $\Ccal$ be its colimit. Assume that $\Ccal_\alpha$ is $n$-strongly compactly assembled for all $\alpha$. Then $\Ccal$ is $n$-strongly compactly assembled and the induced functor $\colim \Ccal_\alpha^\Sigma \rightarrow \Ccal^\Sigma$ is an equivalence.
\end{theorem}

The proof of theorem \ref{theorem lift strongly compacts} is completely analogous to that of theorem \ref{theorem compacts in filtered colimit}.


\ifx\inmain\undefined
\bibliographystyle{myamsalpha2}
\bibliography{References}
\fi


\section{Smooth categories over a rigid semisimple base}\label{section invertible semisimple}

The goal of this section is to supply proofs of theorems \ref{theorem principal 2 introduction} and \ref{teo introduction spectral} in the case when $R$ is a field. We will in fact obtain this as a corollary of a more general result classifying smooth categories over a simisimple rigid symmetric monoidal Grothendieck abelian category $\acal$ (and its connective derived category).

We begin in \ref{subsection azumaya} with some general background concerning the notions of smooth, proper, fully dualizable, and invertible categories linear over a base presentable symmetric monoidal category. Specialized to the case of categories of modules this yields notions of smooth, proper, and Azumaya algebras. 

In \ref{subsection smooth 11} we prove the first main result of this section: if $\ccal$ is a smooth $\acal$-linear cocomplete category, then $\ccal$ is the category of modules over a smooth algebra in $\acal$. We carry this out in three steps:
\begin{itemize}
\item The fact that $\ccal$ is smooth is used to write the identity of $\ccal$ as a retract of an endofunctor that factors through $\acal$. This allows us to deduce that $\ccal$ is a spectral Grothendieck abelian category.
\item The spectrality of $\ccal$ is used to deduce that its dual $\ccal^\vee$ is locally finitely generated. Reversing the roles of $\ccal$ and $\ccal^\vee$ we conclude that $\ccal$ is semisimple.
\item The smoothness of $\ccal$ is used one last time to show that $\ccal$ admits a single compact projective $\acal$-generator.
\end{itemize} 

In \ref{subsection smooth infty} we prove a variant of the above result: if $\ccal$ is a smooth $\der(\acal)_{\geq 0}$-linear cocomplete category, then $\ccal$ is the category of modules in $\der(\acal)_{\geq 0}$ over a smooth algebra in $\acal$. Our proof of this has two main steps:
\begin{itemize}
\item  The smoothness of $\ccal$ is used to show that  $\Ho(\ccal)$ is a spectral category.
\item Combining the above with our general results on dualizable linear categories from \ref{subsection compactly assembled} we conclude that $\ccal$ is Grothendieck prestable, and the derived category of its heart. The result then follows as a consequence of the $(1,1)$-categorical theorem.
\end{itemize}

As discussed in the introduction, in the stable context there are many more examples of fully dualizable linear categories, and we make no attempt at classifying them. The problem becomes more tractable if instead of fully dualizable categories we study invertible categories: we will prove in  \ref{subsection invertible stable} that every invertible $\der(\acal)$-linear category is the category of modules in $\der(\acal)$ over an Azumaya algebra in $\acal$.


\subsection{Smooth, proper, and Azumaya algebras}\label{subsection azumaya}

We begin with some generalities concerning the notion of smoothness and properness for linear categories.

\begin{definition}\label{def smooth and proper}
Let $\Mcal$ be a presentable symmetric monoidal category. We say that an $\Mcal$-linear cocomplete category $\ccal$ is smooth (resp. proper) if and only if it is dualizable in $\Mod_\Mcal(\catl)$ and the unit (resp. counit) of the duality admits a colimit preserving $\Mcal$-linear right adjoint. We say that $\ccal$ is fully dualizable if it is both smooth and proper.
\end{definition}

\begin{remark}
Let $\Mcal$ be a presentable symmetric monoidal category. It follows from proposition \ref{proposition dualizable is presentable} that if a $\Mcal$-linear cocomplete category is smooth or proper then is automatically presentable.
\end{remark}

\begin{remark}\label{remark base change proper}
Let $f: \Mcal \rightarrow \Mcal'$ be a morphism of commutative algebras in $\Pr^L$. If $\ccal$ is a smooth (resp. proper) $\Mcal$-linear presentable category then $\ccal \otimes_\Mcal \Mcal'$ is a smooth (resp. proper) $\Mcal'$-linear presentable category.
\end{remark}

\begin{remark}\label{remark properness when comp gen}
Let $1 \leq n \leq \infty$. Let $\Mcal$ be a presentable symmetric monoidal additive $(n,1)$-category. Assume that $\Mcal$ is generated under colimits by $n$-strongly compact objects, and that $n$-strongly compact objects and dualizable objects of $\Mcal$ agree. Let $\ccal$ be a dualizable $\Mcal$-linear  presentable category.
\begin{itemize}
\item $\ccal$ is smooth if and only if the identity is $n$-strongly compact in $\Funct_{\Mcal}(\ccal, \ccal)$.
\item Assume that $\ccal$ is  generated under colimits by $n$-strongly compact objects. Then $\ccal$ is proper if and only if for every pair of $n$-strongly compact objects $X, Y$ in $\ccal$ the Hom object $\Hom^\enh_\ccal(X, Y)$ is an $n$-strongly compact object of $\Mcal$.
\end{itemize} 
\end{remark}

\begin{remark}\label{remark end of id dualizable}
Let $\Mcal$ be a presentable symmetric monoidal category and let $\ccal$ be a fully dualizable $\Mcal$-linear presentable category. Then the dual to the right adjoint of the unit (resp. counit) of the duality for $\ccal$ provides an $\Mcal$-linear left adjoint for the counit (resp. unit) of the duality for $\ccal$. We note that these adjoints themselves also have all adjoints, which are $\Mcal$-linear and colimit preserving, see for instance \cite{LurieField} proposition 4.2.3. In particular, it follows from this that the object of endomorphisms of the identity in $\Funct_\Mcal(\ccal, \ccal)$ is dualizable.
\end{remark}

Specializing to categories of modules, we obtain notions of smooth and proper algebras:

\begin{definition}
Let $\Mcal$ be a presentable symmetric monoidal category. An algebra object $A$ in $\Mcal$ is said to be smooth (resp. proper) if $\LMod_A(\Mcal)$ is a smooth (resp. proper) $\Mcal$-linear presentable category.
\end{definition}

\begin{remark}
Let $\Mcal$ be a presentable symmetric monoidal category. Then an algebra $A$ in $\Mcal$ is smooth (resp. proper) if and only if $A^\op$ is smooth (resp. proper).
\end{remark}

\begin{remark}
Let $\Mcal$ be a presentable symmetric monoidal category and let $A$ be an algebra in $\Mcal$. Then:
\begin{itemize}
\item $A$ is smooth if and only if $A$ is right dualizable as a right $A \otimes A^\op$-module.
\item $A$ is proper if and only if $A$ is dualizable as an object of $\Mcal$.
\end{itemize}
\end{remark}

\begin{remark}\label{remark characterizations smooth}
Let $1 \leq n \leq \infty$. Let $\Mcal$ be a presentable symmetric monoidal additive $(n,1)$-category. Assume that $\Mcal$ is generated under colimits by $n$-strongly compact objects, and that $n$-strongly compact objects and dualizable objects of $\Mcal$ agree. Let $A$ be an algebra in $\Mcal$. Then the following are equivalent:
\begin{itemize}
\item $A$ is smooth.
\item $A$ is $n$-strongly compact as an $A$-bimodule.
\item The canonical morphism of $A$-bimodules $A \otimes A \rightarrow A$ admits a section (in other words, $A$ is separable).
\end{itemize} 
\end{remark}

\begin{remark}\label{remark separable are semisimple}
Let $\acal$ be a semisimple rigid symmetric monoidal Grothendieck abelian category and let $A$ be a smooth algebra in $\acal$. It follows from remark \ref{remark characterizations smooth} that every left $A$-module in $\acal$ is a retract of a free left $A$-module. The fact that every object of $\acal$ is projective now implies that every left $A$-module is projective. Therefore $\LMod_A(\acal)$ is semisimple.
\end{remark}

\begin{remark}
Let $\acal$ be a Grothendieck abelian category, rigid and generated by compact projective objects, and let $A$ be a flat algebra in $\acal$. Then $A$ is smooth as an algebra in $\acal$ if and only if it is smooth as an algebra in $\der(\acal)_{\geq 0}$.
\end{remark}

We now concentrate in those smooth and proper algebras whose category of modules is invertible.

\begin{definition}
Let $\Mcal$ be a presentable symmetric monoidal category. An algebra object $A$ in $\Mcal$ is said to be Azumaya if $\LMod_A(\Mcal)$ is an invertible object of $\Mod_\Mcal(\Pr^L)$.
\end{definition}

\begin{remark}
Let $\Mcal$ be a presentable symmetric monoidal category and let $A$ be an algebra in $\Mcal$. Then $\LMod_A(\Mcal)$ is a dualizable object of $\Mod_\Mcal(\Pr^L)$, with dual given by $\RMod_A(\Mcal)$. It follows that $A$ is Azumaya if and only if $A^\op$ is Azumaya, and both conditions are equivalent to $A \otimes A^\op$ being Morita equivalent to the unit algebra in $\Mcal$.
\end{remark}

Our next proposition provides an alternative characterization of Azumaya algebras which makes no reference to Morita theory:

\begin{proposition}\label{prop equivalences azumaya}
Let $\Mcal$ be a presentable symmetric monoidal category and let $A$ be an algebra in $\acal$. Then $A$ is Azumaya if and only if the following conditions are satisfied:
\begin{enumerate}
\item $A$ is dualizable (as an object of $\Mcal$).
\item The functor $- \otimes A : \Mcal \rightarrow \Mcal$ is conservative.
\item The canonical left $A \otimes A^\op$-module structure on $A$ exhibits $A \otimes A^\op$ as the endomorphism algebra of $A$.
\end{enumerate}
\end{proposition}

The proof of proposition \ref{prop equivalences azumaya} depends on the following lemma:

\begin{lemma}\label{lemma equivalences conservative}
Let $\Mcal$ be a presentable symmetric monoidal category and let $X$ be a dualizable object of $\Mcal$. The following are equivalent:
\begin{enumerate}[\normalfont (1)]
\item  The functor $- \otimes X : \Mcal \rightarrow \Mcal$ is conservative.
\item The functor $- \otimes X^\vee: \Mcal \rightarrow \Mcal$ is conservative.
\item The object $X$ generates $\Mcal$ under colimits and tensors.
\end{enumerate}
\end{lemma}
\begin{proof}
We first show that (1) implies (2). Let $f: Y \rightarrow Z$ be a morphism in $\mcal$ such that $f \otimes X^\vee$ is an isomorphism. We wish to show that $f$ itself is an isomorphism. Since $X$ is a retract of $X \otimes X^\vee \otimes X$ we have that $f \otimes X$ is a retract of $f \otimes X \otimes X^\vee \otimes X$. The latter is an isomorphism and therefore $f \otimes X$ is an isomorphism too. The fact that $f$ is an isomorphism now follows from our assumption that $- \otimes X$ is conservative.

Next we show that (2) implies (3). Let $\Dcal$ be the full subcategory of $\Mcal$ generated by $X$ under colimits and tensors. We wish to show that $\Dcal = \Mcal$. To do so it suffices to show that the family of functors $\Hom_{\Mcal}(Z, -)$ over all $Z$ in $\Dcal$ is jointly conservative.  This would follow if we show that the functor $- \otimes X : \Mcal \rightarrow \Mcal$ admits a conservative right adjoint. Indeed, this has a right adjoint given by $- \otimes X^\vee$, which we have assumed to be conservative.

It remains to show that (3) implies (1). Let $f: Y \rightarrow Z$ be a morphism in $\mcal$ such that $f \otimes X$ is an isomorphism. We wish to show that $f$ is an isomorphism. Since $X$ generates $\Mcal$ under colimits and tensors we see that $f \otimes Y$ is an isomorphism for all $Y$ in $\mcal$. Specializing $Y$ to the unit of $\mcal$ shows that $f$ is an isomorphism, as desired. 
\end{proof}

\begin{proof}[Proof of proposition \ref{prop equivalences azumaya}]
Assume first that $A$ is Azumaya. Then the functor 
\[
F(-) = - \otimes_{A \otimes A^\op} A : \RMod_{A \otimes A^\op}(\Mcal) \rightarrow \Mcal
\]
 is an equivalence. In particular, the $A \otimes A^\op$-$1_\Acal$ bimodule $A$ is right dualizable. It follows from this that $A$ is right dualizable as an object of $\Mcal$, so that (1) holds. The fact that (3) holds follows from the $\Acal$-linearity of $F$ together with the fact that $A \otimes A^\op$ is the algebra of endomorphisms of $A \otimes A^\op$ as a right $A \otimes A^\op$-module. It now remains to prove (2). Since $F$ is colimit preserving and $\mcal$-linear and $A \otimes A^\op$ generates $\RMod_{A \otimes A^\op}(\mcal)$ under colimits and tensors, we see that $A$ generates $\mcal$ under colimits and tensors. Item (2) now follows from lemma \ref{lemma equivalences conservative}.

Assume now that $A$ satisfies conditions (1), (2) and (3). Let $F(-)$ be as above. We wish to show that $F$ is an equivalence. It follows from (1) that $A$ is right dualizable as a $A \otimes A^\op$-$1_\Acal$-bimodule. Hence $F$ admits a right adjoint, and the unit of the adjunction is given by tensoring with the morphism of $A \otimes A^\op$-bimodules $A \otimes A^\op \rightarrow A \otimes A^\vee$ induced from the left $A \otimes A^\op$-module structure on $A$. It follows from (3) that this is an isomorphism, so that $F$ is fully faithful.  Its image consists of the full subcategory of $\mcal$ generated by $A$ under colimits and tensors. The fact that $F$ is an equivalence now follows from (2) together with lemma \ref{lemma equivalences conservative}.
\end{proof}

\begin{corollary}\label{coro compare azumayas 1}
Let $\acal$ be a rigid symmetric monoidal Grothendieck abelian category generated by compact projective objects and let $\acal^\cp$ be the full subcategory of $\acal$ on the compact projective objects. Then an algebra $A$ in $\acal^\cp$ is Azumaya as an algebra in $\acal$ if and only if it is Azumaya as an algebra in $\der(\acal)_{\geq 0}$.
\end{corollary}
\begin{proof}
We verify that conditions  (2) and (3) from proposition \ref{prop equivalences azumaya} hold for $A$ as an algebra in $\acal$ if and only if they hold for $A$ as an algebra in $\der(\acal)_{\geq 0}$:
\begin{itemize}
\item Since $- \otimes A : \der(\acal)_{\geq 0} \rightarrow \der(\acal)_{\geq 0}$ is exact, it is conservative if and only if its restriction to the heart is conservative. This restriction agrees with the functor $- \otimes A: \acal \rightarrow \acal$.
\item Both $A \otimes A^\op$ and $\End(A) = A^\vee \otimes A$ belong to $\acal^\cp$, so the map $A \otimes A^\op \rightarrow \End(A)$ is an isomorphism when computed in $\der(\acal)_{\geq 0}$ if and only if it is an isomorphism when computed in $\acal$.
\end{itemize}
\vspace{-\baselineskip}
\end{proof}

\begin{corollary}\label{coro compare azumayas 2}
Let $\Mcal$ be a rigid symmetric monoidal Grothendieck prestable category generated under colimits by compact projective objects. Then an algebra $A$ in $\mcal$ is Azumaya if and only if it is Azumaya when regarded as an algebra in $\Sp(\mcal)$.
\end{corollary}
\begin{proof}
We verify that conditions  (2) and (3) from proposition \ref{prop equivalences azumaya} hold for $A$ as an algebra in $\Mcal$ if and only if they hold for $A$ as an algebra in $\Sp(\Mcal)$:
\begin{itemize}
\item The functor $- \otimes A: \Sp(\Mcal) \rightarrow \Sp(\Mcal)$ is t-exact, so it is conservative if and only if it is conservative when restricted to connective objects.
\item The map $A \otimes A^\op \rightarrow \End(A) = A^\vee \otimes A$ is compatible with the inclusion $\Mcal \rightarrow \Sp(\mcal)$, so it is an isomorphism when computed in $\Mcal$ if and only if it is an isomorphism when computed in $\Sp(\Mcal)$.
\end{itemize}
\vspace{-\baselineskip}
\end{proof}

\begin{corollary}
Let $\acal$ be a rigid symmetric monoidal Grothendieck abelian category generated by compact projective objects and let $\acal^\cp$ be the full subcategory of $\acal$ on the compact projective objects.  Then an algebra $A$ in $\acal^\cp$ is Azumaya as an algebra in $\acal$ if and only if it is Azumaya as an algebra in $\der(\acal)$.
\end{corollary}
\begin{proof}
Combine corollaries \ref{coro compare azumayas 1} and \ref{coro compare azumayas 2}.
\end{proof}


\subsection{Smooth \texorpdfstring{$(1,1)$}{(1,1)}-categories}\label{subsection smooth 11}

We are now ready to state our first main result, that describes smooth $(1,1)$-categories over a rigid semisimple base:

\begin{theorem}\label{theorem abelian}
Let $\Acal$ be a semisimple rigid symmetric monoidal Grothendieck abelian category and let $\Ccal$ be a smooth $\acal$-linear cocomplete category. Then $\Ccal = \LMod_A(\Acal)$ for some smooth algebra $A$ in $\acal$.
\end{theorem}

The remainder of this section is devoted to the proof of theorem \ref{theorem abelian}.

\begin{lemma}\label{lemma factor identity}
Let $\acal$ be a symmetric monoidal Grothendieck abelian category, rigid and generated by compact projective objects. Let $\ccal$ be a smooth $\acal$-linear presentable category. Then there exist objects $F$ in $\ccal^\vee$ and $G$ in $\ccal$ such that the identity of $\ccal$ is a retract of the composite functor
\[
\ccal \xrightarrow{F(-)} \acal \xrightarrow{- \otimes G} \ccal.
\]
\end{lemma}
\begin{proof}
Let $G$ be a colimit generator for $\ccal$ and $F$ be a colimit generator for $\ccal^\vee$. Then $G \otimes F$ is a colimit generator for $\ccal \otimes_\acal \ccal^\vee = \Funct_\acal(\ccal, \ccal)$. Applying corollary \ref{coro properties dualizable 11} we see that $\ccal$ is Grothendieck abelian, and therefore every object receives an epimorphism from a direct sum of copies of $G \otimes F$. The fact that $\ccal$ is smooth implies that the identity $\id_\ccal$ is compact projective, and it is therefore a retract of a finite direct sum of copies of $G \otimes F$. Replacing $G$ by a direct sum of copies of $G$ if necessary, we may in fact assume that $\id_\ccal$ is a retract of $G \otimes F$. The lemma now follows from the observation that under the equivalence $\Funct_\acal(\ccal, \ccal) = \ccal \otimes_\acal \ccal^\vee$, the object $G \otimes F$ corresponds to  $F(-) \otimes G$.
\end{proof}

\begin{lemma}\label{lemma conclude spectrality}
Let $\Ccal$ be an idempotent complete locally small classical additive category with small direct sums and let $\Dcal$ be a spectral category. Assume that the identity functor of $\ccal$ is a retract of a composite functor
\[
\ccal \xrightarrow{F} \Dcal \xrightarrow{G} \ccal
\]
where $F$ and $G$ preserve small direct sums. Then $\ccal$ is a spectral category.
\end{lemma}
\begin{proof}
We first show that $\ccal$ admits all kernels and cokernels. Let $f: X \rightarrow Y$ be a morphism in $\ccal$. We wish to show that $f$ admits a kernel and a cokernel. Since $f$ is a retract of $G(F(f))$ and $\ccal$ is idempotent complete, it suffices to show that $G(F(f))$ admits a kernel and a cokernel. The fact that $\Dcal$ is spectral implies that $F(f)$ is the direct sum of an isomorphism and a null map. Hence $G(F(f))$ is a direct sum of an isomorphism and a null map. Both of these kinds of maps admit both a kernel and a cokernel in $\ccal$, and therefore $G(F(f))$ also admits a kernel and a cokernel, as desired.

We now show that every monomorphism in $\ccal$ admits a retraction. Let $i: X \rightarrow Y$ be a monomorphism. Then $i$ is a retract of $G(F(i))$. Since $\Dcal$ is spectral, we may factor $F(i)$ as a composition of a retraction followed by a section. We now have a commutative diagram
\[
\begin{tikzcd}
X \arrow{d}{j_X} \arrow{rr}{i} & & Y\arrow{d}{j_Y} \\
G(F(X)) \arrow{r}{r} & Z  \arrow{r}{s'} & G(F(Y))
\end{tikzcd}
\]
where the vertical arrows and the map $s'$ are sections, and $r$ is a retraction. Let $s$ be a section for $r$ and $r_X, r_Y, r'$ be retractions for $j_X, j_Y$ and $s'$, respectively, and note that we may assume that the diagram obtained from the above by replacing $j_X, j_Y$ with $r_X, r_Y$ commutes. 

We claim that the composition $r_X s r' j_Y$ is a retraction for $i$. Note that
\[
r_X s r' j_Y i  = r_X s r' s' r j_X = r_X s r j_X.
\]
We must show that the above is the identity on $X$. Since $i$ is a monomorphism, it suffices to show that $ i r_X s r j_X$ is equal to $i$. Since the diagram obtained from the above by replacing the vertical arrows with their retractions commutes, we have that $i r_X s r j_X$  is the same as 
\[
r_Y s' r s r j_X = r_Y s' r j_X = r_Y j_Y i = i,
\] as desired.

In a similar way one proves that every epimorphism in $\ccal$ admits a section. In particular we see that $\ccal$ is an additive category with all kernels and cokernels and such that every monomorphism is a kernel and every epimorphism is a cokernel. Therefore $\ccal$ is abelian. 

Since we know that $\ccal$ is locally small, admits infinite direct sums, and every monomorphism in $\ccal$ is split, in order to show that $\ccal$ is spectral it only remains to show that it admits a generator and has left exact filtered colimits.

We first show that $\ccal$ has left exact filtered colimits. Let $f: X \rightarrow Y$ be a map in $\ccal$, and assume that $f$ is written as a filtered colimit of a family of monomorphisms $f_\alpha: X_\alpha \rightarrow Y_\alpha$. We must show that $f$ is a monomorphism. Since $f$ is a retract of $G(F(f))$, it suffices to show that $G(F(f))$ is a monomorphism. The fact that $\dcal$ is spectral implies that $G$ preserves monomorphisms, so we may reduce to showing that $F(f)$ is a monomorphism. Since $F$ preseves direct sums and every exact sequence in $\ccal$ is split, we have that $F$ preserves arbitrary colimits. Thus $F(f)$ is the colimit of the maps $F(f_\alpha)$, and using the fact that $\dcal$ is Grothendieck we may further reduce to showing that $F(f_\alpha)$ is a monomorphism for all $\alpha$. This follows from the fact that $f_\alpha$ is a monomorphism and thus it admits a retraction.

It remains to show that $\ccal$ admits a generator. Let $U$ be a generator for $\dcal$. Let $X$ be an object of $\ccal$. Since $\dcal$ is spectral, $F(X)$ is a retract of a direct sum of copies of $U$. Thus $G(F(X))$ is a retract of a direct sum of copies of $G(U)$, and since $X$ is a retract of $G(F(X))$ we see that $X$ is also a retract of a direct sum of copies of $G(U)$. Thus $G(U)$ is a generator for $\ccal$.
\end{proof}

\begin{lemma}\label{lemma dual is locally finitely generated}
Let $\ccal$ be a spectral Grothendieck abelian category. Then $\Funct^L(\ccal, \Ab)$ is a locally finitely generated Grothendieck abelian category.
\end{lemma}
\begin{proof}
By proposition \ref{prop classif spectral} there is a left exact localization $q: \LMod_A(\Ab) \rightarrow \ccal$ where $A$ is a left self injective von Neumann regular ring. The dual to $q$ supplies a colimit preserving fully faithful functor $i: \Funct^L(\ccal, \Ab) \rightarrow \RMod_A(\Ab)$, whose image is the full subcategory of $\RMod_A(\Ab)$ on those right $A$-modules $M$ with the property that $M \otimes_A f$ is invertible for every $q$-local morphism $f$. Since $A$ is von Neumann regular, $M \otimes_A f$ is invertible if and only if $M \otimes_A \Ker(f) = M \otimes_A \coker(f) = 0$. It follows that a right $A$-module $M$ belongs to the image of $i$ if and only if $M \otimes_A X =0$ for all $X$ in the kernel of $q$. Using again the fact that $A$ is von Neumann regular we conclude that the image of $i$ is closed under passage to subobjects. The lemma now follows from the fact that $\RMod_A(\Ab)$ is a locally finitely generated Grothendieck abelian category. 
\end{proof}

\begin{lemma}\label{lemma suffices to show generated by compact projectives}
Let $\acal$ be a symmetric monoidal Grothendieck abelian category, rigid and generated by compact projective objects. Let $\Ccal$ be a smooth $\acal$-linear Grothendieck abelian category, generated by compact projective objects. Then $\Ccal =\LMod_A(\Acal)$ for some smooth algebra $A$ in $\Acal$.
\end{lemma}
\begin{proof}
Let $\lbrace X_t \rbrace_{t \in T}$ be a generating set of compact projective objects for $\ccal$. For each finite subset $\alpha \subseteq T$ let $\ccal_\alpha$ be the smallest full subcategory of $\ccal$ closed under colimits and the action of $\acal$, and containing $\lbrace X_t \rbrace_{t \in \alpha}$. Note that each $\ccal_\alpha$ inherits an $\acal$-linear structure from $\ccal$, and that $\ccal = \colim \ccal_\alpha$. Let $\delta$ be the image of $1_\acal$ under the unit map $\acal \rightarrow \ccal \otimes_\acal \ccal^\vee$. We have $\ccal = \colim \ccal_\alpha$ and therefore $\ccal \otimes_\acal \ccal^\vee = \colim \ccal_\alpha \otimes_\acal \ccal^\vee$. Applying lemma \ref{lemma filtered colimit of cg} we see that $\delta$ belongs to $\ccal_\alpha \otimes_\acal \ccal^\vee$ for some $\alpha$ and therefore the identity on $\ccal$ belongs to the image of the inclusion $\Funct_\acal(\ccal, \ccal_\alpha) \rightarrow \Funct_\acal(\ccal, \ccal)$. This implies that the inclusion of $\ccal_\alpha$ into $\ccal$ admits a section, so that $\ccal = \ccal_\alpha$. The object $\bigoplus_{t \in \alpha} X_t$ is therefore a compact projective $\acal$-generator for $\ccal$. The lemma now follows from proposition \ref{prop lex localization}.
\end{proof}

\begin{proof}[Proof of theorem \ref{theorem abelian}]
Applying proposition \ref{proposition dualizable is presentable} we see that $\ccal$ and $\ccal^\vee$ are presentable. Combining lemma \ref{lemma factor identity} with lemma \ref{lemma conclude spectrality} we see that $\ccal$ and $\ccal^\vee$ are spectral categories. Combining remark \ref{remark description duality over M} with lemma \ref{lemma dual is locally finitely generated} we see that $\ccal$ and $\ccal^\vee$ are in fact semisimple. The theorem now follows from lemma \ref{lemma suffices to show generated by compact projectives}.
\end{proof}


\subsection{Smooth \texorpdfstring{$(\infty,1)$}{(∞,1)}-categories}\label{subsection smooth infty}

We now formulate a variant of theorem \ref{theorem abelian} that applies to smooth additive categories over a rigid semisimple base in the derived (but unstable) context:

\begin{theorem}\label{theorem prestable}
Let $\Acal$ be a semisimple rigid symmetric monoidal Grothendieck abelian category and let $\Ccal$ be a smooth $\der(\acal)_{\geq 0}$-linear cocomplete category. Then $\Ccal = \LMod_A(\der(\acal)_{\geq 0})$ for some smooth algebra $A$ in $\acal$.
\end{theorem}

The proof of theorem \ref{theorem prestable} requires some preliminary lemmas.

\begin{lemma}\label{lemma split identity en prestable}
Let $\acal$ be a symmetric monoidal Grothendieck abelian category, rigid and generated by compact projective objects. Let $\Ccal$ be a smooth $\der(\acal)_{\geq 0}$-linear presentable category. Then there exists an object $G$ in $\Ccal$ and an object $F$ in $\Ccal^\vee$ with the property that the identity functor of $\Ccal$ is a retract of the composite functor
\[
\Ccal \xrightarrow{F(-)} \der(\Acal)_{\geq 0} \xrightarrow{ G \otimes -} \Ccal.
\]
\end{lemma}
\begin{proof}
Analogous to the proof of lemma \ref{lemma factor identity}.
\end{proof}

\begin{lemma}\label{lemma ho is spectral prestable}
Let $\acal$ be a symmetric monoidal Grothendieck abelian category, rigid and generated by compact projective objects. Let $\Ccal$ be an invertible $\der(\acal)_{\geq 0}$-linear presentable category. Then $\Ho(\Ccal)$ is a spectral category.
\end{lemma}
\begin{proof}
It follows from lemma \ref{lemma identity retract} that we may factor the identity on $\Ho(\Ccal)$ as the composition
\[
\Ho(\Ccal) \xrightarrow{\Ho(F(-))} \Ho(\der(\Acal)_{\geq 0}) \xrightarrow{\Ho(G \otimes - )} \Ho(\Ccal)
\]
where $F$ and $G$ are objects of $\Ccal^\vee$ and $\Ccal$, respectively. The lemma now follows from  lemma \ref{lemma conclude spectrality}, since $\Ho(\der(\acal)_{\geq 0})$ is equivalent to the category of nonnegatively graded objects of $\acal$, which is semisimple.
\end{proof}

\begin{proof}[Proof of theorem \ref{theorem prestable}]
Applying \ref{coro properties dualizable infty} we see that $\ccal$ is a separated Grothendieck prestable category. By theorem \ref{theorem abelian} there exists a smooth algebra $A$ in $\acal$ and an $\acal$-linear equivalence $\LMod_A(\acal) = \ccal^\heartsuit$. Since $\LMod_A(\der(\acal)_{\geq 0}) = \der(\LMod_A(\acal))_{\geq 0}$, to prove the corollary it will suffice to show that $\ccal$ is the connective derived category of its heart. Since $\ccal$ is separated, we may reduce to showing that $\ccal$ is generated by $\ccal^\heartsuit$.

Let $X$ be an object of $\ccal$ and let $\mu: H_0(X) \rightarrow (\tau_{\geq 1}X) [1]$ be the morphism classifying the extension $\tau_{\geq 1} X \rightarrow X \rightarrow H_0(X)$.  By lemma \ref{lemma ho is spectral prestable} we may write $\mu$ as a sum of an isomorphism and a zero map. Since every summand of $H_0(X)$ is $0$-truncated and every summand of $(\tau_{\geq 1}X)[1]$ is $2$-connective, we must have that $\mu$ is homotopic to zero. It follows that the projection $X \rightarrow H_0(X)$ admits a section $H_0(X) \rightarrow X$, which is a morphism from a $0$-truncated object into $X$ inducing an isomorphism on $H_0$. Since $X$ was arbitrary we conclude that $\ccal^\heartsuit$ generates $\ccal$, as desired.
 \end{proof} 


\subsection{Invertible stable categories}\label{subsection invertible stable}

We now study invertible stable categories over a rigid semisimple base.

\begin{theorem}\label{theorem stable}
Let $\acal$ be a semisimple rigid symmetric monoidal Grothendieck abelian category and let $\ccal$ be an invertible cocomplete $\der(\acal)$-linear category. Then $\Ccal = \LMod_A(\der(\acal))$ for some Azumaya algebra $A$ in $\acal$.
\end{theorem}

The remainder of this section is devoted to the proof of theorem \ref{theorem stable}. 

\begin{notation}
Let $\ccal$ be a category. For each object $X$ in $\Ccal$ we denote by $[X]$  its image in the homotopy category $\Ho(\Ccal)$.
\end{notation}

\begin{lemma}\label{lemma homot category is rigid ss}
Let $\acal$ be a semisimple rigid  symmetric monoidal Grothendieck abelian category. Then the symmetric monoidal structure on $\Ho(\der(\Acal))$ induced from $\der(\Acal)$ is compatible with colimits, and makes $\Ho(\der(\Acal))$ into a semisimple rigid symmetric monoidal Grothendieck abelian category.
\end{lemma}
\begin{proof}
Passing to homologies provides an equivalence $\Ho(\der(\Acal)) = \Acal^\ZZ$. It follows from this that $\Ho(\der(\Acal))$ is a semisimple Grothendieck abelian category. Since the symmetric monoidal structure on $\der(\Acal)$ is compatible with infinite direct sums the same fact holds for the symmetric monoidal structure on $\Ho(\der(\Acal))$. We now conclude that the symmetric monoidal structure on $\Ho(\der(\Acal))$ is compatible with all colimits, since every coproduct preserving functor of spectral categories is colimit preserving.

It remains to show that dualizable and compact projective objects of $\Ho(\der(\Acal))$ coincide. The unit of $\Ho(\der(\Acal))$ is given by $[1_\Acal]$ where $1_\Acal$ is the unit of $\Acal$. We may write $1_\Acal$ as a finite direct sum of simple objects of $\Acal$. Objects which are indecomposable in $\Acal$ necessarily remain indecomposable in $\Ho(\der(\Acal))$, thus we see that the unit of $\Ho(\der(\Acal))$ is a finite direct sum of simple objects, which means that it is compact projective. This implies that every dualizable object of $\Ho(\der(\Acal))$ is compact projective. Conversely, assume given an object $X$ of $\der(\Acal)$ such that $[X]$ is a compact projective object of $\Ho(\der(\Acal))$. We want to prove that $[X]$ is dualizable. Write $X$ as a finite direct sum of objects $X_\alpha$ such that $[X_\alpha]$ is simple for all $\alpha$. Then for each $\alpha$ the object $X_\alpha$ is a shift of a simple object of $\Acal$. Since simple objects of $\Acal$ are dualizable we conclude that $X_\alpha$ is dualizable in $\der(\Acal)$ for all $\alpha$. Therefore $X$ is a dualizable object of $\der(\Acal)$ and $[X]$ is dualizable in $\Ho(\der(\Acal))$, as desired.
\end{proof}

\begin{lemma}\label{lemma identity retract}
Let $\acal$ be a semisimple rigid symmetric monoidal Grothendieck abelian category and let $\ccal$ be an invertible $\der(\acal)$-linear presentable category. Then there exists an object $G$ in $\ccal$ and an object $F$ in $\ccal^\vee$ with the property that the identity functor of $\ccal$ is a retract of the composite functor
\[
\Ccal \xrightarrow{F(-)} \der(\acal) \xrightarrow{G \otimes -} \ccal.
\]
\end{lemma}
\begin{proof}
Let $G$ be a colimit generator for $\Ccal$ and $F$ be  a colimit generator for $\Ccal^\vee$. Then $G \otimes  F$ is a colimit generator for $\Ccal \otimes_{\der(\Acal)} \Ccal^\vee = \der(\Acal)$. Denote by $X$ the induced object of $\der(\Acal)$. Let $\Ucal$ be the collection of simple objects of $\Acal$ which are direct summands of $H_n(X)$ for some integer $n$. Then every object $Y$ in the closure of $X$ under colimits has the property that $H_n(Y)$ is a direct sum of simples in $\Ucal$ for every integer $n$. In particular this applies to the unit $1_\Acal$ in $\Acal$. Replacing $G$ with $\bigoplus_{n \in \ZZ} G[n]$ we may assume that every simple in $\Ucal$ is a direct summand of $H_0(X)$. Replacing $G$ further with an infinite direct sum of copies of $G$ we may in fact assume that $1_\Acal$ is a direct summand of $H_0(X)$. Since every object of $\der(\Acal)$ is a direct sum of its homologies we in fact have that $1_\Acal$ is a direct summand of $X$. The lemma now follows from the observation that under the equivalences $\Funct_{\der(\Acal)}(\Ccal, \Ccal) = \Ccal \otimes_{\der(\Acal)} \Ccal^\vee = \der(\Acal)$ the objects $1_\Acal$ and $G \otimes F$ correspond to the identity on $\Ccal$ and $F(-) \otimes G$, respectively.
\end{proof}

\begin{lemma}\label{lemma dual of ho}
Let $\acal$ be a semisimple rigid symmetric monoidal Grothendieck abelian category and let $\ccal$ be an invertible $\der(\acal)$-linear presentable category.  Then:
\begin{enumerate}[\normalfont (1)]
\item $\Ho(\ccal)$ and $\Ho(\ccal^\vee)$ are spectral categories.
\item There is an equivalence of categories $\Ho(\ccal^\vee) = \Funct^L(\Ho(\ccal), \Ab)$.
\end{enumerate}
\end{lemma}
\begin{proof}
By lemma \ref{lemma identity retract} we have $\der(\acal)$-linear colimit preserving functors $F: \ccal \rightarrow \der(\acal)$ and $G: \der(\acal) \rightarrow \ccal$ such that the identity on $\ccal$ sits in a split retraction $\id_\ccal \xrightarrow{s} G \circ F \xrightarrow{r} \id_\ccal$. It follows in particular that we may write the identity on $\Ho(\Ccal)$ as a retract of  the composition
\[
\Ho(\Ccal) \xrightarrow{\Ho(F)} \Ho(\der(\Acal)) \xrightarrow{\Ho(G )} \Ho(\Ccal).
\]
 Combining lemmas \ref{lemma homot category is rigid ss} and  \ref{lemma conclude spectrality} we see that $\Ho(\ccal)$ is spectral. Reversing the role of $\ccal$ and $\ccal^\vee$ we have that $\Ho(\ccal^\vee)$ is also spectral. This proves (1).

It remains to show that $\Ho(\ccal)$ and $\Ho(\ccal^\vee)$ are dual to each other. By \cite{SAG} proposition 7.7.1 we have that $\ccal$ and $\ccal^\vee$ are dual presentable stable categories, and in particular $\ccal^\vee = \Funct^L(\ccal, \Sp)$. Composition with the functor $H_0: \Sp \rightarrow \Ab$ provides a functor $p: \Funct^L(\ccal, \Sp) \rightarrow \Funct^\oplus(\ccal, \Ab)$, where $\Funct^\oplus$ denotes the category of functors which preserve arbitrary direct sums. We have $\Funct^\oplus(\ccal, \Ab) = \Funct^\oplus(\Ho(\ccal), \Ab)$, and since $\Ho(\ccal)$ is spectral the latter agrees with $\Funct^L(\Ho(\ccal), \Ab)$. To prove the lemma it will suffice to show that $p$ induces an equivalence $\Ho(\Funct^L(\ccal,\Sp)) = \Funct^\oplus(\ccal, \Ab)$.

We have a commutative diagram of $(1,1)$-categories
\[
\begin{tikzcd}
\Ho(\Funct^L(\ccal, \Sp)) \arrow{r}{G^*} \arrow{d}{p} & \Ho(\Funct^L(\der(\acal), \Sp)) \arrow{d}{q} \arrow{r}{F^*}& \Ho(\Funct^L(\ccal, \Sp)) \arrow{d}{p} \\
\Funct^\oplus(\ccal, \Ab) \arrow{r}{G^*} & \Funct^\oplus(\der(\acal), \Ab) \arrow{r}{F^*} & \Funct^\oplus(\ccal, \Ab)
\end{tikzcd}
\]
where $q$ is the functor of composition with $H_0$. We claim that $q$ is an equivalence. Since $\der(\acal)$ is a rigid compactly generated presentable symmetric monoidal stable category, we have an equivalence $\der(\acal) = \Funct^L(\der(\acal), \Sp)$ induced by the pairing
\[
\der(\acal) \times \der(\acal) \xrightarrow{- \otimes -} \der(\acal) \xrightarrow{\Hom^\enh_{\der(\acal)}(1_\acal, -)} \Sp.
\]
To show that $q$ is an equivalence we have to prove that the composite functor
\[
\der(\acal) \times \der(\acal) \xrightarrow{- \otimes -} \der(\acal) \xrightarrow{\Ext^0_{\der(\acal)}(1_\acal, -)} \Ab
\]
induces an equivalence between $\Ho(\der(\acal))$ and the category of coproduct preserving functors $\der(\acal) \rightarrow \Ab$. Since $\Ho(\der(\acal))$ is semisimple, passing to the homotopy category induces an equivalence $\Funct^\oplus(\der(\acal),\Ab) = \Funct^L(\Ho(\der(\acal)), \Ab)$.  We may thus reduce to showing that the functor
\[
\Ho(\der(\acal)) \times \Ho(\der(\acal)) \xrightarrow{- \otimes -} \Ho(\der(\acal)) \xrightarrow{\Hom^\enh_{\Ho(\der(\acal))}([1_\acal], -)} \Ab
\]
is the counit of a self duality for $\Ho(\der(\acal))$ in $\Pr^L$. This is a consequence of the fact that $\Ho(\der(\acal))$ is a semisimple rigid symmetric monoidal Grothendieck abelian category.

We now prove that $p$ is faithful. Assume given a morphism $\mu$ in $\Ho(\Funct^L(\ccal, \Sp))$ such that $p(\mu) = 0$. Then $qG^*\mu = G^*p\mu = 0$, and since $q$ is an equivalence we have that $G^*\mu = 0$. Hence $\mu$ is a retract of $F^*G^*\mu = 0$, which means that $\mu = 0$. Combining this with the fact that $p$ is additive we deduce that $p$ is faithful.

Next we show that $p$ is fully faithful. Let $X, Y$ be objects in $\Ho(\Funct^L(\ccal, \Sp))$ and let $\nu: p(X) \rightarrow p(Y)$ be a morphism. We wish to show that $\nu$ lifts to a morphism $X \rightarrow Y$. We have a split retraction in the arrow category of $\Funct^\oplus(\ccal, \Ab)$ as follows:
\[
\begin{tikzcd}
p(X) \arrow{d}{\nu} \arrow{r}{s^*} & F^*G^*p(X) \arrow{d}{F^*G^*\nu} \arrow{r}{r^*} & p(X) \arrow{d}{\nu} \\
p(Y) \arrow{r}{s^*} & F^*G^*p(Y) \arrow{r}{r^*} & p(Y)
\end{tikzcd}
\]
The top and bottom row are the images under $p$ of the split retractions $\smash{X \xrightarrow{s^*} F^*G^*(X) \xrightarrow{r^*} X}$ and $\smash{Y \xrightarrow{s^*}F^*G^*(Y) \xrightarrow{r^*}Y}$. Since we have $\nu = r^* \circ (F^*G^*\nu) \circ s^*$, to prove that $\nu$ admits a lift it is enough to show that $F^*G^*\nu$ admits a lift along $p$ to a map $\xi: F^*G^*(X) \rightarrow F^*G^*(X)$. This may be done by setting $\xi = F^*q^{-1}G^*\nu$.

It remains to prove that $p$ is essentially surjective. Let $W$ be an object of $\Funct^\oplus(\ccal, \Ab)$. Then $W$ is a retract of $F^*G^*W = p F^*q^{-1}G^*W$. The fact that $p$ is surjective now follows from the fact that it is fully faithful and its source is idempotent complete.
\end{proof}

\begin{lemma}\label{lemma alcanza probar comp gen}
Let $\acal$ be a semisimple rigid symmetric monoidal Grothendieck abelian category and let $\ccal$ be an invertible $\der(\acal)$-linear presentable category. If $\Ho(\ccal)$ is semisimple then $\Ccal =\LMod_A(\der(\Acal))$ for some Azumaya algebra $A$ in $\Acal$.
\end{lemma}
\begin{proof}
The fact that $\Ho(\ccal)$ is semisimple implies  that $\ccal$ is compactly generated by those objects whose image in $\Ho(\ccal)$ is simple.  Let $\lbrace X_t \rbrace_{t \in T}$ be a generating set of compact objects for $\ccal$ with $[X_t]$ simple for all $t$. For each finite subset $\alpha \subseteq T$ let $\ccal_\alpha$ be the smallest full subcategory of $\ccal$ closed under colimits and the action of $\der(\acal)$, and containing $\lbrace X_t \rbrace_{t \in \alpha}$. Note that each $\ccal_\alpha$ inherits a $\der(\acal)$-linear structure from $\ccal$, and that $\ccal = \colim \ccal_\alpha$. Let $\delta$ be the image of $1_\acal$ under the unit map $\der(\acal) \rightarrow \ccal \otimes_{\der(\acal)} \ccal^\vee$. We have $\ccal = \colim \ccal_\alpha$ and therefore $\ccal \otimes_{\der(\acal)} \ccal^\vee = \colim \ccal_\alpha \otimes_{\der(\acal)} \ccal^\vee$. Applying lemma \ref{lemma filtered colimit of cg} we see that $\delta$ belongs to $\ccal_\alpha \otimes_{\der(\acal)} \ccal^\vee$ for some $\alpha$ and therefore the identity on $\ccal$ belongs to the image of the inclusion $\Funct_{\der(\acal)}(\ccal, \ccal_\alpha) \rightarrow \Funct_{\der(\acal)}(\ccal, \ccal)$. This implies that the inclusion of $\ccal_\alpha$ into $\ccal$ admits a section, so that $\ccal = \ccal_\alpha$. 

Changing $\alpha$ if necessary we may assume that $\alpha$ is minimal with the property that $\ccal = \ccal_\alpha$. Let $A$ be the opposite of the endomorphism algebra of $\bigoplus_{t \in \alpha} X_t$, so that we have a $\der(\acal)$-linear equivalence $\ccal = \LMod_A(\der(\acal))$. Note that any Azumaya algebra in $\der(\acal)$ which belongs to $\der(\acal)^\heartsuit$ is in fact an Azumaya algebra of $\acal$. Consequently, to prove the lemma it remains to show that $A$ belongs to $\acal$.  Let $t, t'$ be two distinct elements of $\alpha$. The minimality of $\alpha$ implies that if $Y$ is a compact projective object of $\acal$ and $n$ is an integer then $\Ext^0_{\ccal}(Y[ n] \otimes X_t , X_{t'}) = 0$, and therefore $\Hom^\enh_{\ccal}(X_t, X_{t'}) = 0$. It follows that $A = \prod_{t \in \alpha} \End^\enh_{\ccal}(X_t)$.  We may therefore reduce to showing that $\End^\enh_{\ccal}(X_t)$ belongs to $\acal$ for all $t$ in $\alpha$. Assume for the sake of contradiction that this does not hold. Then we may pick a nonzero integer $n$, a compact projective object $Y$ in $\acal$, and a nonzero map $Y \otimes X_t [ n] \rightarrow X_t$, making $X_t$  a summand of $Y \otimes X_t[n]$. It follows from this by induction that for every positive integer $m$ there is a nonzero map $Y^\otimes \otimes X_t[ mn] \rightarrow X_t$, and therefore $H_{mn}(A) \neq 0$. This contradicts the fact that $A$ is compact in $\der(\acal)$ (since it is an Azumaya algebra).
\end{proof}

\begin{proof}[Proof of theorem \ref{theorem stable}]
Applying proposition \ref{proposition dualizable is presentable} we see that $\Ccal$ and $\Ccal^\vee$ are presentable. By lemma \ref{lemma dual of ho} we have that $\Ho(\Ccal)$ and $\Ho(\Ccal^\vee)$ are spectral categories, and furthermore $\Ho(\ccal) = \Funct^L(\Ho(\ccal^\vee), \Ab)$. Applying lemma \ref{lemma dual is locally finitely generated} we see that $\Ho(\Ccal)$ is in fact semisimple.  The theorem now follows from lemma \ref{lemma alcanza probar comp gen}. 
\end{proof}


\ifx\inmain\undefined
\bibliographystyle{myamsalpha2}
\bibliography{References}
\fi


\section{Fully dualizable categories over commutative rings}\label{section G rings}

The main goal of this section is to supply proofs of theorems \ref{theorem principal 2 introduction}, \ref{teo introduction spectral}, and their variants. We may summarize our strategy in the $(1,1)$-categorical case as follows:
\begin{enumerate}[\normalfont (1)]
\item Using the results from \ref{subsection compactly assembled}, we have that if $\ccal$ is a fully dualizable $R$-linear cocomplete category then $\ccal$ is automatically Grothendieck abelian and has exact products. The main task is to show that $\ccal$ admits a compact projective generator \'etale locally on $\Spec(R)$.
\item  We first address the case when $R$ is a local Artinian commutative ring. This is done by induction on the length of $R$ as an $R$-module.   The case when $R$ has length $1$ (i.e., is a field) follows from the results of section \ref{section invertible semisimple}.  For the inductive step one shows that compact projective generators may be deformed along an elementary extension $R \rightarrow S$ of local Artinian commutative rings (see definition  \ref{definition elementary extension}).
\item  We then address the case when $R$ is a complete local Noetherian commutative ring with maximal ideal $\mathfrak{m}$. In this case one uses the result from step (2) to construct a compatible sequence of compact projective generators for $\Ccal \otimes_R R/\mathfrak{m}^n$, and in the limit an object $X$ of $\Ccal$. The full dualizability of $\ccal$ is used to show that $X$ is a generator for $\ccal$.  The Gabriel-Popescu theorem then shows that $\Ccal$ is a left exact localization of the category of modules over the endomorphism algebra of $X$, and we finally proceed  to show that this localization has trivial kernel.
\item The general case is proven using Popescu's smoothing theorem together with the result from step (3) and our main theorem from section \ref{section compactly assembled}.
\end{enumerate}

The $(1,1)$-categorical theorem is used to deduce the $(\infty,1)$-categorical theorem in the case of (classical) G-rings, by showing that any fully dualizable $R$-linear Grothendieck prestable category is the connective derived category of its heart. The general case is then proven by deforming compact generators along the Postnikov tower of $R$, with the clutching theorems from \cite{SAG} section 16.2 being applied for the inductive step.

As discussed in the introduction to the paper, we will deduce our main theorems from a general result that applies to graded categories as well. In this case, instead of working with $R$-linear cocomplete categories, we consider instead cocomplete categories linear over a base symmetric monoidal $R$-linear Grothendieck abelian or prestable category, which we assume to be rigid, generated under colimits by compact projective objects, proper over $R$, and with semisimple fibers on a dense subset of $\Spec(R)$.

 We begin this section in \ref{subsection proper} with some basic facts concerning proper $R$-linear Grothendieck abelian and prestable categories. In \ref{subsection completion} we collect a few results on the procedure of completion with respect to an ideal for objects of $R$-linear Grothendieck prestable categories, that will be needed when carrying out step (3) in our proof. 
 
 The proof of our main results is given in \ref{subsection fully dualizable 11} and \ref{subsection fully dualizable infty}. We then show in \ref{subsection rings of def} that these results extend beyond G-rings under additional compact generation hypothesis on $\ccal$ and $\ccal^\vee$.  We finish the section in \ref{subsection invertible stable over R} with a proof of theorem \ref{teo principal stable introduction}.


\subsection{Proper \texorpdfstring{$R$}{R}-linear categories}\label{subsection proper}

We will be interested in this section in properness for $R$-linear Grothendieck abelian and Grothendieck prestable categories generated under colimits by compact projective objects (corresponding to the cases $n = 1$ and $n = \infty$ of remark \ref{remark properness when comp gen}).  It follows from remark \ref{remark base change proper} that if $\ccal$ is a proper $R$-linear Grothendieck prestable category then $\ccal^\heartsuit$ is a proper $\pi_0(R)$-linear Grothendieck abelian category. The following proposition shows that the procedure of deriving Grothendieck abelian categories also preserves properness:

\begin{proposition}\label{prop proper clasica vs prestable}
Let $R$ be a commutative ring and let $\ccal$ be an $R$-linear Grothendieck abelian category. Assume that $\ccal$ is generated by compact projective objects and is proper over $R$. Then $\der(\ccal)_{\geq 0}$ is proper over $R$.
\end{proposition}
\begin{proof}
Let $X, Y$ be a pair of compact projective objects of $\ccal$. Then for each compact projective $R$-module $Z$ we have
\begin{align*}
\Hom_{\Mod_R^\cn}(Z, \Hom_{\der(\ccal)_{\geq 0}}^\enh(X, Y)) & = \Hom_{\der(\ccal)}(Z \otimes X, Y) \\ &= \Hom_{\ccal}(Z \otimes X, Y) 
\\ &= \Hom_{\Mod_R^\heartsuit}(Z, \Hom_{\ccal}^\enh(X, Y)).
\end{align*}
The above equivalence is functorial in $Z$ and therefore induces an isomorphism of $R$-modules $\Hom_{\der(\ccal)_{\geq 0}}^\enh(X, Y) = \Hom_{\ccal}^\enh(X, Y)$. The proof finishes by observing that an object of $\Mod_R^\heartsuit$ is compact projective if and only if it is compact projective inside $\Mod_R^\cn$.
\end{proof}

\begin{remark}\label{remark proper is derived}
Let $R$ be a commutative ring and let $\ccal$ be an $R$-linear Grothendieck prestable category. Assume that $\ccal$ is generated under colimits by compact projective objects and is proper over $R$. Then $\Hom_\ccal(X, Y)$ is $0$-truncated for each pair of compact projective objects of $\ccal$, and in particular we see that every compact projective object of $\ccal$ is $0$-truncated. It follows from this that there is an $R$-linear equivalence $\ccal = \der(\ccal^\heartsuit)$.
\end{remark}

In the presence of properness projective objects are automatically flat:

\begin{proposition}\label{prop flatness of compact projectives}
Let $R$ be a connective commutative ring spectrum and let $\ccal$ be an $R$-linear Grothendieck prestable category. Assume that $\ccal$ is generated under colimits compact projective objects and proper over $R$. Then every projective object of $\ccal$ is flat over $R$.
\end{proposition}
\begin{proof}
Since every projective object is a retract of a direct sum of compact projective objects, and flatness is preserved by retracts and direct sums it is enough to show that every compact projective object $X$ in $\ccal$ is flat over $R$. This amounts to showing that $M \otimes_R X$ is $0$-truncated for every $0$-truncated $R$-module $M$. To prove this it is enough to show that $\Hom^\enh_\ccal(Y, M \otimes_R X)$ is a $0$-truncated $R$-module for every compact projective object $Y$. This follows from the fact that  $\Hom^\enh_\ccal(Y, M \otimes_R X) = M \otimes_R \Hom^\enh_\ccal(Y, X)$ combined with the fact that dualizable $R$-modules are flat.
\end{proof}

We focus in the remainder of this section in the case where $R$ is a (classical) commutative ring. We will formulate our result for proper $R$-linear Grothendieck abelian categories generated by compact projective objects; it follows from remark \ref{remark proper is derived} that the Grothendieck prestable setting does not present extra generality in this case.

  We start by formulating a version of Nakayama's lemma:

\begin{proposition}\label{proposition nakayama}
Let $R$ be a local commutative ring with residue field $k$ and let $\ccal$ be an $R$-linear Grothendieck abelian category generated by compact projective objects and proper over $R$. Let $X$ be a finitely generated object of $\Ccal$. If  $X \otimes_R k = 0$ then $X = 0$.
\end{proposition}
\begin{proof}
Pick  an epimorphism $f: Y \rightarrow X$ with $Y$ compact projective. Passing to $R$-modules of maps from $Y$ we obtain an epimorphism
\[
\Hom^\enh_\ccal(Y, Y) \rightarrow \Hom^\enh_\ccal(Y, X) .
\]
 In particular, since the left hand side is a compact projective $R$-module we have that $\Hom^\enh_\ccal(Y, X)$ is a finitely generated $R$-module. We now have
 \[
\Hom^\enh_\ccal(Y, X) \otimes_R k   =  \Hom^\enh_\ccal(Y, X \otimes_R k) = 0.
\]
An application of Nakayama's lemma shows that $\Hom^\enh_\ccal(Y, X) = 0$. Hence $f = 0$ and therefore $X = 0$, as desired.
\end{proof}

\begin{corollary}\label{coro check epi on fiber}
Let $R$ be a local commutative ring with residue field $k$ and let $\ccal$ be an $R$-linear Grothendieck abelian category generated by compact projective objects and proper over $R$. Let $f: X \rightarrow Y$ be a morphism in $\ccal$, with $Y$ finitely generated. If $f \otimes_R k$ is an epimorphism then $f$ is an epimorphism.
\end{corollary}
\begin{proof}
Follows from proposition \ref{proposition nakayama} since $\operatorname{Coker}(f) \otimes_R k = \operatorname{Coker}(f \otimes_R k) = 0$.
\end{proof}

\begin{corollary}\label{coro lift equivalences of comp projectives}
Let $R$ be a local commutative ring with residue field $k$ and let $\ccal$ be an $R$-linear Grothendieck abelian category. Assume that $\ccal$ is generated by compact projective objects and proper over $R$. If $X$ and $Y$ are compact projective objects such that $X \otimes_R k$ and $Y \otimes_R k$ are isomorphic, then $X$ and $Y$ are isomorphic. Furthermore, any morphism $f: X \rightarrow Y$ such that $f \otimes_R k$ is an isomorphism, is itself an isomorphism.
\end{corollary}
\begin{proof}
Since $X$ is projective any isomorphism between $X \otimes_R k$ and $Y \otimes_R k$ may be lifted to a morphism $f: X \rightarrow Y$. It remains to show that in this case $f$ is an isomorphism. The map $f$ is an epimorphism by corollary \ref{coro check epi on fiber}. In particular, the kernel of $f$ is finitely generated as well. By proposition \ref{prop flatness of compact projectives} we have that $\Ker(f) \otimes_R k = \Ker(f \otimes_R k) = 0$. Another application of proposition \ref{proposition nakayama} shows that $f$ is a monomorphism, and therefore an isomorphism.
\end{proof}

We now prove that compact projective objects in fibers extend to \'etale neighborhoods.

\begin{proposition}\label{prop extend compact projectives}
Let $R$ be a commutative ring and let $\ccal$ be an $R$-linear Grothendieck abelian category. Assume that $\ccal$ is generated   by compact projective objects and proper over $R$. Let $\mathfrak{p}$ be a prime ideal of $R$ with residue field $k$ and let $X$ be a compact projective object of $\ccal \otimes_R k$.
\begin{enumerate}[\normalfont (1)]
\item If $R$ is complete local and $\mathfrak{p}$ is maximal then there exists a compact projective object in $\ccal$ whose image in $\ccal \otimes_R k$ recovers $X$.
\item There exists an \'etale morphism $R \rightarrow R'$ with $R' \otimes_R k \neq 0$ and a compact projective object in $\ccal \otimes_R R'$ whose image in $\ccal \otimes_R (k \otimes_R R')$ recovers $X \otimes_k (k \otimes_R R')$.
\end{enumerate}
\end{proposition}
\begin{proof}
We first address part (1). Denote $\ccal_k = \ccal \otimes_R k$. Since $\ccal_k$ is generated under colimits by objects of the form $Y \otimes_R k$ we may pick a compact projective object $Y$ in $\ccal$ such that  $X$ is a retract of $Y \otimes_R k$. Let $r$ be an idempotent endomorphism on $Y \otimes_R k$ with image $X$. The lemma will follow if we are able to show that $r$ admits a lift to an idempotent endomorphism of $Y$.

 Since $Y$ is compact projective the extension of scalars map $\End^\enh_\Ccal(Y) \rightarrow \End^\enh_{\Ccal_k}(Y_k)$ is equivalent to the  canonical map of $R$-algebras $\End^\enh_\ccal(Y) \rightarrow \End^\enh_\ccal(Y) \otimes_R k$. Set $A = \End^\enh_\ccal(Y)$, so that our map is given by the quotient $A \rightarrow A/\mathfrak{m}$.   We may identify $r$ with an idempotent element of $A/\mathfrak{m}A$, and our task is to show that this lifts to an idempotent element of $A$. Consider the sequence of square zero extensions $A/\mathfrak{m}A \leftarrow A/\mathfrak{m}^2A \leftarrow A/\mathfrak{m}^3 A \leftarrow \ldots$. We may lift $r$ to a compatible sequence of idempotents $r = r_1, r_2, \ldots$. Since $A$ is a compact projective $R$-module we have an isomorphism $A = \lim A/\mathfrak{m}^nA$ and hence the sequence induces in the limit the desired idempotent in $A$ lifting $r$.
 
 We now prove part (2).  As before, pick a compact projective object $Y$ in $\ccal$ such that $X$ is a retract of $Y \otimes_R k$ and let $A$ be the $R$-algebra of endomorphisms of $Y$. Then we have an $R$-linear fully faithful embedding $\LMod_A(\Mod^\heartsuit_R) \rightarrow \ccal$. Since $X$ belongs to $\LMod_A(\Mod^\heartsuit_R) \otimes_R k$ it is enough to prove the proposition in the case $\ccal = \LMod_A(\Mod^\heartsuit_R)$.  Write $R$ as a filtered colimit of a diagram of commutative rings $R_\alpha$ of finite type over $\ZZ$. Since the $R$-module underlying $A$ is dualizable, there exists an index $\alpha$ and an $R_\alpha$-algebra $A_\alpha$ whose underlying $R_\alpha$-module is dualizable, and such that $A = R \otimes_{R_\alpha} A_\alpha$.
 
 We assume without loss of generality that $\alpha$ is initial. For each index $\beta$ let $A_\beta =R_\beta \otimes_{R_\alpha} A_\alpha$, and let $k_\beta$ be the residue field of $R_\beta$ at the preimage of $\mathfrak{p}$ under the map $R_\beta \rightarrow R$. We have $k = \colim k_\beta$, so the idempotent of $A \otimes_R k$ with image $x$ lifts to an idempotent of $A_\beta \otimes_{R_\beta} k_\beta$ for some $\beta$. Replacing $R$ by $R_\beta$ and $\ccal$ by $\LMod_{A_\beta}(\Mod_{R_\beta}^\heartsuit)$ we may now assume that $R$ is of finite type over $\ZZ$, and in particular a G-ring.
 
By part (1) we may find a compact projective object in $\ccal \otimes_R R^\wedge_{\mathfrak{p}}$ whose image in $\ccal \otimes_R k$ recovers $X$. Combining theorem \ref{theorem lift strongly compacts} with Popescu's smoothing theorem we may find a smooth $R$-algebra $S$ such that $S \otimes_R k \neq 0$ and a compact projective object $Z$ in $\ccal \otimes_R S$ whose image in $\ccal \otimes_R (S \otimes_R k)$ recovers $X \otimes_k (S \otimes_R k)$. Pick a morphism of commutative rings $S \rightarrow R'$ such that the induced map $R \rightarrow R'$ is \'etale and $R' \otimes_R k \neq 0$.  Then $Z \otimes_S R'$ is a compact projective object of $\ccal \otimes_R R'$ with the desired property.
\end{proof}

\begin{corollary}\label{corollary extend equivalence etale locally}
Let $R$ be a commutative ring and let $\ccal, \ccal'$ be $R$-linear Grothendieck abelian categories. Assume that $\ccal$ and $\ccal'$ are generated   by compact projective objects and fully dualizable over $R$. Let $\mathfrak{p}$ be a prime ideal of $R$ with residue field $k$. If $\ccal \otimes_R k = \ccal' \otimes_R k$ then there exists an \'etale $R$-algebra $R'$ such that $R' \otimes_R k \neq 0$ and an equivalence $\ccal \otimes_R R' = \ccal' \otimes_R R'$.
\end{corollary}
\begin{proof}
We have that $\ccal$ and $\ccal'$ are dualizable over $R$, and furthermore $\Funct(\ccal, \ccal')$ and $\Funct(\ccal', \ccal)$ are Grothendieck abelian categories generated by compact projective objects and proper over $R$. Let $F: \ccal \otimes_R k \rightarrow \ccal' \otimes_R k$ and $G: \ccal' \otimes_R k \rightarrow \ccal \otimes_R k$ be inverse equivalences. Then $F$ and $G$ define compact projective objects in $\Funct(\ccal, \ccal') \otimes_R k$ and $\Funct(\ccal', \ccal) \otimes_R k$, respectively. An application of proposition \ref{prop extend compact projectives} shows that, after replacing $R$ with an \'etale $R$-algebra if necessary, we may find compact projective lifts $\overline{F}: \ccal \rightarrow \ccal'$ and $\overline{G}: \ccal' \rightarrow \Ccal$ for $F$ and $G$. The fact that $\ccal$ and $\ccal'$ are proper implies that $\overline{F} \circ \overline{G}$ and $\overline{G} \circ \overline{F}$ are also compact projective. Since $F \circ G$ and $G \circ F$ are identities, it follows from corollary \ref{coro lift equivalences of comp projectives} that $\overline{F} \circ \overline{G}$ and $\overline{G} \circ \overline{F}$ become isomorphic to identities after passing to a localization of $R$.
\end{proof}

\begin{corollary}\label{coro smooth and proper locally trivial}
Let $R$ be a commutative ring and let $A$ be a smooth and proper algebra in $\Mod^\heartsuit_R$. Let $\mathfrak{p}$ be a prime ideal of $R$ with residue field $k$. Then there exists an \'etale $R$-algebra $R'$ such that $R' \otimes_R k \neq 0$ with the property that $A \otimes_R R'$ is Morita equivalent to  a finite product of copies of $R'$.
\end{corollary}
\begin{proof}
Follows from corollary \ref{corollary extend equivalence etale locally} together with the fact that every separable algebra over $k$ becomes Morita equivalent to a finite product of copies of $k$ after passage to a finite separable field extension.
\end{proof}

\begin{remark}
The conclusion of proposition \ref{prop extend compact projectives} does not hold if we replace \'etale morphisms with Zariski morphisms: consider for instance $\ccal = \Mod^\heartsuit_{\ZZ[x,x^{-1}]}$ as a category linear over $\ZZ[x^2, x^{-2}]$.
\end{remark}

In what follows we will be working with proper $R$-linear categories with semisimple fibers. We will need the following:

\begin{proposition}\label{prop subcat closed under passage to subobjects}
Let $R$ be a commutative ring and let $\ccal$ be an $R$-linear Grothendieck abelian category. Assume that $\ccal$ is generated by compact projective objects, proper over $R$, and that the set of points $x$ in $\Spec(R)$ such that $\ccal \otimes_R \kappa(x)$ is semisimple is dense in the Zariski topology. Let $X$ be a compact projective object of $\ccal$. Then the full subcategory of $\ccal$ generated under colimits by $X$ is closed under passage to subobjects.
\end{proposition}

The proof of proposition \ref{prop subcat closed under passage to subobjects} requires some preliminary lemmas.

\begin{lemma}\label{lemma split after cover}
Let $R$ be a commutative ring and let $\ccal$ be an $R$-linear Grothendieck abelian category. Assume that $\ccal$ is generated by compact projective objects, proper over $R$, and fix a point in $\Spec(R)$ with residue field $k$ such that $\ccal \otimes_R k$ is semisimple. Let $X$ be a compact projective object of $\ccal$. Then there exists an \'etale morphism $R \rightarrow R'$ such that $R' \otimes_R k = k'$ is a field and a finite sequence of compact projective objects $Y_i$ in $\ccal \otimes_R R'$  such that $X \otimes_R R' = \bigoplus Y_i$ and $Y_i \otimes_R' k'$ is a simple object of $\ccal \otimes_R k'$ for every $i$.
\end{lemma}
\begin{proof}
For every finite separable field extension $F$ of $k$ let $n_F$ be the number of simple summands of $X \otimes_R F$ in $\ccal \otimes_R F$. This is bounded by the dimension of $\End^\enh_{\ccal \otimes_R F}(X \otimes_R F)$, which itself agrees with the dimension of $\End^\enh_{\ccal \otimes_R k}(X \otimes_R k)$. We may therefore choose $F$ such that $n_F$ is maximal. This extension has the property that the simple summands of $X \otimes_R F$ remain simple under further separable finite extensions.  Let $R \rightarrow S$ be an \'etale morphism such that $S \otimes_R k = F$. Write $X \otimes_R F$ as a sum of simple objects $S_i$. Applying proposition \ref{prop extend compact projectives} we may find an \'etale morphism $S \rightarrow R'$ with the property that $R' \otimes_R k \neq 0$, and a sequence of compact projective objects $Y_i$ in $\ccal \otimes_R R'$ such that $Y_i \otimes_R k = S_i \otimes_{F} (F \otimes_S R')$. Passing to a localization of $R'$ we may further assume that $R' \otimes_R k = k'$ is a finite separable field extension of $k$. Note that then $X \otimes_R k' = \bigoplus Y_i \otimes_{R} k $ is a decomposition in simples. It now follows from corollary \ref{coro lift equivalences of comp projectives} that after replacing $R'$ with a localization if necessary we have $X \otimes_R R' = \bigoplus Y_i$.
\end{proof}

\begin{lemma}\label{lemma epi after cover}
Let $R$ be a commutative ring and let $\ccal$ be an $R$-linear Grothendieck abelian category. Assume that $\ccal$ is generated by compact projective objects, proper over $R$, and fix a point in $\Spec(R)$ with residue field $k$ such that $\ccal \otimes_R k$ is semisimple. Let $X$ be a compact projective object of $\ccal$ and let $Y$ be a finitely generated subobject of $X$. Then there exists an \'etale morphism $R \rightarrow R'$ such that $R' \otimes_R k \neq 0$ with the property that that $Y \otimes_R R'$ receives an epimorphism from a finite direct sum of copies of $X \otimes_R R'$.
\end{lemma}
\begin{proof}
By lemma \ref{lemma split after cover} we may assume (after passing to a cover if necessary) that $X = \bigoplus X_i$ for a finite sequence of compact projective objects such that $S_i = X_i \otimes_R k $ is simple for all $i$. Pick an epimorphism $\rho: Z \rightarrow Y$ with $Z$ compact projective. Write $Z \otimes_R k = Q \oplus \bigoplus S_{i_\alpha}$ where $\Hom_{\ccal \otimes_R k}(S_i, Q) = 0$ for all $i$. Let $W = \bigoplus X_{i_\alpha}$. Since $W \otimes_R k$ is a direct summand of $Z \otimes_R k$ and $Z$ and $W$ are projective we may find maps $f:Z \rightarrow W$ and $g: W \rightarrow Z$ such that $(f \circ g) \otimes_R k$ is the identity. By corollary \ref{coro lift equivalences of comp projectives}, after passing to a localization of $R$ we may assume that $f \circ g$ is an isomorphism. We may therefore write $Z = W \oplus Z'$, where $Z' \otimes_R k = Q$. Note that $\Hom^\enh_\ccal(Z', X) \otimes_R k = \Hom^\enh_{\ccal \otimes_R k}(Q, X \otimes_R k) = 0$. After passing to a localization of $R$ we may assume that $\Hom^\enh_\ccal(Z', X) = 0$.  We conclude that the image of $\rho$ is equivalent to the image of $\rho \circ g$, which receives an epimorphism from a direct sum of copies of $X$ since $W$ does.
\end{proof}

\begin{proof}[Proof of proposition \ref{prop subcat closed under passage to subobjects}]
Let $\ccal'$ be the full subcategory of $\ccal$ generated by $X$. Let $Y$ be an object of $\ccal'$ and let $Z$ be a subobject of $Y$ inside $\ccal$. Our task is to show that $Z$ belongs to $\ccal'$. Since Grothendieck abelian categories satisfy \'etale descent, it suffices to prove this after passage to an \'etale cover of $\Spec(R)$. 

Since $Z$ is a colimit of finitely generated subobjects we may reduce to the case when $Z$ is finitely generated. Pick an epimorphism $S \otimes X \rightarrow Y$ where $S$ is a set. Then $Z$ receives an epimorphism from a subobject $Z'$ of $T \otimes X$, where $T$ is a finite subset of $S$. Replacing $X$ with $T \otimes X$ we may reduce to the case where $Z'$ is a subobject of $X$. By lemma \ref{lemma epi after cover} we may, after passing to a cover, assume that $Z'$ admits an epimorphism from a finite direct sum of copies of $X$. Replacing $X$ with a finite direct sum of copies of $X$ we may assume the existence of an epimorphism $X \rightarrow Z'$, which implies that there exists an epimorphism $X \rightarrow Z$. Let $W$ be its kernel. Write $W = \colim W_\alpha$ where $W_\alpha$ are finitely generated subobjects of $W$. Then  $Z = \colim X/W_\alpha$, so it suffices to prove that $X/W_\alpha$ belongs to $\ccal'$ for each $\alpha$. Replacing $Z$ by $X/W_\alpha$ and $W$ by $W_\alpha$ we may now assume that $W$ is finitely generated. Passing to a cover if necessary we may, thanks to lemma \ref{lemma epi after cover}, find an epimorphism $U \otimes X \rightarrow W$ for some finite set $U$. Then $Z$ is the cokernel of the map $U \otimes X \rightarrow X$, so it belongs to $\ccal'$, as desired.
\end{proof}

\begin{corollary}\label{coro C is locally noetherian}
Let $R$ be a Noetherian commutative ring and let $\ccal$ be an $R$-linear Grothendieck abelian category. Assume that $\ccal$ is generated by compact projective objects, proper over $R$, and that the set of points $x$ in $\Spec(R)$ such that $\ccal \otimes_R \kappa(x)$ is semisimple is dense in the Zariski topology. Then $\ccal$ is locally Noetherian. In particular, every finitely generated object of $\ccal$ admits a resolution by compact projective objects.
\end{corollary}
\begin{proof}
It suffices to prove that every compact projective object $X$ in $\ccal$ is Noetherian. By proposition \ref{prop subcat closed under passage to subobjects} it is enough to show that $X$ is a Noetherian object of the full subcategory $\ccal'$ of $\ccal$ generated  by $X$. This follows from the fact that the functor $\Hom^\enh_\ccal(X, -): \ccal \rightarrow \Mod_R^\heartsuit$ is conservative when restricted to $\ccal'$ and sends subobjects of $X$ to subobjects of the Noetherian $R$-module $\Hom^\enh_\ccal(X, X)$.
\end{proof}


\subsection{Completions and prestable categories}\label{subsection completion}

We now discuss some properties of the procedure of completion of objects of $R$-linear Grothendieck prestable categories.

\begin{proposition}\label{prop commute geom realiz and inverse limits}
Let $\Ccal$ be a Grothendieck prestable category, separated and such that products in $\Sp(\Ccal)$ are t-exact. Let $\ccal^{\NN^\op}$ be the category of inverse sequences in $\ccal$. Let $(X_{n, \bullet})$ be a semisimplicial object in $\ccal^{\NN^\op}$, where here we denote by $n$ the sequence index and by $\bullet$  the semisimplicial direction. Assume that for every $n, m \geq 0$ the map $X_{n+1, m} \rightarrow X_{n, m}$ induces an epimorphism on $H_0$. Then the limit functor  $\ccal^{\NN^\op} \rightarrow \ccal$ preserves the colimit of $(X_{n, \bullet})$.
\end{proposition}
\begin{proof}
We wish to show that the map $\xi: | \lim X_{n, \bullet} | \rightarrow \lim |X_{n, \bullet}|$ is an isomorphism. Since $\ccal$ is separated we may reduce to showing that $\xi$ induces an isomorphism on homologies. Fix $t \geq 0$. We will show that $\xi$ induces an isomorphism on $H_s$ for all $s < t-1$. 

Let $\Delta_{\text{inj}}$ be the wide subcategory of $\Delta$ on the injective maps, and let $\Delta_{\text{inj}, \leq t}$ be the full subcategory of $\Delta_{\text{inj}}$ on the simplices of dimension at most $t$. We have a commutative square
\[
\begin{tikzcd}
\colim_{\Delta^\op_{\text{inj}, \leq t}} \lim X_{n, \bullet} \arrow{r}{\xi_t} \arrow{d}{\nu_1} & \arrow{d}{\nu_2} \lim \colim_{\Delta^\op_{\text{inj}, \leq t}} X_{n,\bullet}\\
{| \lim X_{n, \bullet} |} \arrow{r}{\xi}& \lim |X_{n, \bullet}|.
\end{tikzcd}
\]
Since the inclusion $\Delta_{\text{inj}, \leq t} \rightarrow \Delta_{\text{inj}}$ is $t$-initial, we have that $\nu_1$ is an isomorphism on $H_s$ for all $s < t$. Furthermore, $\nu_2$ is an inverse limit of maps with that property, and since products in $\Sp(\ccal)$ are t-exact we have that $\nu_2$ is an isomorphism on $H_s$ for all $s < t-1$. To prove our claim it now suffices to show that $\xi_t$ is an isomorphism.

Our assumptions imply that the maps $ \colim_{\Delta^\op_{\text{inj}, \leq t}} X_{n+1,\bullet} \rightarrow  \colim_{\Delta^\op_{\text{inj}, \leq t}} X_{n,\bullet}$ induce epimorphisms on $H_0$. It follows from this together with the fact that products in $\Sp(\ccal)$ are t-exact that the limits $\lim X_{n,\bullet}$ and $\lim \colim_{\Delta^\op_{\text{inj}, \leq t}} X_{n,\bullet}$ are preserved  by the inclusion into $\Sp(\Ccal)$. The fact that $\xi_t$ is an isomorphism is now a consequence of the fact that $\Delta^\op_{\text{inj}, \leq t}$ is a finite category.
\end{proof}

\begin{corollary}\label{coro tensor preserves}
Let $R$ be a connective $E_\infty$-ring and let $\Ccal$ be an $R$-linear Grothendieck prestable category, separated and such that products in $\Sp(\Ccal)$ are t-exact. Let $X_0 \leftarrow X_1 \leftarrow X_2 \leftarrow \ldots$ be a sequence in $\ccal$ whose transitions induce epimorphisms on $H_0$. Let $M$ be an almost finitely presented  connective $R$-module. Then the map $M \otimes (\lim X_n) \rightarrow \lim (M \otimes X_n)$ is an isomorphism. 
\end{corollary}
\begin{proof}
Since $M$ is almost finitely presented there exists a simplicial resolution $M_\bullet$ of $M$ by compact projective $R$-modules. The result follows from proposition \ref{prop commute geom realiz and inverse limits} applied to the simplicial sequence $(M_\bullet \otimes X_n)$.
\end{proof}
 
 \begin{remark}
Let $R_0 \leftarrow R_1 \leftarrow R_2 \leftarrow \ldots $ be a sequence of connective $E_\infty$-rings and let $R$ be a connective $E_\infty$-ring equipped with a map $R \rightarrow \lim R_n$. Let $\Ccal$ be an $R$-linear Grothendieck prestable category. Then we have a sequence of $R$-linear Grothendieck prestable categories
\[
\Ccal \rightarrow \ldots \Ccal \otimes_R R_2 \rightarrow \Ccal \otimes_R R_1 \rightarrow \Ccal \otimes_R R_0
\]
which induces an $R$-linear functor $p: \Ccal \rightarrow \lim \Ccal \otimes_R R_n$. We think about objects in $\lim \Ccal \otimes_R R_n$ as compatible sequences $(X_n)$ with $X_n$ in $\ccal \otimes_R R_n$ for all $n \geq 0$. The functor $p$ has a right adjoint $p^R:  \lim \Ccal \otimes_R R_n \rightarrow \Ccal$ that sends a sequence $(X_n)$ to $\lim X_n$ (where here we regard each $X_n$ as an object of $\ccal$ via restriction of scalars). Note that in general $p^R$ does not preserve colimits and does not commute with the action of $\Mod_R^\cn$.
\end{remark}

\begin{corollary}\label{coro completes right adjoint 1}
Let $R_0 \leftarrow R_1 \leftarrow R_2 \leftarrow \ldots $ be a sequence of connective $E_\infty$-rings and let $R$ be a connective $E_\infty$-ring equipped with a map $R \rightarrow \lim R_n$. Assume that for every $n \geq 0$ the transition $R_{n+1} \rightarrow R_n$ induces an epimorphism on $\pi_0$. Let $\Ccal$ be an $R$-linear Grothendieck prestable category, separated and such that products in $\Sp(\Ccal)$ are t-exact. Then the right adjoint to the functor $p: \Ccal \rightarrow \lim \Ccal \otimes_R  R_n$ preserves geometric realizations and commutes with the action of the full subcategory of $\Mod_R^\cn$ on the almost finitely presented connective $R$-modules.
\end{corollary}
\begin{proof}
This is a direct consequence of proposition \ref{prop commute geom realiz and inverse limits} and corollary \ref{coro tensor preserves}.
\end{proof}

\begin{proposition}\label{prop completes ffff}
Let $R_0 \leftarrow R_1 \leftarrow R_2 \leftarrow \ldots $ be a sequence of connective $E_\infty$-rings and let $R$ be a connective $E_\infty$-ring equipped with a map $R \rightarrow \lim R_n$.  Assume the following:
\begin{itemize}
\item For every $n \geq 0$ the transition $R_{n+1} \rightarrow R_n$ induces an epimorphism on $\pi_0$.
\item The map from $R$ to the pro-$E_\infty$-ring defined by the sequence $(R_n)$ is an epimorphism in the category of pro-$E_\infty$-rings.
\item $R_n$ is almost finitely presented as an $R$-module for all $n$.
\end{itemize}
 Let $\ccal$ be an $R$-linear Grothendieck prestable category, separated and such that products in $\Sp(\Ccal)$ are t-exact. Then the right adjoint to the functor $p: \ccal \rightarrow \lim \ccal \otimes_R R_n$ is fully faithful.
\end{proposition}
\begin{proof}
We will prove the lemma by showing that the counit of the adjunction is an isomorphism. Let $(X_n)$ be an object of $\lim \ccal \otimes_R R_n$. Our goal is to prove that for every $s \geq 0$ the map $(\lim X_n) \otimes_R R_s \rightarrow X_s$ is an isomorphism. Our assumptions imply that $R_s$ is the limit  in the category of pro-$E_\infty$-rings of the sequence $R_s \otimes_R R_n$. It follows that the pro-object associated to the sequence $(R_s \otimes_R R_n)$ is equivalent to the constant pro-object on $R_s$, and therefore the projection $q: \ccal \otimes_R R_s \rightarrow \lim \ccal \otimes_R R_n \otimes_{R} R_s$ is an isomorphism. 

Consider now the object $(X_n \otimes_{R} R_s)$ in $\lim \ccal \otimes_R R_n \otimes_{R} R_s$. The fact that $q$ is an isomorphism implies that  the map $(\lim (X_n \otimes_{R} R_s)) \otimes_{R_s} (R_s \otimes_R R_s) \rightarrow X_s \otimes_{R} R_s$ is an isomorphism.  Since $R_s$ is almost finitely presented as an $R$-module there exists a simplicial $R$-module with colimit $R_s$  which is levelwise compact projective. Combining this with corollary \ref{coro tensor preserves} we obtain an equivalence $(\lim X_n) \otimes_{R} R_s = \lim (X_n \otimes_R R_s)$, so it follows that the map $(\lim X_n) \otimes_R (R_s \otimes_R R_s) \rightarrow X_s \otimes_R R_s$ is an isomorphism. Tensoring with $R_s$ over $R_s \otimes_R R_s$ we conclude that the map $(\lim X_n) \otimes R_s \rightarrow X_s$ is an isomorphism, as desired.
\end{proof}

\begin{proposition}\label{prop factors through pi0}
Let $R$ be a commutative ring and $x_1, \ldots, x_t$ be a finite sequence of elements of $R$. Consider for each $n \geq 1$ the commutative ring spectrum 
\[
R_n = R \otimes_{\ZZ[x_1, \ldots, x_t]}  \ZZ[x_1, \ldots, x_t]/(x_1^n,x_2^n \ldots, x_t^n).
\] Then for each $n \geq 1$ the canonical map $f_{2n, n}: R_{2n} \rightarrow R_n$ factors through $\pi_0(R_{2n})$.
\end{proposition}
\begin{proof}
We argue by induction on $t$. We consider first the case when $t = 1$, so that the sequence consists of a single element $x$. Let $F$ be the free commutative ring spectrum on $\Sigma (\pi_1(R_{2n}))$. Then the map $R_{2n} \rightarrow \pi_0(R_{2n})$ factors as a composition
\[
R_{2n} \rightarrow R_{2n}\otimes_{F} \SS \rightarrow \pi_0(R_{2n})
\]
where here $\SS$ denotes the sphere spectrum and the map $F \rightarrow \SS$ is zero on generators. The second map above is $2$-connective, and since $R_n$ is $1$-truncated it is enough to prove that $f_{{2n}, n}$ factors through $R_{2n}\otimes_{F} \SS$. This amounts to showing that $f_{2n, n}$ induces the zero map on $\pi_1$. Unwinding the definitions, we have that $\pi_1(f_{2n, n})$ is given by the map
\[
\Ker(x^{2n}: R \rightarrow R) \xrightarrow{x^n} \Ker(x^{n}: R \rightarrow R) 
\]
which is zero, as desired.

Assume now that $t > 1$ and that the proposition is known for all $s < t$. For each pair of positive integers $n, m$ let 
\[
R_{n, m} =  R \otimes_{\ZZ[x_1, \ldots, x_t]}  \ZZ[x_1, \ldots, x_t]/(x_1^n, x_2^n \ldots, x_{t-1}^n, x_t^m)
\]
and
\[
R'_{n, m} = R/(x_1^n, \ldots, x_{t-1}^n) \otimes_{\ZZ[x_t]} \ZZ[x_t]/(x_t^m).
\]
 We have a commutative square of commutative ring spectra
\[
\begin{tikzcd}
R_{2n, 2n} \arrow{d}{} \arrow{r}{} & \arrow{d}{} R_{n, 2n}  \\
R_{2n, n} \arrow{r}{} & R_{n, n} .
\end{tikzcd}
\]
Our task is to show that the diagonal map factors through $\pi_0(R_{2n, 2n})$.  Applying our inductive hypothesis to the sequence $x_1, \ldots, x_{t-1}$ we obtain a commutative diagram
\[
\begin{tikzcd}
R_{2n, 2n} \arrow{d}{} \arrow{r}{} & R'_{2n, 2n} \arrow{d}{} \arrow{r}{} &  \arrow{d}{} R_{n, 2n}  \\
R_{2n, n} \arrow{r}{} & R'_{2n, n} \arrow{r}{} & R_{n, n} .
\end{tikzcd}
\]
Our result now follows from another application of the inductive hypothesis to show that the map $R'_{2n, 2n} \rightarrow R'_{2n, n}$ factors through $\pi_0(R'_{2n, 2n}) = \pi_0(R_{2n, 2n})$.
\end{proof}

\begin{corollary}\label{coro fundamental completion}
Let $R$ be a commutative ring and let $I$ be an ideal in $R$ generated by elements $x_1, \ldots, x_t$. Then the pro-$E_\infty$-ring spectrum defined by the sequence $(R_n)$ from proposition \ref{prop factors through pi0} is equivalent to the one defined by the sequence $(R/I^n)$.
\end{corollary}

\begin{remark}\label{remark base change pro object}
Let $R \rightarrow S$ be a morphism of commutative rings and let $I$ be a finitely generated ideal of $R$. Then it follows from corollary \ref{coro fundamental completion} that the pro-$E_\infty$-ring spectrum defined by the sequence $(S \otimes_R R/I^n)$ is equivalent to the one defined by the sequence $(S/(SI)^n)$. In other words, this assignment of pro-$E_\infty$-ring spectra to pairs of a ring and a finitely generated ideal is preserved by base change.

Specializing this to the case $S = R/I^m$ for some $m \geq 1$ shows that the pro-$E_\infty$-ring defined by the sequence $(R/I^m \otimes_R R/I^n)$ is equivalent to $R/I^m$. It follows from this that the morphism $R \rightarrow \lim R/I^n$ induces an epimorphism of pro-$E_\infty$-ring spectra. In particular, if we assume that $R/I^n$ is an almost finitely presented $R$-module for all $n$ (which for instance holds whenever $R$ is Noetherian) then the sequence $R_n = R/I^n$ satisfies the conditions of proposition \ref{prop completes ffff}.
\end{remark}

If $R$ is a Noetherian commutative ring and $I \subseteq R$ is an ideal, then it follows from corollary \ref{coro completes right adjoint 1} that $M \otimes_R^L R^\wedge_I = \lim M \otimes^L_R R/I^n$ for every finitely generated $R$-module $M$. Combined with the fact that $R^\wedge_I$ is a flat $R$-module this implies that $\lim \Tor^R_s(M, R/I^n) = 0$ for all $s \geq 1$. A stronger claim is in fact true: the pro-$R$-module defined by the sequence $\Tor^R_s(M, R/I^n)$ vanishes for all $s \geq 1$. The following proposition generalizes this fact:
 
 \begin{proposition}\label{prop pro object 0 truncated}
 Let $R$ be a Noetherian commutative ring and $I\subseteq R$ be an ideal. Let $\ccal$ be an $R$-linear Grothendieck abelian category. Assume that $\ccal$ is generated by compact projective objects, proper over $R$, and that the set of points $x$ in $\Spec(R)$ such that $\ccal \otimes_R \kappa(x)$ is semisimple is dense in the Zariski topology. Let $X$ be a finitely generated object of $\ccal$. Then for every $s \geq 1$ the pro-object of $\ccal$ defined by the sequence $(\Tor_s(R/I^n, X))$ vanishes.
  \end{proposition}
\begin{proof}
Pick an epimorphism $Y \rightarrow X$ with $Y$ compact projective, and let $\Ccal'$ be the full subcategory of $\ccal$ generated by $Y$. Recall from proposition \ref{prop subcat closed under passage to subobjects} that $\ccal'$ is closed under passage to subobjects in $\ccal$. In particular the inclusion $\ccal' \rightarrow \ccal$ is left exact, so it commutes with $\Tor$. Replacing $\ccal$ with $\ccal'$ we may now assume that $Y$ is a compact projective generator for $\ccal$. In this case it is sufficient to prove that for every $s \geq 1$ the pro-$R$-module defined by the sequence 
\[
\Hom^\enh_\ccal(Y, \Tor_s(R/I^n, X)) = \Tor^R_s(R/I^n, \Hom^\enh_\ccal(Y, X))
\]
vanishes.  Replacing $\ccal$ by $\Mod_R^\heartsuit$ and $X$ by $\Hom^\enh_\ccal(Y, X)$ we may now reduce to the case when $\ccal = \Mod_R^\heartsuit$. The property that our pro-object vanishes is preserved by extensions in $X$, so it is enough to consider the case when $X = R/J$ for some ideal $J$. In this case the claim follows from remark \ref{remark base change pro object}.
\end{proof} 

 
\subsection{Fully dualizable \texorpdfstring{$(1,1)$}{(1,1)}-categories}\label{subsection fully dualizable 11}

The following is our main theorem concerning fully dualizable $R$-linear $(1,1)$-categories:

\begin{theorem}\label{theo abelian con coefficients}
Let $R$ be a G-ring and let $\acal$ be a symmetric monoidal $R$-linear Grothendieck abelian category. Assume the following:
\begin{itemize}
\item $\acal$ is rigid and generated by compact projective objects.
\item $\acal$ is proper over $R$.
\item The set of points $x$ in $\Spec(R)$ such that $\acal \otimes_R \kappa(x)$ is semisimple is dense in the Zariski topology.
\end{itemize}
Let $\ccal$ be a fully dualizable $\acal$-linear cocomplete category. Then there exists a faithfully flat \'etale morphism of commutative rings $R \rightarrow R'$ and a smooth and proper algebra $A$ in $\acal \otimes_R R'$  such that $\Ccal \otimes_{R} R'$ is equivalent to $\LMod_A(\acal \otimes_R R')$ as an $\acal \otimes_R R'$-linear category.
\end{theorem}

Before going into the proof, we record a few consequences.

\begin{corollary}\label{coro etale locally trivial}
Let $R$ be a  G-ring and let $\Ccal$ be an invertible $\Mod^\heartsuit_R$-linear cocomplete category. Then there exists a faithfully flat \'etale morphism of  commutative rings  $R \rightarrow R'$ such that $\Ccal \otimes_{R} R'$ is equivalent to $\Mod^\heartsuit_{R'}$ as an $R'$-linear category.
\end{corollary}
\begin{proof}
By theorem \ref{theo abelian con coefficients} we may after passing to a faithfully flat \'etale $R$-algebra assume that $\ccal$ is the category of left modules over an Azumaya $R$-algebra $A$. The corollary now follows from the fact that Azumaya $R$-algebras are \'etale locally Morita equivalent to the unit (see corollary \ref{coro smooth and proper locally trivial}).
\end{proof}

\begin{notation}\label{notation twist mod}
Let $\mathcal{L}: \operatorname{CRing}\rightarrow \Spc$ be the functor that associates to each commutative ring $R$ the space of $R$-linear Grothendieck abelian categories which are \'etale locally on $\Spec(R)$ equivalent to $\Mod_R^\heartsuit$. This is a sheaf for the \'etale topology, with a pointing given by the object $\Mod_\ZZ^\heartsuit$ in $\mathcal{L}(\ZZ)$. The resulting pointed object is equivalent to $B^2\GG_m$. If $\mathcal{G}$ is a $\GG_m$-gerbe  on $\Spec(R)$ we denote by $\Mod^\heartsuit_{R, \mathcal{G}}$ the induced point in $\mathcal{L}(\Spec(R))$.
\end{notation}

\begin{corollary}\label{coro exists gerbe invertible} 
Let $R$ be a  G-ring and let $\Ccal$ be an invertible $\Mod^\heartsuit_R$-linear cocomplete category. Then there exists a $\GG_m$-gerbe $\mathcal{G}$ on $\Spec(R)$ and an $R$-linear equivalence ${\ccal = \Mod^\heartsuit_{R, \mathcal{G}}}$
\end{corollary}
\begin{proof}
This is a direct consequence of corollary  \ref{coro etale locally trivial} and the definitions.
\end{proof}

\begin{corollary}\label{coro classify fully dualizables abelian} 
Let $R$ be a G-ring and let $\Ccal$ be a fully dualizable $\Mod^\heartsuit_R$-linear cocomplete category. Then there exists a finite \'etale $R$-algebra $\tilde{R}$, a $\GG_m$-gerbe $\mathcal{G}$ on $\Spec(\tilde{R})$ and an $R$-linear equivalence $\ccal = \Mod^\heartsuit_{\tilde{R}, \mathcal{G}} $.
\end{corollary}
\begin{proof}
Let $\tilde{R} = \End^\enh_{\Funct_R(\ccal, \ccal)}(\id_\ccal)$ be the $R$-linear center of $\ccal$. Then $\tilde{R}$ is a commutative $R$-algebra, and $\ccal$ may be equipped with a canonical $\tilde{R}$-linear structure. Since $\ccal$ is dualizable the formation of $\Funct_{R}(\ccal, \ccal)$ commutes with base change, and since $\id_\ccal$ is compact projective the formation of its endomorphisms commutes with base change as well. The fact that $\ccal$ is fully dualizable implies that $\tilde{R}$ is a dualizable $R$-module. The assertion that $\tilde{R}$ is \'etale may then be reduced by base change to the case where $R = k$ is an algebraically closed field. In this case theorem \ref{theo abelian con coefficients} implies that $\ccal$ is the category of left modules over a finite product of copies of $k$, which has  \'etale center.

It remains to show that $\ccal =  \smash{\Mod^\heartsuit_{\tilde{R}, \mathcal{G}}}$ for some gerbe $\mathcal{G}$ on $\Spec(\tilde{R})$.  By corollary \ref{coro exists gerbe invertible} it suffices to show that $\ccal$ is invertible over $\tilde{R}$.  This can be checked \'etale locally by virtue of the fact that $\ccal$ is a Grothendieck abelian category (corollary \ref{coro properties dualizable 11}). By an application of theorem \ref{theo abelian con coefficients} we may reduce to the case where  $\ccal$ is the category of left modules over a smooth and proper algebra $A$ in $\Mod_R^\heartsuit$. Applying corollary \ref{coro smooth and proper locally trivial} we may further reduce to the case when $A$ is a finite product of copies of $R$, in which case the assertion is clear.
\end{proof}

We devote the remainder of this section to the proof of theorem \ref{theo abelian con coefficients}.

\begin{definition}\label{definition elementary extension}
Let $f: R \rightarrow S$ be a morphism of local Artinian commutative rings, and let $\mathfrak{m}$ be the maximal ideal of $R$. We say that $f$ is an elementary extension if it is surjective and $\Ker(f)$ is isomorphic as an $R$-module to $R/\mathfrak{m}$.
\end{definition}

\begin{remark}
Let $f: R \rightarrow S$ be a morphism of local Artinian commutative rings and let $\mathfrak{m}$ be the maximal ideal of $R$. Then $f$ is an elementary extension if and only if it induces an isomorphism $S = R/Rx$ for some nonzero nonunit element $x$ in $R$ such that $x\mathfrak{m} = 0$.
\end{remark}

\begin{lemma}\label{lemma quotient is sequence of elementaries}
Let $f: R \rightarrow S$ be a surjective morphism of local Artinian commutative rings. Then there exists a sequence of elementary extensions of local Artinian commutative rings $R = R_0 \rightarrow R_1 \rightarrow R_2 \ldots \rightarrow R_n = S$.
\end{lemma}
\begin{proof}
Since $R$ is Noetherian any sequence of quotients of $R$ stabilizes. To prove the lemma it will therefore suffice to show that if we have a sequence of local Artinian commutative rings $R = R_0 \rightarrow R_1 \rightarrow \rightarrow R_2 \rightarrow \ldots \rightarrow R_k \rightarrow S$ such that $R_i \rightarrow R_{i+1}$ is an elementary extension for all $0 \leq i < k$, then either $R_k = S$ or there exists a factorization $R_k \rightarrow R_{k+1} \rightarrow S$ where the first map is an elementary extension. Replacing $R$ with $R_k$ we may reduce to showing that if $f$ is not an isomorphism then we have a factorization $R \rightarrow T \rightarrow S$ where the first map is an elementary extension. In this case $\Ker(f)$ is nonzero and since the only prime ideal of $R$ is $\mathfrak{m}$ and $R$ is Noetherian we see that $\mathfrak{m}$ is an associated prime of $\Ker(f)$. Let $x$ in $\Ker(f)$ be an element with annihilator $\mathfrak{m}$. Then the proof finishes by setting $T = R/Rx$.
\end{proof}

\begin{remark}\label{remark properties CS}
Let $f: R \rightarrow S$ be a surjective map of commutative rings, and let $\Ccal$ be an $R$-linear Grothendieck abelian category. Let $\Ccal_S = \Ccal \otimes_R S$. Then the extension of scalars functor $\Ccal \rightarrow \Ccal_S$ admits a fully faithful right adjoint $\iota$. The unit of this localization is given by tensoring with $f$, so that an object $X$ in $\Ccal$ belongs to the image of $\iota$ if and only if the map $R \otimes X \rightarrow S \otimes X$ is an isomorphism. We will often identify $\ccal_S$ with its image under $\iota$.

 For every object $X$ in $\Ccal$ we have an exact sequence
\[
\Ker(f) \otimes X \rightarrow R \otimes X \rightarrow S \otimes X \rightarrow 0
\]
so that $X$ belongs to $\ccal_S$ if and only if the first map above is zero. It follows that $x_\alpha$ is a set of generators for the $R$-module $\Ker(f)$ then $X$ belongs to $\ccal_S$ if and only if $x_\alpha: X \rightarrow X$ is zero for all $\alpha$.

It follows from the above description that $\ccal_S$ is closed under limits, colimits and passage to subobjects inside $\ccal$ (however note that it is not closed under passage to extensions in general). Furthermore, $\iota$ admits a right adjoint which sends each object $X$ to the intersection over all $\alpha$ of the kernel of $x_\alpha: X \rightarrow X$. In particular, if $\Ker(f)$ is finitely generated then the right adjoint to $\iota$ preserves filtered colimits, which implies that $\iota$ sends compact objects to compact objects.
\end{remark}

\begin{lemma}\label{lemma exist filtration}
Let $R$ be a local Artinian commutative ring with residue field $k$. Let $\ccal$ be an $R$-linear Grothendieck abelian category and let $X$ be an object of $\ccal$. Then there exists a finite filtration $0 = X_0 \subseteq X_1 \subseteq \ldots \subseteq X_t = X$ such that $X_{i+1} / X_i$ belongs to $\Ccal \otimes_R k$ for all $i$.
\end{lemma}
\begin{proof}
By lemma \ref{lemma quotient is sequence of elementaries} we may pick a sequence of elementary extensions of local Artinian commutative rings $R = R_0 \rightarrow R_1 \rightarrow \ldots \rightarrow R_n = k$. We will prove that the lemma holds for the rings $R_s$ by reverse induction on $s$. The case $s = n$ is clear. Assume now that $s < n$ and that the lemma is known for $R_{s+1}$. Let $x$ be a generator for the kernel of $R_s \rightarrow R_{s+1}$. Then we have an exact sequence $0 \rightarrow X' \rightarrow X \rightarrow X'' \rightarrow 0$ where $X' = \Ker(x: X \rightarrow X)$ and $X'' = \operatorname{Im}(x: X \rightarrow X)$. Since $x^2 = 0$ we have that both $X'$ and $X''$ belong to $\ccal \otimes_{R_{s}} R_{s+1}$.  The inductive hypothesis allows us to construct filtrations for $X'$ and $X''$ whose associated graded pieces belong to $\Ccal \otimes_R k$. We now obtain the desired filtration on $X$ by putting together these two filtrations.
\end{proof}

\begin{lemma}\label{lemma construct non split extension}
Let $R$ be a local Artinian commutative ring with residue field $k$.  Let $\Ccal$ be an $R$-linear Grothendieck abelian category such that $\ccal \otimes_R k$ is semisimple, and let $X$ be an object of $\ccal$.  Then one of the following two happen:
\begin{enumerate}[\normalfont (a)]
\item $X$ is projective.
\item There exists a non-split extension of $X$ by a simple object of $\ccal \otimes_R k$.
\end{enumerate}
\end{lemma}
\begin{proof}
Assume that $X$ is not  projective in $\ccal$. We will show that (b) holds. Choose  $M$ in $\Ccal$ such that $\Ext_\ccal^1(X, M) \neq 0$. By lemma \ref{lemma exist filtration} we may pick a filtration $0 = M_0 \subseteq M_1 \subseteq \ldots \subseteq M_t = M$ with $M_{i+1}/M_i$ in $\ccal_k$ for all $i$. Let $j$ be the smallest index such that $\Ext_\ccal^1(M, M_j) \neq 0$. Then $\Ext_\ccal^1(X, M_{j}/M_{j-1}) \neq 0$. Replacing $M$ by $M_{j}/M_{j-1}$ if necessary we may assume that $M$ belongs to $\ccal_k$. 

Since $\ccal$ is semisimple we may write $M$ as a direct sum of a family of simple objects $S_i$. Then $M$ is a direct summand of $\prod_i S_i$, and therefore $\Ext_\ccal^1(X, \prod_i S_i) \neq 0$. Set $W = \prod_i S_i$, and let $Z$ be the product of the family $S_i$ computed in the derived category $\der(\ccal)$. Then $W = \tau_{\geq 0}Z$, so we have an exact sequence 
\[
\Ext^0_{\der(\ccal)}(X, \tau_{\leq -1}Z) \rightarrow \Ext^1_{\der(\ccal)}(X, W) \rightarrow \Ext^1_{\der(\ccal)}(X, Z).
\]
Here the first term vanishes, and since the middle term is nonzero we conclude that the third term is nonzero. This is the same as $\prod_i \Ext_\ccal^1(X, S_i)$, so   we have that $\Ext_\ccal^1(X, S_i) \neq 0$ for some $i$, and therefore a non-split extension of $X$ by $S_i$ exists, as desired.
\end{proof}

\begin{lemma}\label{lemma flat into local artinian}
Let $R$ be a local Artinian commutative ring with residue field $k$ and let $\Ccal$ be an $R$-linear Grothendieck abelian category. Let $X$ be an object of $\ccal$. The following are equivalent:
\begin{enumerate}[\normalfont(1)]
\item $X$ is flat over $R$.
\item The map $\mu: \mathfrak{m} \otimes X \rightarrow R \otimes X = X$ induced from the inclusion $\mathfrak{m} \rightarrow R$ is a monomorphism.
\item $\Tor_1(k, X) = 0$.
\end{enumerate}
\end{lemma}
\begin{proof}
The fact that (1) implies (2) follows directly from the definitions. The equivalence of (2) and (3) follows  from the fact that we have an exact sequence
\[
0 = \Tor_1(R, X) \rightarrow \Tor_1(k, X) \xrightarrow{} \Tor_0(\mathfrak{m}, X) \rightarrow \Tor_0(R, X).
\]
Assume now that (3) holds. Since the property that $\Tor_1(Y, X) = 0$ is preserved under passage to extensions and filtered colimits and $\Mod^\heartsuit_{R}$ is generated by $k$ under extensions and filtered colimits we see that $\Tor_1(Y, X) = 0$ for all $Y$ in $\Mod^\heartsuit_{R}$. Assume now given a monomorphism $i: Z \rightarrow Z'$ in $\Mod^\heartsuit_{R}$. Then the kernel of $i \otimes \id_X : Z \otimes X \rightarrow Z' \otimes X$ receives an epimorphism from $\Tor_1(Z'/Z, X)$, and is therefore $0$. This proves that (1) holds.
\end{proof}

\begin{lemma}\label{lemma construct flat generator}
Let $f: R \rightarrow S$ be an elementary extension of local Artinian commutative rings with residue field $k$. Let $\Ccal$ be an $R$-linear Grothendieck abelian category such that $\ccal \otimes_R k$ is semisimple. Assume given a non-split extension
\[
0 \rightarrow M \rightarrow U \rightarrow X \rightarrow 0
\]
in $\ccal$, where $M$ is a simple object of $\Ccal \otimes_R k$ and $X$ is a projective object of $\ccal \otimes_R S$, flat over $S$, and such that $X \otimes_R k$ is simple. Then $U$ is projective, flat over $R$, and $U \otimes_R S = X$.
\end{lemma}
\begin{proof}
Let $\mathfrak{m}$ be the maximal ideal of $R$. Set $\Ccal_S = \Ccal \otimes_R S$ and $\ccal_k = \ccal \otimes_R k$.  Fix a generator $x$ for $\Ker(f)$. Since $X$ is projective in $\ccal_S$ we see that $U$ cannot belong to $\ccal_S$, and hence $x: U \rightarrow U$ is nonzero.  Since $x$ acts by zero on $X$ we have that $x: U \rightarrow U$ factors through $M$. Its image is a nonzero subobject of $M$, and since $M$ is simple we conclude that the image of $x: U \rightarrow U$ is equal to $M$. In particular, we claim:
\begin{itemize}
\item [$(\star)$] If  $N$ is an $R$-module on which $x$ acts by zero, the map $N \otimes M \rightarrow N \otimes U$ obtained by tensoring $N$ with the inclusion $M \rightarrow U$ is zero.
\end{itemize}
To see this, note that it is enough to show that $\id \otimes x : N \otimes U \rightarrow N \otimes U$ vanishes, which follows from the fact that this map is equivalent to $x \otimes \id: N \otimes U \rightarrow N \otimes U$.

By lemma \ref{lemma flat into local artinian}, to show that $U$ is flat over $R$ it suffices to show that the map $j: \mathfrak{m} \otimes U \rightarrow R \otimes U = U$ obtained by tensoring the inclusion $\mathfrak{m} \rightarrow R$ with $U$ is a monomorphism. Let $\mathfrak{m}_S = \mathfrak{m}/Rx$. We have a commutative diagram in $\Ccal$ with exact rows as follows:
 \[
 \begin{tikzcd}
& Rx \otimes U \arrow{d}{j_1} \arrow{r}{} & \mathfrak{m} \otimes U \arrow{r}{} \arrow{d}{j} & \mathfrak{m}_S \otimes U \arrow{d}{j_2} \arrow{r}{} & 0 \\
0 \arrow{r}{} & M \arrow{r}{} & U \arrow{r}{} & X \arrow{r}{} & 0
\end{tikzcd}
\]
To show that $j$ is a monomorphism it will suffice to show that $j_1$ and $j_2$ are monomorphisms.

We first consider $j_1$. Our assumptions guarantee that $Rx$ is isomorphic to $k$ as an $R$-module, so that $Rx \otimes U$ belongs to $\Ccal \otimes_R k$. The same holds for $M$. To show that $j_1$ is a monomorphism it will suffice to prove the following two assertions:
\begin{enumerate}[(a)]
\item $j_1$ is nonzero.
\item $Rx \otimes U$ is simple.
\end{enumerate}
We begin by addressing (a). To show that $j_1$ is nonzero it suffices to show that the map $Rx \otimes U \rightarrow R \otimes U = U$ obtained by tensoring the inclusion $Rx \rightarrow R$ with $U$ is nonzero. To do so it suffices to show that the map
\[
U = R \otimes U \xrightarrow{x \otimes \id_U} R \otimes U = U
\]
is nonzero . This is the same as the action of $x$ on $U$, which we have already shown to be nonzero.

It remains to address (b). Consider the exact sequence
\[
Rx \otimes M \rightarrow Rx \otimes U \rightarrow Rx \otimes X \rightarrow 0.
\]
Here $Rx \otimes X = k \otimes X$ is simple by our assumption on $X$. We may thus reduce to proving that the map $Rx \otimes M \rightarrow Rx \otimes U$ is zero. This follows from $(\star)$ since $x$ acts by zero on $Rx$. 

We now show that $j_2$ is a monomorphism. This is the composition of the map $\mu_1: \mathfrak{m}_S \otimes U \rightarrow \mathfrak{m}_S \otimes X$ obtained by tensoring $\mathfrak{m}_S$ with the projection $U \rightarrow X$, and the map $\mu_2: \mathfrak{m}_S \otimes X \rightarrow S \otimes X = X$ obtained by tensoring the inclusion $\mathfrak{m}_S \rightarrow S$ with $X$. The kernel of $\mu_1$ is the image of the map $\mathfrak{m}_S \otimes M \rightarrow \mathfrak{m}_S \otimes U$ obtained by tensoring $\mathfrak{m}_S$ with the inclusion $M \rightarrow U$. It now follows from $(\star)$ that $\mu_1$ is a monomorphism, since $x$ acts by $0$ on $\mathfrak{m}_S$. The map $\mu_2$ is a monomorphism by virtue of our assumption that $X$ is flat over $S$. We conclude that $j_2$ is a monomorphism, as desired.

We now show that $U$ is projective. We will do so by proving that $\Ext^1_\ccal(U, Y) = 0$ for all $Y$ in $\ccal$. By lemma \ref{lemma exist filtration} it suffices to address the case when $Y$ belongs to $\ccal_k$. Then we have
\[
\Ext^1_\ccal(U, Y) = \Ext^1_{\der(\ccal)}(U, Y) = \Ext^1_{\der(\ccal)\otimes_R k}( U \otimes_R k, Y) = \Ext^1_{\ccal \otimes_R k}(U \otimes_R k, Y)
\]
where here we use flatness of $U$ to identify $U \otimes_R^L k$ with $U \otimes_R k$. Recall from our proof of (a) that $U \otimes_R k = X \otimes_R k$. The fact that the above group vanishes is now a consequence of the fact that $X \otimes_R k$ is projective in $\ccal \otimes_R k$.

It remains to prove that $U \otimes_R S = X$. Consider the exact sequence
\[
S \otimes M \rightarrow S \otimes U \rightarrow S \otimes X \rightarrow 0.
\]
It suffices to show that the first map is zero. This follows from an application of $(\star)$, since $x$ acts by zero on $S$. 
\end{proof}

\begin{lemma}\label{lemma defo is compact}
Let $f: R \rightarrow S$ be a  surjective morphism of local Artinian commutative rings and let $\ccal$ be an $R$-linear Grothendieck abelian category. Let $X$ be a projective object of $\ccal$ such that $X \otimes_R S$ is compact in $\ccal \otimes_R S$. Then $X$ is compact in $\ccal$.
\end{lemma}
\begin{proof}
Applying lemma \ref{lemma quotient is sequence of elementaries} we may reduce to the case when $f$ is an elementary extension. Let $x$ be a generator of $\Ker(f)$. Let $Y_\alpha$ be a filtered diagram of objects of $\ccal$. For each $\alpha$ let $Y_\alpha' = \Ker(x: Y_\alpha \rightarrow Y_\alpha)$ and $Y_\alpha'' = \operatorname{Im}(x:Y_\alpha \rightarrow Y_\alpha)$. Then we have a commutative diagram of $R$-modules with exact rows
\[
\begin{tikzcd}[column sep = small]
0 \arrow{r}{} & \colim \Hom^\enh_\ccal(X, Y_\alpha') \arrow{r}{} \arrow{d}{} & \colim \Hom^\enh_\ccal(X, Y_\alpha) \arrow{r}{} \arrow{d}{} & \colim \Hom^\enh_\ccal(X, Y_\alpha')  \arrow{r}{} \arrow{d}{} & 0 \\
0 \arrow{r}{} &  \Hom^\enh_\ccal(X, \colim Y_\alpha')  \arrow{r}{} &  \Hom^\enh_\ccal(X,\colim Y_\alpha) \arrow{r}{}  &  \Hom^\enh_\ccal(X, \colim Y_\alpha'') \arrow{r}{} & 0.
\end{tikzcd}
\]
We wish to prove that the middle vertical arrow is an isomorphism. This will follow if we can prove that the other two vertical arrows are isomorphisms. This follows from the fact that $X \otimes_R S$ is compact in $\ccal \otimes_R S$, since $Y_\alpha'$ and $Y_\alpha''$ belong to $\ccal \otimes_R S$.
\end{proof}

\begin{lemma}\label{lemma deform generators abelian}
Let $f: R \rightarrow S$ be surjective map of local Artinian commutative rings with residue field $k$. Let $\acal$ be a symmetric monoidal $R$-linear Grothendieck abelian category. Assume that $\acal$ is rigid, generated by compact projective objects, proper over $R$, and that $\acal \otimes_R k$ is semisimple. Let $\Ccal$ be a fully dualizable $\acal$-linear Grothendieck abelian category. Assume given a finite family $\lbrace X_t \rbrace$ of objects of $\ccal \otimes_R S$ with the following properties:
\begin{itemize}
\item $X_t$ is compact projective and flat over $S$ for all $t$.
\item $X_t \otimes_S k$ is simple for all $t$.
\item $\bigoplus X_t$ is an $\acal \otimes_R S$-generator for $\ccal \otimes_R S$.
\end{itemize}
Then there exist a family of objects $X'_t$ of $\ccal$ with the following properties:
\begin{itemize}
\item $X'_t \otimes_R S = X_t$ for all $t$.
\item $X'_t$ is compact projective and flat over $R$ for all $t$.
\item  $\bigoplus X'_t$ is an $\acal$-generator for $\ccal$.
\end{itemize}
\end{lemma}
\begin{proof}
Applying lemma \ref{lemma quotient is sequence of elementaries} we may reduce to the case when $f$ is an elementary extension. Let $\ccal_S$, $\ccal_k$, $\acal_S$ and $\acal_k$ the base changes of $\ccal$ and $\acal$. By theorem \ref{theorem abelian} combined with remark \ref{remark separable are semisimple} we see that $\ccal_k$ is semisimple. Define for each $t$ an object $X'_t$ of $\ccal$, as follows:
\begin{itemize}
\item If $X_t$ is projective in $\ccal$ then $X'_t = X_t$.
\item If $X_t$ is not projective in $\ccal$, then using lemma \ref{lemma construct non split extension} construct a non-split extension $0 \rightarrow M_t \rightarrow U_t \rightarrow X_t \rightarrow 0$ in $\ccal$ where $M_t$ is simple in $\ccal_t$, and set $X'_t = U_t$. By lemma \ref{lemma construct flat generator} we have that $X'_t$ is  flat over $R$, projective, and satisfies $X'_t \otimes_R S = X_t$.
\end{itemize}

Lemma \ref{lemma defo is compact} implies that $X'_t$ is in fact compact projective for every $t$. We claim that  $X' = \bigoplus X'_t$ is an $\acal$-generator for $\ccal$. Let $Y$ be an arbitrary object of $\ccal$. By lemma \ref{lemma exist filtration} we may pick a filtration $0 = Y_0 \subseteq Y_1 \subseteq \ldots \subseteq Y_n = Y$ with successive quotients in $\ccal_k$. Pick for each $i$ an object $Z_i$ in $\acal_k$ and an epimorphism $Z_i \otimes (X' \otimes_R k) \rightarrow Y_{i+1}/Y_i$ in $\ccal_k$ (which exists since $X' \otimes_R S = \bigoplus X_t$ is an $\acal_S$-generator for $\ccal_S$). For each $i$ pick an epimorphism $Z'_i \rightarrow Z_i$ in $\acal$ with $Z'_i$ projective. Since $X'$ is projective the induced maps $Z'_i \otimes X' \rightarrow Y_{i+1}/Y_i$ may be lifted to a sequence of morphisms $Z'_i \otimes X' \rightarrow Y$. The resulting map $\bigoplus Z'_i \otimes X'\rightarrow Y$ is then an epimorphism. Since $Y$ was arbitrary we conclude that $X'$ is an $\acal$-generator for $\ccal$, as desired.

To prove the proposition it now suffices to show that for all $t$ the object $X_t$ is not projective in $\ccal$ (so that $X'_t = U_t$ is flat over $R$ for all $t$).  Assume for the sake of contradiction that there is an index $t$ such that $X_t$ is projective in $\ccal$. Since $\ccal$ admits a compact projective $\acal$-generator, it is generated by compact projective objects. Let $A$ be the algebra of endomorphisms of $X_t$ in $\acal$. Applying remark \ref{remark properness when comp gen} we see that $A$ is compact projective as an object of $\acal$, and therefore the $R$-module $\Hom^\enh_\acal(1_\acal, A)$ is compact projective.  This coincides with the $R$-module of endomorphisms of $X_t$, which is an $S$-module since $X_t$ belongs to $\ccal_S$. It follows that $\Hom^\enh_\acal(1_\acal, A) = 0$ and hence $X_t = 0$. This contradicts the fact that $X_t \otimes_S k$ is simple (and in particular nontrivial).
\end{proof}

\begin{lemma}\label{lemma compactness y completion}
Let $R$  complete local Noetherian commutative ring with maximal ideal $\mathfrak{m}$ and residue field $k$. Let $\ccal$ be an $R$-linear Grothendieck abelian category, generated by compact projective objects and proper over $R$. Let $X$ be an object of $\ccal$.
\begin{enumerate}[\normalfont(1)]
\item If $X$ is compact projective then $X = \lim X \otimes_R R/\mathfrak{m}^n$.
\item If $X = \lim X \otimes^L_R R_n$ and $X \otimes^L_R k$ is compact projective in $\der(\ccal)_{\geq 0} \otimes_R k$, then $X$ is compact projective.
\end{enumerate}
\end{lemma}
\begin{proof}
We first prove part (1). Set $R_n = R/\mathfrak{m}^n$ for all $n \geq 1$.  Let $\eta: X \rightarrow \lim X \otimes_R R_n$ be the canonical map. We wish to show that $\eta$ is an isomorphism. It suffices for this to prove that $\Hom^\enh_\ccal(Y, \eta)$ is an isomorphism for all compact projective objects $Y$ in $\ccal$. This is equivalent to the canonical map
\[
\Hom^\enh_\ccal(Y, X) \rightarrow \lim \Hom^\enh_{\ccal \otimes_R R_n}(Y \otimes_R R_n, X \otimes_R R_n)
\]
which is, in turn, equivalent to the map
\[
\Hom^\enh_\ccal(Y, X) \rightarrow \lim \Hom^\enh_{\ccal}(Y, X) \otimes_R R_n.
\]
The above is an equivalence since $\Hom_{\Mcal}^\enh(Y, X)$ is a dualizable $R$-module.

We now prove part (2).  By  proposition \ref{prop extend compact projectives} we may pick a compact projective object $X'$ in $\ccal$ whose image in $\ccal \otimes_R k$ is isomorphic to $X \otimes_R k$.  The fact that $X'$ is projective allows us to lift this isomorphism to a morphism $f: X' \rightarrow X$ in $\ccal$. Let $Z$ be the cofiber of $f$ inside $\der(\ccal)_{\geq 0}$, and note that $Z \otimes^L_R k = 0$. It follows from this that $Z \otimes^L_R R_n = 0$ for all $n \geq 1$. Combining this with an application of part (1) to $X'$, we see that $\Ext^1_{\der(\ccal)}(Z, X') = 0$. Hence $Z$ is a retract of  $X$, which implies $Z = \lim Z \otimes^L_R R_n = 0$. We conclude that $f$ is an isomorphism, and therefore $X$ is compact projective, as desired.
\end{proof}

\begin{lemma}\label{lemma deal with completes classical}
Let $R$ be a complete local Noetherian ring with residue field $k$ and let $\acal$ be a symmetric monoidal $R$-linear Grothendieck abelian category. Assume that $\acal$ is rigid, generated by compact projective objects, proper over $R$, and that $\acal \otimes_R k$ is semisimple. Let $\ccal$ be a fully dualizable $\acal$-linear Grothendieck abelian category. Then $\ccal$ admits a compact projective $\acal$-generator.
\end{lemma}
\begin{proof}
Let $\mathfrak{m}$ be the maximal ideal of $R$. For each $n \geq 1$ set $R_n = R/\mathfrak{m}^n$, $\acal_n = \acal \otimes_R R_n$ and $\Ccal_n = \Ccal \otimes_R R_n$. By theorem \ref{theorem abelian} combined with remark \ref{remark separable are semisimple} we see that $\ccal_1$ is semisimple and admits a compact projective $\acal_1$-generator $X_1$. Since $X_1$ is a finite sum of simple objects we may apply lemma \ref{lemma deform generators abelian} inductively to find a compatible sequence of compact projective $\acal_n$-generators $X_n$ in $\ccal_n$, flat over $R_n$. 

Set $X = \lim X_n$.  Embed $\ccal$ inside $\der(\ccal)_{\geq 0}$, and note that we have for each $n$ an equivalence $(\der(\ccal)_{\geq 0} \otimes_R R_n)^\heartsuit = \ccal_n$. Since $X_n$ is flat over $R_n$ we see that its image inside $\der(\ccal)_{\geq 0} \otimes_R R_n$ is also flat over $R_n$. Hence the sequence $(X_n)$ defines an object of $\lim \der(\ccal)_{\geq 0}  \otimes_R R_n$. We may identify $X$ with the image of $(X_n)$ under the right adjoint to the projection $p: \der(\ccal)_{\geq 0} \rightarrow \lim \der(\ccal)_{\geq 0}  \otimes_R R_n$. Applying corollary \ref{coro properties dualizable 11} we see that products in $\ccal$ are exact, and therefore by proposition \ref{prop completes ffff} we have that $X \otimes_R^L R_n = X_n$ for all $n \geq 1$. In particular, $X \otimes_R R_n = X_n$ for all $n$.

We will show that $X$ is a compact projective $\acal$-generator for $\ccal$. We begin by proving that it is an $\acal$-generator. Changing the roles of $\Ccal$ and $\ccal^\vee$ in the previous arguments we see that  $\ccal^\vee$ contains an object $Y$ such that $Y_n = Y \otimes_R R_n$  is a compact projective $\acal_n$-generator for $\ccal^\vee_n$ for all $n$.  Consider the object $X \otimes Y$ inside $\Ccal \otimes_\acal \ccal^\vee$, and note that $(X \otimes Y) \otimes_R R_n = X_n \otimes Y_n$ is a compact projective $\acal_n$-generator for  $\ccal_n \otimes_{\acal_n} \ccal^\vee_n$ for all $n$.

Let $\delta$ be the image of the unit $1_\acal$ under the unit map $\eta: \acal \rightarrow \ccal \otimes_\acal \ccal^\vee$, and note that $\delta$ is compact projective since $\ccal$ is smooth. For each $n \geq 1$ set $\delta_n = \delta \otimes_R R_n$. By lemma \ref{lemma compactness y completion} applied to $\acal$ we have $1_\acal = \lim 1_{\acal} \otimes_R R_n$.  Since $\ccal$ is fully dualizable the map $\eta$ admits a left adjoint, and in particular preserves limits. It follows that $\delta = \lim \delta_n$. Fix an epimorphism $Z \otimes (X_1 \otimes Y_1) \rightarrow  \delta_1$, where $Z$ is a compact projective object of $\acal_1$. By proposition \ref{prop extend compact projectives} we may find a compact projective object $Z'$ in $\acal$ whose image in $\acal_1$ recovers $Z$.  Replacing $Y$ with $Y \oplus Z' \otimes Y$ we may now assume the existence of an epimorphism $\rho_1 : X_1 \otimes Y_1 \rightarrow \delta_1$. Using the fact that $X_n \otimes Y_n$ is projective for all $n$ we may construct inductively a compatible sequence of maps $\rho_n:X_n \otimes Y_n \rightarrow \delta_n$, which in the limit defines a morphism $\rho: X \otimes Y  \rightarrow \delta$. 

We claim that $\rho$ admits a section. To prove this it suffices to show that the induced map 
\[
\rho_*: \Hom^\enh_{\ccal \otimes_\acal \ccal^\vee}(\delta, X \otimes Y ) \rightarrow \Hom^\enh_{\ccal \otimes_\acal \ccal^\vee}(\delta, \delta)
\] is an epimorphism in $\acal$. Using corollary \ref{coro check epi on fiber} we may reduce to proving that $\rho_* \otimes_R k$ is an epimorphism in $\acal_1$. Since $\delta$ is compact projective we have that $\rho_* \otimes_R k$ is equivalent to the map 
\[
(\rho_1)_* : \Hom^\enh_{\ccal_1 \otimes_{\acal_1} \ccal_1^\vee}(\delta_1, X_1 \otimes Y_1 ) \rightarrow \Hom^\enh_{\ccal_1 \otimes_{\acal_1} \ccal_1^\vee}(\delta_1, \delta_1)
\]
which is an epimorphism since $\delta_1$ is projective and $\rho_1$ is an epimorphism.

Let $\ccal'$ be the smallest subcategory of $\ccal$ closed under colimits, the action of $\acal$, and containing $X$. Then the induced map $\ccal' \otimes_\acal \ccal^\vee \rightarrow \ccal \otimes_\acal \ccal^\vee$ is fully faithful. Its image contains $X \otimes Y$ and is closed under retracts, so it also contains $\delta$. It now follows that the identity of $\ccal$ belongs to the image of the functor $\Funct_{\acal}(\ccal, \ccal') \rightarrow \Funct_{\acal}(\ccal, \ccal)$ of composition with the inclusion $i: \ccal' \rightarrow \ccal$. Therefore $i$ admits a section, which implies that $\ccal' = \ccal$. This concludes that proof that $X$ is an $\acal$-generator for $\ccal$.

It remains to show that $X$ is compact projective. Let $A$ be (the opposite of) the endomorphism algebra of $X$ inside $\acal$. By proposition \ref{prop lex localization prestable} the functor of tensoring with $X$ yields an $\acal$-linear left exact localization $q: \LMod_A(\acal) \rightarrow \ccal$, which is an equivalence if and only if $X$ is compact projective. We will finish the proof by showing that $q$ is an equivalence. Since $q$ is left exact it is enough to prove that if $M$ is a $0$-truncated left $A$-module such that $q(M) = 0$ then $M = 0$. Since $\LMod_A(\acal)$ is generated by compact projective objects it is enough to show that if $N$ is a finitely generated subobject of $M$ then $N = 0$. The fact that $q$ is left exact implies that $q(N) = 0$. Replacing $M$ by $N$ we may now reduce to the case when $M$ is finitely generated as a left $A$-module.

For each $n$ let $A_n$ be (the opposite of) the endomorphism algebra of $X_n$ inside $\acal_n$, and observe that we have an equivalence of algebras $A = \lim A_n$ (where here we regard $A_n$ as an algebra in $\acal$ via restriction of scalars). The fact that $X_n$ is compact projective for all $n$ implies that the sequence of algebras $A_n$ is compatible with base changes. Since $\LMod_{A_n}(\acal_n)$ is a fully dualizable $\acal_n$-linear category, we see that $A_n$ is proper, and in particular flat. It follows from this that $(A_n)$ defines an object in $\lim \der(\acal_n)_{\geq 0}$. An application of  proposition \ref{prop completes ffff} now shows that $A \otimes^L_R R_n = A_n$ for all $n \geq 1$. By virtue of part (2) of lemma \ref{lemma compactness y completion} we have that $A$ is compact projective as an object of $\acal$, and therefore  $M$ is finitely generated as an object of $\acal$. 

Consider now the commutative square of categories
\[
\begin{tikzcd}
\LMod_A(\acal) \arrow{d}{q} \arrow{r}{f^*} & \LMod_{A_1}(\acal_1) \arrow{d}{q_1} \\
\Ccal \arrow{r}{f^*} & \Ccal_1
\end{tikzcd}
\]
obtained from $q$ by tensoring with $\Mod_R^\heartsuit \rightarrow \Mod_{R_1}^\heartsuit$. Observe that this is horizontally right adjointable. Denote by $f_*$ the right adjoints to the horizontal arrows. Then $f_*q_1(M \otimes_R k) = q f_* (M \otimes_R k) = q(M)\otimes_R k = 0$, and since restriction of scalars is conservative we have $q_1(M \otimes_R k) = 0$. We may identify the functor $q_1$ with the functor of tensoring with the right $A_1$-module $X_1$ in $\Ccal_1$, which is an equivalence by virtue of the fact that $X_1$ is a compact projective $\acal_1$-generator for $\ccal_1$. It follows that $M \otimes_R k = 0$. The fact that $M = 0$ now follows from an application of proposition  \ref{proposition nakayama}.
\end{proof}

\begin{lemma}\label{lemma deal with filtered colimits classical}
Let $R_\alpha$ be a filtered diagram of commutative rings with colimit $R$. Assume given an index $\alpha_0$, a symmetric monoidal $R_{\alpha_0}$-linear Grothendieck abelian category $\acal_{\alpha_0}$ rigid and generated by compact projective objects, and a smooth $\acal_{\alpha_0}$-linear Grothendieck abelian category $\Ccal_{\alpha_0}$. If $\Ccal_{\alpha_0} \otimes_{R_{\alpha_0}} R$ has a compact projective $\acal_{\alpha_0} \otimes_{R_{\alpha_0}} R$-generator then there exists a transition $\alpha_0 \rightarrow \alpha$ such that $\Ccal_{\alpha_0} \otimes_{R_{\alpha_0}} R_\alpha$ has a compact projective $\acal_{\alpha_0} \otimes_{R_{\alpha_0}} R_\alpha$-generator.
\end{lemma}
\begin{proof}
Without loss of generality we assume that $\alpha_0$ is an initial index. For each $\alpha$ let $\acal_\alpha$ and $\ccal_\alpha$ be the base changes of $\acal_{\alpha_0}$ and $\ccal_{\alpha_0}$ to $R_\alpha$, and similarly denote by $\acal$ and $\ccal$ the base changes to $R$. Observe that we have $\Mod_R^\heartsuit = \colim \Mod_{R_\alpha}^\heartsuit$ in $\Pr^L$ (see the argument from \cite{HA} lemma 7.3.5.12), and therefore $\acal = \colim \acal_\alpha$ and $\ccal = \colim \ccal_\alpha$. 

Since $\ccal_\alpha$ is a dualizable $\acal_\alpha$-linear category for all $\alpha$ we see that the categories $\ccal_\alpha$ are $1$-strongly compactly assembled. Similarly, $\ccal$ is $1$-strongly compactly assembled. Let $X$ be a compact projective $\acal$-generator for $\ccal$. Applying theorem \ref{theorem lift strongly compacts} we may pick an index $\alpha$ and a compact projective lift $X_\alpha$ of $X$ to $\ccal_\alpha$. Restricting our diagram to the undercategory of $\alpha$ if necessary we may without loss of generality assume that $\alpha = \alpha_0$ is an initial index.

For each $\beta$ let $X_\beta = X_\alpha \otimes_{R_\alpha} R_\beta$. We will finish the proof by showing that there exists an index $\beta$ such that $X_\beta $ is an $\acal_\beta$-generator for $\ccal_\beta$. For each $\beta$ let $\dcal_{\beta}$ be the smallest subcategory of $\ccal_\beta$ containing $X_\beta$ and closed under colimits and the action of $\acal_\beta$. Then the action of $\acal_{\beta}$ on $\ccal_\beta$ restricts to give $\dcal_{\beta}$ an $\acal_\beta$-linear structure. Note that for every transition $\beta \rightarrow \beta'$ we have $\dcal_{\beta'} = \dcal_{\beta} \otimes_{R_\beta} R_{\beta'}$. Furthermore, the base change of these to $R$ recovers $\ccal$, and in particular we have $\ccal \otimes_\acal \ccal^\vee = \colim \dcal_\beta \otimes_{\acal_\beta} \ccal_\beta^\vee$. Let $\delta$ be the image of $1_\acal$ under the unit map $\acal \rightarrow \ccal \otimes_\acal \ccal^\vee$, and for each $\beta$ let $\delta_\beta$ be the image of $1_{\acal_\beta}$ under the unit map $\acal_\beta \rightarrow \ccal_\beta \otimes_{\acal_\beta} \ccal_\beta^\vee$. 
 Another application of theorem  \ref{theorem lift strongly compacts} shows that there exists $\beta$ and a compact projective lift $Y_\beta$ of $\delta$ to $ \dcal_\beta \otimes_{\acal_\beta} \ccal_\beta^\vee$. The image of $Y_\beta$ inside $\ccal_\beta \otimes_{\acal_\beta} \ccal_\beta^\vee$ is a compact projective object whose image in $\ccal \otimes_\acal \ccal^\vee$ agrees with the image of $\delta_\beta$. A final application of theorem \ref{theorem lift strongly compacts} shows that (changing $\beta$ if necessary) we may assume that $Y_\beta = \delta_\beta$. Hence  $\dcal_\beta \otimes_{\acal_\beta} \ccal_\beta^\vee$ contains $\delta_\beta$, which implies that the image of the inclusion $\Funct_{\acal_\beta}(\ccal_\beta, \dcal_\beta) \rightarrow \Funct_{\acal_\beta}(\ccal_\beta, \ccal_\beta)$ contains the identity. It follows that the inclusion $\dcal_\beta \rightarrow \ccal_\beta$ has a section, and therefore $\ccal_\beta = \dcal_\beta$. This shows that $X_\beta$ is an $\acal_\beta$-generator for $\ccal_\beta$, as desired.
\end{proof} 
 
  \begin{lemma}\label{lemma deal with products}
Let $\lbrace R_i \rbrace$ be a finite family of connective $E_\infty$-rings with product $R$. Let $\acal$ be a symmetric monoidal $R$-linear Grothendieck abelian category, rigid and generated by compact projective objects. Let $\ccal$ be an $\acal$-linear Grothendieck abelian category. Assume that  $\ccal \otimes_R R_i$ admits a compact projective $\acal\otimes_R R_i$-generator for all $i$.  Then there exists an algebra $A$ in $\acal  $  such that $\Ccal  $ is equivalent to $\LMod_A(\acal  )$ as an $\acal $-linear category.
 \end{lemma}
 \begin{proof}
 By proposition \ref{prop lex localization prestable} it suffices to show that $\Ccal $ admits a compact projective $\acal  $-generator. Zariski descent for Grothendieck abelian categories implies that the $\acal  $-linear functors $p^*_i: \Ccal  \rightarrow \Ccal \otimes_R R_i$ form a product diagram. For each $i$ the functor $p_i^*$ admits an $\acal  $-linear right adjoint $(p_i)_*$, which is also left adjoint to $p_i^*$. Pick for each $i$  a compact projective $\acal\otimes_R R_i$-generator  $X_i$ for $\ccal \otimes_R R_i$. Then $\bigoplus (p_i)_* X_i$ is a compact projective $\acal $-generator for $\ccal  $. 
  \end{proof}
 
\begin{proof}[Proof of theorem \ref{theo abelian con coefficients}]
By corollary \ref{coro properties dualizable 11} we have that $\ccal$ and $\ccal^\vee$ are Grothendieck abelian categories with exact products. By lemma \ref{lemma deal with products} it suffices to show that for every point in $\Spec(R)$ with residue field $k$ such that $\acal \otimes_R k$ is semisimple there exists an \'etale morphism $R \rightarrow R'$ such that $R' \otimes_R k \neq 0$ having  the property that $\ccal \otimes_R R'$ admits a compact projective $\acal \otimes_R R'$-generator.  Applying lemma \ref{lemma deal with completes classical} we see that $\ccal \otimes_R R_{\mathfrak{p}}^\wedge$ admits a compact projective $\acal \otimes_R  R_{\mathfrak{p}}^\wedge$-generator. It now follows from lemma \ref{lemma deal with filtered colimits classical} together with Popescu's smoothing theorem that there exists a smooth $R$-algebra $S$ with the property that $S \otimes_R k \neq 0$ and $\ccal \otimes_R S$ admits a compact projective $\acal \otimes_R S$-generator.  Pick a morphism of commutative rings $S \rightarrow R'$  such that the induced map $R \rightarrow R'$ is \'etale and $R' \otimes_R k \neq 0$. Then $R'$ has the desired property.
\end{proof}


\subsection{Fully dualizable \texorpdfstring{$(\infty,1)$}{(∞,1)}-categories}\label{subsection fully dualizable infty}

The following is our main theorem concerning fully dualizable $R$-linear categories:

\begin{theorem}\label{theo prestable con coefficients}
Let $R$ be an $E_\infty$-ring such that $\pi_0(R)$ is a G-ring, and let $\Mcal$ be a symmetric monoidal $R$-linear Grothendieck prestable category. Assume the following:
\begin{itemize}
\item $\Mcal$ is rigid and generated under colimits compact projective objects.
\item $\Mcal$ is proper over $R$.
\item The set of points $x$ in $\Spec(R)$ such that $(\Mcal \otimes_R \kappa(x))^\heartsuit$ is semisimple is dense in the Zariski topology.
\end{itemize}
Let $\ccal$ be a fully dualizable $\Mcal$-linear cocomplete category. Then there exists a faithfully flat \'etale morphism of connective $E_\infty$-rings $R \rightarrow R'$ and a smooth and proper algebra $A$ in $\Mcal \otimes_R R'$  such that $\Ccal \otimes_{R} R'$ is equivalent to $\LMod_A(\Mcal \otimes_R R')$ as an $\Mcal \otimes_R R'$-linear category.
\end{theorem}

Before going into the proof, we record a few consequences.

\begin{corollary}\label{coro etale locally trivial prestable}
Let $R$ be an $E_\infty$-ring such that $\pi_0(R)$ is a G-ring and let $\Ccal$ be an invertible $\Mod^\cn_R$-linear cocomplete category. Then there exists a faithfully flat \'etale morphism of connective $E_\infty$-rings $R \rightarrow R'$ such that $\Ccal \otimes_{R} R'$ is equivalent to $\Mod^\cn_{R'}$ as an $R'$-linear category.
\end{corollary}
\begin{proof}
By theorem \ref{theo prestable con coefficients} we may after passing to a faithfully flat \'etale $R$-algebra assume that $\ccal$ is the category of left modules over a connective Azumaya $R$-algebra $A$. The corollary now follows from the fact that connective Azumaya $R$-algebras are \'etale locally Morita equivalent to the unit (\cite{SAG} theorem 11.5.7.11).
\end{proof}

\begin{notation}
Let $\mathscr{L}: \CAlg(\Sp) \rightarrow \Spc$ be the functor from connective $E_\infty$-rings into spaces which associates to each connective $E_\infty$-ring $R$ the space of $R$-linear Grothendieck prestable categories which are \'etale locally on $\Spec(R)$ equivalent to $\Mod_R^\cn$. Then $\mathscr{L}$ is a sheaf for the \'etale topology, with a pointing given by the object $\Mod^\cn_{\SS}$ in $\mathscr{L}(\SS)$.  The resulting pointed object is equivalent to $B^2\operatorname{GL}_1$. For each $\operatorname{GL}_1$-gerbe $\mathcal{G}$ on $\Spec(R)$ we will denote by $\Mod^\cn_{R, \mathcal{G}}$ the associated twist of $\Mod^\cn_{R}$.
\end{notation}

\begin{corollary}\label{coro exists gerbe invertible prestable} 
Let $R$ be an $E_\infty$-ring such that $\pi_0(R)$ is a G-ring and let $\Ccal$ be an invertible $\Mod^\cn_R$-linear cocomplete category. Then there exists a $\operatorname{GL}_1$-gerbe $\mathcal{G}$ on $\Spec(R)$ and an $R$-linear equivalence ${\ccal = \Mod^\cn_{R, \mathcal{G}}}$
\end{corollary}
\begin{proof}
This is a direct consequence of corollary  \ref{coro etale locally trivial prestable} and the definitions.
\end{proof}

\begin{corollary}\label{coro classify fully dualizables prestable} 
Let $R$ be an  $E_\infty$-ring such that $\pi_0(R)$ is a G-ring and let $\Ccal$ be a fully dualizable $\Mod^\cn_R$-linear cocomplete category. Then there exists a finite \'etale $R$-algebra $\tilde{R}$, a $\operatorname{GL}_1$-gerbe $\mathcal{G}$ on $\Spec(\tilde{R})$ and an $R$-linear equivalence $\ccal = \Mod^\cn_{\tilde{R}, \mathcal{G}} $.
\end{corollary}
\begin{proof}
Let $\tilde{R} = \End^\enh_{\Funct_R(\ccal, \ccal)}(\id_\ccal)$ be the $R$-linear center of $\ccal$. Then $\tilde{R}$ is an $E_2$ $R$-algebra, and $\ccal$ may be equipped with a canonical $\tilde{R}$-linear structure. Since $\ccal$ is dualizable the formation of $\Funct_{R}(\ccal, \ccal)$ commutes with base change, and since $\id_\ccal$ is compact projective the formation of its endomorphisms commutes with base change as well. The fact that $\ccal$ is fully dualizable implies that $\tilde{R}$ is a dualizable $R$-module.  The assertion that $\tilde{R}$ is \'etale may then be reduced by base change to the case where $R = k$ is an algebraically closed field. In this case theorem \ref{theo prestable con coefficients} implies that $\ccal$ is the category of left modules over a finite product of copies of $k$, which has \'etale center. 

We note that since $\tilde{R}$ is $E_\infty$ then $\tilde{R}$ admits a unique enhancement to an $E_\infty$ $R$-algebra. It remains to show that $\ccal =  \smash{\Mod^\cn_{\tilde{R}, \mathcal{G}}}$ for some $\operatorname{GL}_1$-gerbe $\mathcal{G}$ on $\Spec(\tilde{R})$. By corollary \ref{coro exists gerbe invertible} it suffices to show that $\ccal$ is invertible over $\tilde{R}$. This can be checked \'etale locally by virtue of the fact that $\ccal$ is a Grothendieck prestable category (corollary \ref{coro properties dualizable infty}). By an application of theorem \ref{theo prestable con coefficients} we may reduce to the case where  $\ccal$ is the category of left modules over a smooth and proper algebra $A$ in $\Mod_R^\cn$. Applying corollary \ref{coro smooth and proper locally trivial} we may further reduce to the case when $\pi_0(A)$ is a finite product of copies of $\pi_0(R)$. Since $A$ is flat over $R$ we in fact have that $A$ is a finite product of copies of $R$. In this case the claim is clear.
\end{proof}

We devote the remainder of this section to the proof of theorem \ref{theo prestable con coefficients}.

\begin{notation}\label{notation sq zero}
Let $R$ be an $E_\infty$-ring. For each $R$-module $M$ we denote by $R \oplus M$ the corresponding split square zero extension of $R$ by $M$. Recall that if $M$ is an $R$-module then an $M[1]$-valued derivation on $R$ is a morphism of $E_\infty$-rings $\delta: R \rightarrow R \oplus M[1]$ whose composition with the projection $R \oplus M[1] \rightarrow R$ recovers the identity on $R$. Such a derivation defines a square zero extension of $R$ by $M$ which we will denote by $R \oplus^\delta M$. In other words, $R \oplus^\delta M$ is defined by the pullback
\[
\begin{tikzcd}
R \oplus^\delta M \arrow{d}{} \arrow{r}{} & R \arrow{d}{(\id, 0)} \\
R \arrow{r}{\delta} & R \oplus M[1] .
\end{tikzcd}
\]
\end{notation}

\begin{lemma}\label{lemma deal with square zero isomorphisms}
Let $R$ be a connective $E_\infty$-ring. Let $M$ be a connective $R$-module and $\delta: R \rightarrow R \oplus M[1]$ be an $M[1]$-valued derivation. Let $\Ccal$ be a $R \oplus^\delta M$-linear Grothendieck prestable category and let $X, Y$ be compact projective objects of $\Ccal$. If $X \otimes_{R \oplus^\delta M} R$ is isomorphic to $Y \otimes_{R \oplus^\delta M} R$ (as objects of $\Ccal \otimes_{R \oplus^\delta M} R)$ then $X$ is isomorphic to $Y$.
\end{lemma}
\begin{proof}
Let $\Ccal_R = \Ccal \otimes_{R \oplus^\delta M} R$ and $X_R, Y_R$ be the images of $X, Y$ in $\Ccal_R$. Let $f: X_R \rightarrow Y_R$ be an isomorphism and $g: Y_R \rightarrow X_R$ be an inverse. Let $U: \Ccal_R \rightarrow \Ccal$ be the forgetful functor and $\eta_X: X \rightarrow U(X_R)$, $\eta_Y:Y \rightarrow U(X_Y)$ the unit maps. We have that $\eta_Y: Y \rightarrow U(Y_R)$ is obtained by tensoring $Y$ with  the projection $R \oplus^\delta M \rightarrow R$ and therefore induces an epimorphism on $H_0$. The projectivity of $X$ allows us to pick a lift $\overline{f}: X \rightarrow Y$ of $U(f) \circ \eta_X$ against $\eta_Y$. Similarly, we may pick a lift $\overline{g}: Y \rightarrow X$ of $U(g) \circ \eta_Y$ against $\eta_X$. Note that the images of $\overline{f}$ and $\overline{g}$ under the map $\Ccal \rightarrow \Ccal_R$ recover $f$ and $g$, respectively. 

We claim that   $\overline{f}$ and $\overline{g}$ are inverses. By symmetry it suffices to prove that $\overline{g}$ is a left inverse to $\overline{f}$. Let $h = \overline{g} \circ \overline{f}$. Then $h: X \rightarrow X$ is a lift of the identity on $X_R$. Let $A$ be the $R \oplus^\delta M$-algebra of endomorphisms of $X$, and let $A_R = A \otimes_{R \oplus^\delta M} R$. The fact that $X$ is compact projective implies that $A_R$ is the $R$-algebra of endomorphisms of $X_R$. The endomorphism $h$ defines an element $[h]$ in $\pi_0(A)$ whose image in $\pi_0(A_R)$ is the unit. We regard $\pi_0(A)$ as a classical $\pi_0(R \oplus^\delta M)$-algebra, so that $\pi_0(A_R) = \pi_0(A)/K\pi_0(A)$ where $K$ is the kernel of the map $\pi_0(R \oplus^\delta M) \rightarrow \pi_0(R)$. Since $K$ is square zero we see that $\pi_0(A)$ is a square zero extension of $\pi_0(A_R)$, and therefore $[h]$ is invertible. This shows that $h$ is an isomorphism, as desired.
\end{proof}

\begin{lemma}\label{lemma deal with square zero previa}
Let $R$ be a connective $E_\infty$-ring. Let $M$ be a connective $R$-module and  $\delta: R \rightarrow R \oplus M[1]$ be an $M[1]$-valued derivation. Let $\Ccal$ be a separated $R \oplus^\delta M$-linear Grothendieck prestable category. Then every compact projective object of $\Ccal \otimes_{R \oplus^\delta M} R$ admits a lift to a compact projective object of $\Ccal$.
\end{lemma}
\begin{proof}
Let $\Ccal_{R} =\Ccal \otimes_{R \oplus^\delta M} R$ and $\Ccal_{R \oplus M[1]} =\Ccal \otimes_{R \oplus^\delta M} (R \oplus M[1])$. Fix a compact projective object $X$ of $\Ccal_R$. Applying \cite{SAG} proposition 16.2.2.1 to the commutative square from notation \ref{notation sq zero} we obtain a pullback square of categories
\[
\begin{tikzcd}
\Ccal \arrow{r}{} \arrow{d}{} & \Ccal_R \arrow{d}{0^*} \\
\Ccal_R \arrow{r}{\delta^*} & \Ccal_{R \oplus M[1]}.
\end{tikzcd}
\] 
We have that both $\delta^*X$ and $0^*X$ have image $X$ under extension of scalars along $R \oplus M[1] \rightarrow R$. It follows from lemma \ref{lemma deal with square zero isomorphisms} that there exists an isomorphism $\delta^*X = 0^*X$, so that we may identify both objects with a (compact projective) object of $\Ccal_{R \oplus M[1]}$ which we denote $X_{R \oplus M[1]}$. The triple $(X, X_{R \oplus M[1]}, X)$ defines an object of $\Ccal$ which we will denote $\overline{X}$. This is a lift of $X$, so to prove the lemma it will suffice to show that $\overline{X}$ is compact projective.

We first show that $\overline{X}$ is compact. Assume given a filtered diagram $Y_\alpha$ in $\Ccal$ and denote by $(Y_\alpha)_R$ and $(Y_\alpha)_{R \oplus M[1]}$ its base changes. We want to show that the  map $\colim \Hom_{\ccal}(\overline{X}, Y_\alpha) \rightarrow \Hom_{\ccal}(\overline{X}, \colim Y_\alpha)$ is an isomorphism. This follows from the compactness of $X$ and $X_{R\oplus M[1]}$, since this map is the pullback of the map  
\[
\colim \Hom_{\ccal_R}(X, (Y_\alpha)_R) \rightarrow  \Hom_{\ccal_R}(X,\colim  (Y_\alpha)_R )
\]
with itself over
\[
\colim \Hom_{\ccal_{R\oplus M[1]}}(X_{R\oplus M[1]}, (Y_\alpha)_{R\oplus M[1]}) \rightarrow  \Hom_{\ccal_{R\oplus M[1]}}(X_{R\oplus M[1]},\colim  (Y_\alpha)_{R\oplus M[1]} ).
\]

It remains to prove that $\overline{X}$ is projective. This amounts to showing that if $Y$ is an object of $\Ccal$ then $\pi_0 \Hom_{\ccal}(\overline{X}, \Sigma Y) = 0$. Let $Y_R = Y \otimes_{R \oplus^\delta M} R$ and $Y_{R \oplus M[1]} = Y \otimes_{R \oplus^\delta M} (R \oplus M[1])$. We have a pullback square
\[
\begin{tikzcd}
 \Hom_{\ccal}(\overline{X}, \Sigma Y)\arrow{r}{} \arrow{d}{} & \Hom_{\ccal_R}(X, \Sigma Y_R) \arrow{d}{0^*} \\
  \Hom_{\ccal_R}(X, \Sigma Y_R)  \arrow{r}{\delta^*} & \Hom_{\Ccal_{R \oplus M[1]}}(X_{R \oplus M[1]},\Sigma Y_{{R \oplus M[1]}}).
\end{tikzcd}
\]
The projectivity of $X$ and $X_{R \oplus M[1]}$ guarantees that 
\[
\pi_0(\Hom_{\ccal_R}(X, \Sigma Y_R)) = \pi_0( \Hom_{\Ccal_{R \oplus M[1]}}(X_{R \oplus M[1]},\Sigma Y_{{R \oplus M[1]}})) = 0.
\] We may thus reduce to showing that the right vertical map is a surjection on $\pi_1$. This is the same as showing that the map $0^*: \Hom_{\ccal_R}(X,  Y_R) \rightarrow \ \Hom_{\Ccal_{R \oplus M[1]}}(X_{R \oplus M[1]}, Y_{{R \oplus M[1]}})$ is a surjection on $\pi_0$. In other words, we have to prove that any map $X_{R \oplus M[1]} \rightarrow Y_{{R \oplus M[1]}}$ lifts to a map $ X \rightarrow Y_R$. This is the same as showing that the composite map $X \rightarrow X_{R \oplus M[1]} \rightarrow Y_{R \oplus M[1]}$ admits a factorization through $Y_R \rightarrow Y_{R \oplus M[1]}$ (where here we identify $ X_{R \oplus M[1]}$ and $ Y_{R \oplus M[1]}$ with their restriction of scalars along $(\id, 0): R \rightarrow R \oplus M[1]$).  This follows from the projectivity of $X$.
\end{proof}

\begin{lemma}\label{lemma deal with square zero para prestable}
Let $R$ be a connective $E_\infty$-ring. Let $M$ be a connective $R$-module and  $\delta: R \rightarrow R \oplus M[1]$ be an $M[1]$-valued derivation. Let $\Ccal$ be a separated $R \oplus^\delta M$-linear Grothendieck prestable category. Let $X_\alpha$ be a family of compact projective objects of $\ccal$ whose image in  $\Ccal \otimes_{R \oplus^\delta M} R$ is a generating family. Then the family $X_\alpha$ generates $\Ccal$ under colimits.
\end{lemma}
\begin{proof}
Let $\Ccal'$ be the full subcategory of $\Ccal$ generated under colimits by the family $X_\alpha$. We want to show that $\ccal' = \ccal$. Let $\ccal_R, \ccal_{R \oplus M[1]}, \ccal'_R, \ccal'_{R \oplus M[1]}$ be the base changes of $\ccal$ and $\ccal'$.  Applying \cite{SAG} proposition 16.2.2.1 to the commutative square from notation \ref{notation sq zero} we see that $\Ccal = \Ccal_R \times_{\Ccal_{R \oplus M[1]}} \Ccal_R$ and $\Ccal' = \Ccal'_R \times_{\Ccal'_{R \oplus M[1]}} \Ccal'_R$. We may thus reduce to showing that $\ccal'_R = \Ccal_R$. We note that $\ccal'_R$ is a full subcategory of $\ccal_R$ closed under colimits and containing the image of the composite functor $\ccal' \rightarrow \Ccal \rightarrow \ccal_R$. In particular, $\ccal'_R$ contains the objects $X_\alpha \otimes_{R \oplus^\delta M} R$.  Our claim now follows from the fact that this family was assumed to generate $\ccal_R$ under colimits.
\end{proof}

\begin{lemma}\label{lemma completes equiv postnikov}
Let $R$ be a connective $E_\infty$-ring and let $\Ccal$ be an $R$-linear Grothendieck prestable category. Assume that $\ccal$ is separated and that products in $\Sp(\ccal)$ are t-exact. Then the functor $p:  \Ccal \rightarrow \lim \Ccal \otimes_R \tau_{\leq n}R$ is an equivalence.
\end{lemma}
\begin{proof}
For each $n \geq 0$ let $R_n = \tau_{\leq n} R$.  We first prove that $p^R$ is fully faithful, by showing that the counit $p p^R(M_n) \rightarrow (M_n)$ is an isomorphism for all sequences $(M_n)$. This amounts to showing that the canonical map $\mu: (\lim M_n) \otimes_R R_s \rightarrow M_s$ is an isomorphism for all $s \geq 0$. This map factors, for each $t \geq s$, as a composition 
\[
(\lim M_n) \otimes_R R_s \xrightarrow{\mu_1} M_t \otimes_R R_s \xrightarrow{\mu_2} M_s
\]
where $\mu_1$ is induced from the projection $\lim M_n \rightarrow M_t$, and $\mu_2$ is induced from the $R$-linear map $M_t \rightarrow M_s$.

For each $n \geq 0$ the transition $M_{n+1} \rightarrow M_n$ is obtained by tensoring the $(n+1)$-connective map $R_{n+1} \rightarrow R_n$ with $M_{n+1}$. In particular, it is itself $(n+1)$-connective. Since products in $\Sp(\ccal)$ are t-exact we have that the map $(\lim M_n) \rightarrow M_t$ is $(t+1)$-connective and therefore $\mu_1$ is $(t+1)$-connective as well. 

The map $\mu_2$ is equivalent to the induction along $R_t \rightarrow R_s$ of the map $\mu_2': M_t \otimes_R R_t \rightarrow M_t$ induced from the identity on $M_t$. The map $\mu_2'$ is obtained by tensoring $M_t$ with the $(t+1)$-connective map $R_t \otimes_R R_t \rightarrow R_t$, and is therefore $(t+1)$-connective. It follows that $\mu_2$ is also $(t+1)$-connective. Now $\mu$ is a composition of $(t+1)$-connective maps so it is $(t+1)$-connective. Letting $t \rightarrow \infty$ we see that $\mu$ is $\infty$-connective. The fact that $\ccal$ is separated now implies that $\mu$ is an equivalence, as desired.

To prove that $p$ is an equivalence it now suffices to show that it is conservative. Since every morphism in $\Mod_R^\cn$ is the fiber of its cofiber it is enough to prove that if $X$ is an object of $\ccal$ such that $p(X) = 0$ then $X = 0$. Since $\ccal$ is separated it is enough for this to prove that $H_t(X) = 0$ for all $t \geq 0$. Arguing by induction on $t$, we may assume that $H_s(X) = 0 $ for all $s < t$. Replacing $X$ by $\Omega^t(X)$ we may reduce to the case $t = 0$. Since $p(X) = 0$ we have in particular that $X \otimes_R \pi_0(R) = 0$. The fact that $H_0(X) = 0$ now follows from the fact that the map $X \rightarrow X \otimes_R \pi_0(R)$ is $1$-connective.
\end{proof}

\begin{lemma}\label{lemma check compact projective by tensoring}
Let $R$ be a connective $E_\infty$-ring and let $\Ccal$ be an $R$-linear Grothendieck prestable category. Assume that $\ccal$ is separated and that products in $\Sp(\ccal)$ are t-exact. Let $X$ be an object of $\ccal$. Then $X$ is compact projective if and only if $X \otimes_R \tau_{\leq n} R$ is a compact projective object of $\Ccal \otimes_R \tau_{\leq n} R$ for all $n \geq 0$.
\end{lemma}
\begin{proof}
For each $n \geq 0$ set $R_n = \tau_{\leq n} R$, $\ccal_n =\ccal \otimes_R R_n$, and $X_n = X \otimes_R R_n$. The only if direction follows from the fact that the extension of scalars functors $\ccal \rightarrow \ccal_n$ admit colimit preserving right adjoints. It remains to prove the if direction.

We first show that $X$ is projective. To prove this it suffices to show that if $Y$ is an object of $\ccal$ then $\Hom_\ccal(X, \Sigma Y)$ is connected. Set $Y_n = Y \otimes_R R_n$ for all $n \geq 0$. Applying lemma \ref{lemma completes equiv postnikov} we  have $\Hom_\ccal(X, \Sigma  Y) = \lim \Hom_{\ccal_n}(X_n, \Sigma  Y_n)$. Since $X_n$ is projective for all $n$ each of the spaces $\Hom_{\ccal_n}(X_n, \Sigma  Y_n)$ is connected. We may therefore reduce to proving that the transitions $\Hom_{\ccal_{n+1}}(X_{n+1}, \Sigma  Y_{n+1}) \rightarrow \Hom_{\ccal_n}(X_n, \Sigma  Y_n)$ induce surjections on $\pi_1$. This is the same as showing that the transitions $\Hom_{\ccal_{n+1}}(X_{n+1}, Y_{n+1}) \rightarrow \Hom_{\ccal_n}(X_n,  Y_n)$ induce surjections on $\pi_0$. In other words, we have to show that every map $X_{n+1} \rightarrow Y_n$ in $\ccal_{n+1}$ factors through $Y_{n+1}$. This is a consequence of the fact that $X_{n+1}$ is projective and the morphism $Y_{n+1} \rightarrow Y_n$ is $0$-connective.

We now show that $X$ is compact. Assume given a filtered diagram $Y_\alpha$ in $\ccal$. We wish to prove that $\colim \Hom_{\ccal}(X, Y_\alpha) = \Hom_{\ccal}(X, \colim Y_\alpha)$. We will do so by proving that it induces an isomorphism on $\pi_t$ for all $t \geq 0$. Using the fact that $X$ is projective we obtain equivalences 
\[
\pi_t(\Hom_{\ccal}(X, Y_\alpha)) = \pi_t(\Hom_{\ccal}(X, \Sigma^t H_t(Y_\alpha)))
\]
 and 
 \[\pi_t(\Hom_{\ccal}(X, \colim Y_\alpha)) = \pi_t(\Hom_{\ccal}(X, \colim \Sigma^t  H_t(Y_\alpha))).
\]
We may therefore reduce to proving that we have an equivalence $\colim \Hom_\Ccal(X, \Sigma^t H_t(Y_\alpha)) = \Hom_\ccal(X, \colim \Sigma^t H_t(Y_\alpha))$. Replacing $Y_\alpha$ with $\Sigma^t H_t(Y_\alpha)$ we may now reduce to the case when $Y_\alpha$ is $t$-truncated for all $\alpha$. To prove this it suffices to show that $\tau_{\leq t}(X)$ is a compact object of $\ccal_{\leq t}$. Since the functor $\Mod_R^\cn \rightarrow \Mod_{R_t}^\cn$ induces an equivalence on $t$-truncated objects we have that the functor $\ccal \rightarrow \ccal_t$ induces an equivalence on $t$-truncated objects as well. Hence we may reduce to proving that $\tau_{\leq t}(X_t)$ is a compact object of $(\ccal_t)_{\leq t}$. This is a consequence of the fact that $X_t$ is compact in $\ccal_t$.
\end{proof}

\begin{lemma}\label{lemma deal with postnikov}
Let $R$ be a connective $E_\infty$-ring and let $\Mcal$ be a symmetric monoidal $R$-linear Grothendieck prestable category. Assume that $\Mcal$ is rigid and generated under colimits by compact projective objects. Let $\ccal$ be a separated $\Mcal$-linear Grothendieck prestable category such that products in $\Sp(\ccal)$ are t-exact. If $\ccal \otimes_R \pi_0(R)$ admits a compact projective $\Mcal\otimes_R \pi_0(R)$-generator then $\ccal$ admits a compact projective $\Mcal$-generator.
\end{lemma}
\begin{proof}
For each $n \geq 0$ set $R_n = \tau_{\leq n} R$, $\Mcal_n = \Mcal \otimes_R R_n$ and $\Ccal_n = \Ccal \otimes_R R_n$.  Fix a compact projective $\Mcal\otimes_R \pi_0(R)$-generator $X_0$ for $\ccal_0$. Applying lemma \ref{lemma deal with square zero previa} inductively we may construct a compatible sequence of compact projective objects $X_n$ in $\ccal_n$.  An inductive application of lemma \ref{lemma deal with square zero para prestable} implies that the family of objects obtained by tensoring $X_n$ with a compact projective object of $\Mcal$ generates $\ccal_n$ under colimits. It follows that $X_n$ is a compact projective $\Mcal_n$-generator of $\ccal_n$ for every $n \geq 0$.

By lemma \ref{lemma completes equiv postnikov} there exists an object $X$ in $\Ccal$ such that $X \otimes_R R_n = X_n$ for all $n$. It follows from lemma \ref{lemma check compact projective by tensoring} that $X$ is compact projective. It remains to show that $X$ is an $\Mcal$-generator for $\ccal$. Let $\ccal'$ be the smallest full subcategory of $\ccal$ containing $X$ and closed under colimits and the action of $\Mcal$. Then $\ccal'$ is a $\Mcal$-linear Grothendieck prestable category generated under colimits by compact projective objects. We wish to show that the inclusion $\ccal' \rightarrow \ccal$ is an equivalence. By lemma \ref{lemma completes equiv postnikov} it is enough to prove that $\ccal' \otimes_R R_n = \ccal_n$ for all $n \geq 0$. Note that $\ccal' \otimes_R R_n$ is a full subcategory of $\ccal_n$  closed under colimits and the action of $\Mcal_n$, and containing $X_n$. Our claim now follows from the fact that $X_n$ is a $\Mcal_n$-generator for $\ccal_n$.
\end{proof}

\begin{proof}[Proof of theorem \ref{theo prestable con coefficients}]
By corollary \ref{coro properties dualizable infty} we have that $\ccal$ is a separated Grothendieck prestable category and products in $\Sp(\ccal)$  are t-exact. Applying theorem \ref{theo abelian con coefficients} we may assume, after changing base to a faithfully flat \'etale $R$-algebra, that $\ccal^\heartsuit$ admits a compact projective $\Mcal^\heartsuit$-generator.   By proposition \ref{prop lex localization prestable} it suffices to show that $\ccal$ admits a compact projective $\Mcal$-generator. Applying lemma \ref{lemma deal with postnikov} we may reduce to the case when $R$ is $0$-truncated. It now suffices to show that $\ccal$ is the connective derived category of its heart.  For this it is enough to prove that $\ccal^\heartsuit$ generates $\ccal$ under colimits.

Let $\ccal'$ be the full subcategory of $\ccal$ generated under colimits by $\ccal^\heartsuit$, and note that $\ccal'$ inherits an $\Mcal$-linear structure from $\ccal$. Let $\delta$ be the image of $1_{\Mcal}$ under the unit map $\Mcal \rightarrow \ccal \otimes_\Mcal \ccal^\heartsuit$. The fact that $\ccal$ is fully dualizable implies that $\delta$ is compact projective and $0$-truncated. Since $\ccal \otimes_\Mcal \ccal^\heartsuit$ is generated by the objects of the form $X \otimes Y$ with $X$ in $\ccal$ and $Y$ in $\ccal^\heartsuit$, we may find finite sequences of objects $X_i$ in $\ccal$ and $Y_i$ in $\ccal^\vee$ and a morphism $f: \bigoplus (X_i \otimes Y_i) \rightarrow \delta$ which induces an epimorphism on $H_0$. The fact that $\delta$ is $0$-truncated implies that $f$ factors through $\bigoplus (H_0(X_i) \otimes Y_i)$, and since $\delta$ is projective we see that $\delta$ is a direct summand of $\bigoplus (H_0(X_i) \otimes Y_i)$. It follows in particular that $\delta$ belongs to $\ccal' \otimes_\Mcal \ccal^\vee$, which implies that the identity on $\ccal$ belongs to the image of the inclusion $\Funct_\Mcal(\ccal, \ccal') \rightarrow \Funct_\Mcal(\ccal, \ccal)$. In other words, the inclusion $\ccal' \rightarrow \ccal$ admits a section, which implies that $\ccal' = \ccal$, as desired.
\end{proof}


\subsection{Rings of definition of fully dualizable categories}\label{subsection rings of def}

Our next goal is to discuss how to extend the above results beyond the context of G-rings, under an additional compact generation hypothesis. The basic mechanism is given by the following:

\begin{proposition}\label{prop lift fully dualizables}
Let $\Mcal_\alpha$ be a filtered diagram of commutative algebras in $\Pr^L$ with compact transition functors. Assume that for every $\alpha$ the category $\Mcal_\alpha$ is generated under colimits by dualizable objects and has compact unit.  Let $\ccal$ be a fully dualizable $\Mcal$-linear cocomplete category, and assume that $\ccal$ and $\ccal^\vee$ are generated under colimits by compact objects. Then there exists an index $\alpha$, a fully dualizable $\Mcal_\alpha$-linear cocomplete category $\ccal_\alpha$, and an equivalence $\ccal_\alpha \otimes_{\Mcal_\alpha} \Mcal =\ccal$. Furthermore, if $\ccal$ is invertible then $\ccal_\alpha$ may be chosen to be invertible as well.
\end{proposition}
\begin{proof}
 Let $\Pr^L_\omega$ be the subcategory of $\Pr^L$ on the compactly generated categories and compact functors. Note that $\Pr^L_\omega$ is a compactly generated presentable category and the inclusion into $\Pr^L$ preserves colimits. Furthermore, the symmetric monoidal structure on $\Pr^L$ restricts to a symmetric monoidal structure on $\Pr^L_\omega$, which in turn restricts to a symmetric monoidal structure on its full subcategory $(\Pr^L_\omega)^\omega$ on the compact objects. 
 
 Our assumptions imply that $\Mcal_\alpha$ is a commutative algebra in $\Pr^L_\omega$ for every $\alpha$, and therefore $\Mcal$ is also a commutative algebra in $\Pr^L_\omega$. We now have $\Mod_{\Mcal}(\Pr^L_\omega) = \colim \Mod_{\Mcal_\alpha}(\Pr^L_\omega)$ which implies, after passing to compact objects, that $\Mod_{\Mcal}(\Pr^L_\omega)^\omega = \colim \Mod_{\Mcal_\alpha}(\Pr^L_\omega)^\omega$.

By proposition \ref{proposition dualizable is presentable} we have that $\ccal$ and its dual are presentable. Since $\Mcal_\alpha$ is generated under colimits by dualizable objects for all $\alpha$ the same thing holds for $\Mcal$, which implies that the action functors $\Mcal \times \ccal \rightarrow \ccal$ and $\Mcal \times \ccal^\vee \rightarrow \ccal^\vee$ preserve compact objects. It follows that $\ccal$ and $\ccal^\vee$ belong to $\Mod_{\Mcal}(\Pr^L_\omega)$. Combining this with the fact that $\ccal$ is fully dualizable in $\Mod_\Mcal(\Pr^L)$ we see that $\ccal$ is a fully dualizable object of $\Mod_{\Mcal}(\Pr^L_\omega)$. Since the unit in $\Mod_{\Mcal}(\Pr^L_\omega)$ is compact we have that $\ccal$ is a fully dualizable object of $\Mod_{\Mcal}(\Pr^L_\omega)^\omega$. The proposition now follows from the fact that $\Mod_{\Mcal}(\Pr^L_\omega)^\omega = \colim \Mod_{\Mcal_\alpha}(\Pr^L_\omega)^\omega$, since the functor that sends each symmetric monoidal $2$-category to its space of fully dualizable (resp. invertible) objects preserves filtered colimits.
\end{proof}
 
 \begin{corollary}\label{coro rings of definitions 11}
Let $\acal$ be a symmetric monoidal Grothendieck abelian category, rigid, generated by compact projective objects, and proper over $\ZZ$. Assume that for every field $k$ the Grothendieck abelian category $\acal \otimes k$ is semisimple. Let $R$ be a commutative ring and let $\ccal$ be a fully dualizable  $\Acal \otimes R$-linear cocomplete category. Assume that $\ccal$ and $\ccal^\vee$ are compactly generated. Then there exists a subalgebra $S \subseteq R$ of finite type over $\ZZ$, a fully dualizable $\acal \otimes S$-linear cocomplete category $\dcal$, and an $\acal \otimes R$-linear equivalence $\dcal \otimes_S R = \ccal$. Furthermore, if $\ccal$ is invertible then $\dcal$ may be chosen to be invertible as well.
 \end{corollary}
 \begin{proof}
 Specialize proposition \ref{prop lift fully dualizables} to the filtered diagram indexed by the poset of finite type subalgebras of $R$, that sends a subalgebra $S$ to $\acal \otimes S$.
 \end{proof}

\begin{corollary}\label{coro los coros classicos extienden}
Corollaries \ref{coro etale locally trivial},  \ref{coro exists gerbe invertible}, and \ref{coro classify fully dualizables abelian} hold over an arbitrary commutative ring, provided that $\ccal$ and $\ccal^\vee$ are compactly generated.
\end{corollary} 
\begin{proof}
Follows directly from the case $\acal = \Ab$ of corollary \ref{coro rings of definitions 11} together with the fact that finite type commutative rings are G-rings.
\end{proof}

In the same way, we have:

\begin{corollary}
Corollaries \ref{coro etale locally trivial prestable}, \ref{coro exists gerbe invertible prestable} and \ref{coro classify fully dualizables prestable} hold over an arbitrary connective $E_\infty$-ring, provided that $\ccal$ and $\ccal^\vee$ are compactly generated.
\end{corollary}

\begin{remark}
Corollary \ref{coro etale locally trivial prestable} was proven for arbitrary connective $E_\infty$-rings in \cite{SAG} theorem 11.5.7.11, under the condition that $\ccal$ and $\ccal^\vee$ are compactly generated Grothendieck prestable categories.
\end{remark}


\subsection{Invertible stable categories}\label{subsection invertible stable over R}

We finish with a classification of invertible stable categories over truncated connective $E_\infty$-rings.

\begin{theorem}\label{theorem stable over R}
Let $R$ be a truncated connective $E_\infty$-ring and let $\Ccal$ be an invertible $R$-linear cocomplete stable category. Then $\Ccal = \LMod_A(\Mod_R)$ for some Azumaya algebra $A$ in $\Mod_R$.
\end{theorem}

The remainder of this section is devoted to a proof of theorem \ref{theorem stable over R}.

\begin{lemma}\label{lemma conservative pullback}
Let 
\[
\begin{tikzcd}
R \arrow{r}{} \arrow{d}{} & R_0 \arrow{d}{} \\
R_1 \arrow{r}{} & R_{01}
\end{tikzcd}
\]
be a pullback diagram of $E_\infty$-rings. Assume given a morphism $f: \Ccal \rightarrow \Dcal$ in $\Mod_{\Mod_R}(\Pr^L)$. If the induced functors $\Ccal \otimes_{R} R_0 \rightarrow \Dcal \otimes_R R_0$ and $\Ccal \otimes_R R_1 \rightarrow \Dcal \otimes_R R_1$ are equivalences then $f$ is an equivalence.
\end{lemma}
\begin{proof}
Let $\Ccal_0 = \Ccal \otimes_R R_0$, $\Ccal_1 = \Ccal \otimes_R R_1$ and $\Ccal_{01} = \Ccal \otimes_R R_{01}$, and define $\Dcal_0, \Dcal_1, \Dcal_{01}$ similarly. We have a commutative square of categories
\[
\begin{tikzcd}
\Ccal \arrow{d}{f} \arrow{r}{} & \Ccal_0 \times_{\Ccal_{01}} \Ccal_1 \arrow{d}{} \\
\Dcal \arrow{r}{} & \Dcal_0 \times_{\Dcal_{01}} \Dcal_1
\end{tikzcd}
\]
where the right vertical arrow is an isomorphism by our hypothesis. The horizontal arrows in the above square are fully faithful by \cite{SAG} proposition 16.2.1.1, and hence $f$ is fully faithful. To finish the proof it will suffice to show that $\Dcal/\Ccal = 0$. We have $(\Dcal/\Ccal)\otimes_{R} R_0 = \Dcal_0/\Ccal_0 = 0$ and $(\Dcal/\Ccal) \otimes_R R_1 = \Dcal_1 / \Ccal_1 = 0$. Hence the projection $\Dcal/\Ccal \rightarrow 0$ induces equivalences after tensoring with $R_0$ and $R_1$. It follows that the projection $\Dcal/\Ccal \rightarrow 0$ is fully faithful, and hence it is an isomorphism, as desired.
\end{proof}

\begin{lemma}\label{lemma deal with square zero para stable}
Let $R$ be a connective $E_\infty$-ring, let $M$ be a connective $R$-module and $\delta: R \rightarrow R \oplus M[1]$ be an $M[1]$-valued derivation. Let $\Ccal$ be an $R \oplus^\delta M$-linear presentable stable category and assume given an $R$-linear equivalence $\varphi: \Ccal \otimes_{R \oplus^\delta M} R  = \Mod_R$. Then there is an $R \oplus^\delta M$-linear equivalence $\Ccal = \Mod_{R \oplus^\delta M}$ which recovers $\varphi$ after tensoring with $R$.
\end{lemma}
\begin{proof}
We have a commutative diagram of categories with invertible vertical arrows
\[
\begin{tikzcd}[column sep = small]
\Ccal \otimes_{R \oplus^\delta M } R \arrow{r}{}  \arrow{d}{\varphi} & \Ccal \otimes_{R \oplus^\delta M } R \otimes_R ( R \oplus M[1] ) \arrow{r}{=} \arrow{d}{} &  \arrow{d}{} \Ccal \otimes_{R \oplus^\delta M } R \otimes_R ( R \oplus M[1] ) & \arrow{l}{} \Ccal \otimes_{R \oplus^\delta M} R  \arrow{d}{\varphi}\\
\Mod_R \arrow{r}{} & \Mod_{R \oplus M[1]} \arrow{r}{\psi} & \Mod_{R \oplus M[1]} & \arrow{l}{} \Mod_R.
\end{tikzcd}
\]
where the leftmost horizontal arrows are given by induction along $\delta$, the rightmost horizontal arrows are given by induction along $(\id,0)$, and the top middle horizontal arrow is induced from the commutativity of the square in notation \ref{notation sq zero}. The equivalence $\psi$ is $R \oplus M[1]$-linear and is therefore given by tensoring with an invertible $R \oplus M[1]$-module $L$. Since $\delta$ is an isomorphism on $\pi_0$, there exists an invertible $R$-module $L'$ whose extension of scalars along $\delta$ recovers $L$. We may now extend the above commutative diagram as follows:
\[
\begin{tikzcd}[column sep = small]
\Ccal \otimes_{R \oplus^\delta M } R \arrow{r}{}  \arrow{d}{\varphi} & \Ccal \otimes_{R \oplus^\delta M } R \otimes_R ( R \oplus M[1] ) \arrow{r}{=} \arrow{d}{} &  \arrow{d}{} \Ccal \otimes_{R \oplus^\delta M } R \otimes_R ( R \oplus M[1] ) & \arrow{l}{} \Ccal \otimes_{R \oplus^\delta M} R  \arrow{d}{\varphi}\\
\Mod_R \arrow{r}{} \arrow{d}{-\otimes_R L'} & \Mod_{R \oplus M[1]} \arrow{r}{\psi} \arrow{d}{- \otimes_{R \oplus M[1]} L} & \Mod_{R \oplus M[1]} \arrow{d}{\id} & \arrow{l}{} \Mod_R \arrow{d}{\id} \\
\Mod_R \arrow{r}{} & \Mod_{R \oplus M[1]} \arrow{r}{\id} & \Mod_{R \oplus M[1]} & \arrow{l}{} \Mod_R 
\end{tikzcd}
\]
Passing to pullbacks of the first and third rows we obtain an $R \oplus^\delta M$-linear equivalence
\[
\xi:( \Ccal \otimes_{R \oplus^\delta M} R) \times_{\Ccal \otimes_{R \oplus^\delta M} (R \oplus M[1])} ( \Ccal \otimes_{R \oplus^\delta M} R)  = \Mod_R \times_{\Mod_{R \oplus M[1]}} \Mod_R.
\]
Denote by 
\[
\iota: \Mod_{R \oplus^\delta M} \rightarrow \Mod_R \times_{\Mod_{R \oplus M[1]}} \Mod_R
\]
 and 
 \[
 \iota_{\Ccal}: \Ccal \rightarrow ( \Ccal \otimes_{R \oplus^\delta M} R) \times_{\Ccal \otimes_{R \oplus^\delta M} (R \oplus M[1])} ( \Ccal \otimes_{R \oplus^\delta M} R) 
 \]
 the canonical functors. By \cite{SAG} proposition 16.2.1.1 both $\iota$ and $\iota_\Ccal$ are fully faithful.
 
 We claim that $\xi \circ \iota_\Ccal$ factors through $\iota$. Let $X$ be an object of $\Ccal$. We wish to show that $\xi \iota_\Ccal(X)$ belongs to the image of $\iota$. We have 
 \[
 X = (X \otimes_{R \oplus^\delta M} R) \times_{X \otimes_{R \oplus^\delta M} (R \oplus M[1])} (X \otimes_{R \oplus^\delta M} R)
 \]
 so it suffices to prove that $\xi \iota_\Ccal( X \otimes_{R \oplus^\delta M} R)$ and $\xi \iota_\Ccal(X \otimes_{R \oplus^\delta M} (R \oplus M[1]))$ belong to the image of $\iota$. In other words, we may reduce to the case when $X$ is obtained by restriction of scalars along $R \oplus^\delta M \rightarrow R$.  Since $\Ccal \otimes_{R \oplus^\delta M} R$ is equivalent to $\Mod_R$, which is generated under colimits and shifts by $R$, we may further reduce to the case when $X$ is given by restriction of scalars of $\varphi^{-1}(R)$. In this case the image of $X$ in $\Ccal \otimes_{R \oplus^\delta M} R$ is given by $\varphi^{-1}(R \otimes_{R \oplus^\delta M} R)$. Hence $\xi \iota_\Ccal(X)$ is an object of $\Mod_R \times_{\Mod_{R \oplus M[1]}} \Mod_R$ whose coordinates are given by 
 \[
 (R \otimes_{R \oplus^\delta M} R \otimes_R L', R \otimes_{R \oplus^\delta M} (R\oplus M[1]), R \otimes_{R \oplus^\delta M} R).
\]
It follows that $\xi \iota_\Ccal(X)$ is a shift of an object in $\Mod_R^\cn \times_{\Mod^\cn_{R \oplus M[1]}} \Mod_R^\cn$. The fact that it belongs to the image of $\iota$ now follows from \cite{SAG} theorem 16.2.0.2.

We now have a well defined $R \oplus^\delta M$-linear functor $f: \Ccal \rightarrow \Mod_{R \oplus^\delta M}$ with the property that it recovers the equivalence $\varphi$ after tensoring with $R$. The proof concludes by an application of lemma \ref{lemma conservative pullback}.
\end{proof}

\begin{lemma}\label{lemma deal with filtered colimits stable}
Let $R_\alpha$ be a filtered diagram of commutative rings with colimit $R$. Assume given an index $\alpha_0$ and a smooth $R_{\alpha_0}$-linear presentable stable category $\ccal_{\alpha_0}$. If $\ccal_{\alpha_0} \otimes_{R_{\alpha_0}} R$ has a compact generator then there exists a transition $\alpha_0 \rightarrow \alpha$ such that $\ccal_{\alpha_0} \otimes_{R_{\alpha_0}} R_\alpha$ has a compact generator.
\end{lemma}
\begin{proof}
Analogous to the proof of lemma \ref{lemma deal with filtered colimits classical}.
\end{proof}

\begin{lemma}\label{lemma postnikov compact generator}
Let $R$ be a truncated connective $E_\infty$-ring and let $\ccal$ be an invertible $R$-linear presentable stable category. Assume that $\ccal \otimes_R \pi_0(R)$ admits a compact generator. Then $\ccal$ admits a compact generator.
\end{lemma}
\begin{proof}
By \cite{AGBrauer} theorem 5.11 we may find a faithfully flat \'etale $\pi_0(R)$ algebra $S_0$ such that $\ccal \otimes_{R} S_0$ is equivalent to $\Mod_{S_0}$ as an $S_0$-linear category. Let $S$ be a faithfully flat \'etale $R$-algebra such that $S \otimes_R \pi_0(R) = S_0$. We have $\pi_0(S) = S_0$, and hence $(\ccal \otimes_R S) \otimes_S \pi_0(S)$ is equivalent to $\Mod_{\pi_0(S)}$ as a $\pi_0(S)$-linear category. An inductive application of lemma \ref{lemma deal with square zero para stable} shows that $\ccal \otimes_R S$ is equivalent to $\Mod_S$. The lemma now follows from \cite{AGBrauer} theorem 6.16. 
\end{proof}

\begin{notation}
Let $R$ be a commutative ring and let $\ccal$ be an $R$-linear presentable stable category. Let $x$ be an element of $R$. We denote by $\ccal_{x\normalfont{\text{-nil}}}$ the kernel of the extension of scalars functor $\ccal \rightarrow \ccal \otimes_R R[x^{-1}]$.
\end{notation}

\begin{lemma}\label{lemma compact generator on completion}
Let $R$ be a commutative ring and $\ccal$ be an invertible $R$-linear presentable stable category. Let $x$ be an element of $R$ and assume that $\ccal \otimes_R R/Rx$ admits a compact generator. Then $\ccal_{x\normalfont{\text{-nil}}}$ admits a compact generator.
\end{lemma}
\begin{proof}
By \cite{AGBrauer} theorem 5.11 we may find a faithfully flat \'etale $R/Rx$ algebra $S_0$ such that $\ccal \otimes_R S_0$ is equivalent to $\Mod_{S_0}$ as an $S_0$-linear category. Let $S$ be a faithfully flat \'etale $R$-algebra such that $S \otimes_R R/Rx = S_0$. By \cite{AGBrauer} theorem 6.16. it suffices to show that $\ccal_{x\normalfont{\text{-nil}}} \otimes_R S$ admits a compact generator.  This category is equivalent to $(\ccal \otimes_R S)_{y\normalfont{\text{-nil}}}$ where $y$ is the image of $x$ in $S$. Replacing $R$ by $S$ we may now assume that $\ccal \otimes_R R/Rx$ is equivalent to $\Mod_{R/Rx}$ as an $R/Rx$-linear category.

An iterated application of lemma \ref{lemma deal with square zero para stable} identifies the inverse system of categories
\[
\ccal \otimes_R R/Rx \leftarrow \ccal \otimes_R R/Rx^2 \leftarrow \ccal \otimes_R R/Rx^3 \leftarrow \ldots
\]
with the inverse system
\[
\Mod_{R/Rx} \leftarrow \Mod_{R/Rx^2} \leftarrow \Mod_{R/Rx^3} \leftarrow \ldots.
\]
Passing to limits, we obtain an equivalence
\[
\ccal_{x\normalfont{\text{-nil}}} = (\Mod_R)_{x\normalfont{\text{-nil}}}
\]
and the lemma now follows from the fact that the right hand side admits a compact generator (namely, the cofiber of $x: R \rightarrow R$).
\end{proof}

\begin{lemma}\label{lemma handle extensions}
Let $R$ be a commutative ring and $\ccal$ be an invertible $R$-linear presentable stable category. Let $x$ be an element of $R$ and assume that both $\ccal \otimes_R R/Rx$ and $\ccal \otimes_R R[x^{-1}]$ admit a compact generator. Then $\ccal$ admits a compact generator. 
\end{lemma}
\begin{proof}
We have a short exact sequence of stable categories
\[
0 \rightarrow \ccal_{x\text{-nil}} \rightarrow \ccal \rightarrow \ccal \otimes_R R[x^{-1}] \rightarrow 0
\]
where the functors admit colimit preserving right adjoints. The third category is assumed to be compactly generated, while the first category is compactly generated by lemma \ref{lemma compact generator on completion}. By \cite{Efimov} proposition 3.3 we have that $\ccal$ is compactly generated. The fact that $\ccal$ admits a compact generator now follows from \cite{AGBrauer} lemma 3.9.
\end{proof}

\begin{proof}[Proof of theorem \ref{theorem stable over R}]
By proposition \ref{proposition dualizable is presentable} we see that $\Ccal$ and $\Ccal^\vee$ are presentable. Our goal is to show that $\ccal$ admits a compact generator. By lemma \ref{lemma postnikov compact generator} we may reduce to the case when $R$ is $0$-truncated.

 Assume for the sake of contradiction that $\ccal$ does not admit a compact generator. Consider the poset $P$ consisting of those ideals $I$ of $R$ with the property that $\ccal \otimes_R R/I$ does not admit a compact generator (where we order the ideals by inclusion). It follows from lemma \ref{lemma deal with filtered colimits stable} that $P$ is closed under filtered colimits inside the poset of all ideals of $R$. Since $P$ is nonempty (as it contains $0$) we deduce that $P$ has a maximal element $I$. Replacing $R$ by $R/I$ we may reduce to the case when $I = 0$.  In other words, we may assume that  $\ccal \otimes_R R/J$ admits a compact generator for every nonzero ideal $J$. 

We claim that $R$ is reduced. Let $x$ be an element of $R$ such that $x^2 = 0$. An application of lemma \ref{lemma handle extensions} shows that $\ccal \otimes_{R} R/{Rx}$ does not admit a compact object. It follows that $Rx = 0$, so that $x = 0$ and $R$ is reduced, as claimed.

If $R$ is the zero ring the desired assertion is clear, so suppose now that $R$ is nonzero. We claim that $R$ is an integral domain. Let $x$ be a nonzero element of $R$, and suppose given another element $y$ such that $xy = 0$. By lemma \ref{lemma handle extensions} we see that $\ccal \otimes_R R[x^{-1}]$ does not admit a compact generator. Since the map $R \rightarrow R[x^{-1}]$ factors through $R/Ry$ we deduce that $\ccal \otimes_R R/Ry$ does not admit a compact generator. Hence $y = 0$, so that $R$ is an integral domain, as desired.

Let $F$ be the fraction field of $R$. By theorem \ref{theorem stable} we have that $\ccal \otimes_R F$ admits a compact generator. Applying lemma \ref{lemma deal with filtered colimits stable} we deduce the existence of a nonzero element $x$ of $R$ such that $\ccal \otimes_R R[x^{-1}]$ admits a compact generator. An application of lemma \ref{lemma handle extensions} shows that $\ccal$ admits a compact generator, which is a contradiction.
\end{proof}


\ifx\inmain\undefined
\bibliographystyle{myamsalpha2}
\bibliography{References}
\fi


\bibliographystyle{myamsalpha2}
\bibliography{References}

\end{document}